%% file: trop_hyp_lag.tex
\documentclass[12pt]{amsart}

 \usepackage{graphicx}
 \usepackage{tikz-cd}
 \usepackage{ifthen}    
 \usepackage{amssymb}   
 \usepackage{amstext}
 \usepackage{amsfonts}
 \usepackage{amsmath}   
 \usepackage{latexsym}  
 \usepackage{color}
 \usepackage{epsfig}
 \usepackage{pstricks, pst-node,pst-text,pst-tree}
 \usepackage[all]{xy}
 \usepackage{array}
 \usepackage{amsthm}
 \usepackage{eucal, mathrsfs} 
 \usepackage{hyperref}

\newcommand{\numberset}[1]{\ensuremath{\mathbb{#1}}}    
\newcommand{\C}{\numberset{C}}  
\newcommand{\R}{\numberset{R}}  
\newcommand{\Z}{\numberset{Z}}  
\newcommand{\PP}{\numberset{P}}  
\newcommand{\RP}{\numberset{R}\numberset{P}}

\newcommand{\inn}[2]{ \langle {#1}, {#2} \rangle}

\newcommand{\h}{\mathbf{h}}
\newcommand{\y}{\mathbf{y}}
\newcommand{\x}{\mathbf{x}}
\newcommand{\gb}{\mathbf{g}}

\theoremstyle{definition}

\newtheorem{thm}{Theorem}[section]
\newtheorem{prop}[thm]{Proposition}
\newtheorem{lem}[thm]{Lemma}
\newtheorem{cor}[thm]{Corollary}
\newtheorem{rem}[thm]{Remark}
\newtheorem{ex}[thm]{Example}
\newtheorem{defi}[thm]{Definition}

\newtheorem{propdefi}[thm]{Proposition-Definition}

\DeclareMathOperator{\conv}{Conv} 
 
\DeclareMathOperator{\cone}{Cone} 

 \DeclareMathOperator{\Gl}{GL}
\DeclareMathOperator{\spn}{span} 
\DeclareMathOperator{\Hom}{Hom}

\DeclareMathOperator{\Log}{Log}

\DeclareMathOperator{\inter}{Int}

\newboolean{mynotes}\setboolean{mynotes}{false}

\newcommand{\mycomments}[1]{
           \ifthenelse{\boolean{mynotes}}
                      {#1}{}
           }
	   
\newcommand{\marginlabel}[1]
{\mbox{}\marginpar[\raggedleft $\longrightarrow$ \\ \tiny\sf #1]{\raggedright $\longleftarrow$\\ \tiny\sf #1}}

\newcommand{\todo}[1]{\mycomments{\marginlabel{\tiny\sf{#1}}}}

\begin{document}

\title{Lagrangian submanifolds from tropical hypersurfaces}
\author{Diego Matessi}

\begin{abstract}
We prove that a smooth tropical hypersurface in $\R^3$ can be lifted to a smooth embedded Lagrangian submanifold in $(\C^*)^3$. This completes the proof of the result announced in the article ``Lagrangian pairs pants'' \cite{lag_pants}. The idea of the proof is to use Lagrangian pairs of pants as the main building blocks.
\end{abstract}

\maketitle


\section{Introduction}
\subsection{Main result} In the article \cite{lag_pants} we introduced a new Lagrangian submanifold of $(\C^*)^n$, which we called a Lagrangian pair of pants. It is a fundamental object in the proof of the following result, announced in the same article
\begin{thm} \label{main_thm} Given a smooth tropical hypersurface $\Xi$ in $\R^2$ or $\R^3$, there is a one parameter family of smooth Lagrangian submanifolds $\mathcal L_t$ of respectively $(\C^*)^2$ or $(\C^*)^3$ such that $\mathcal L_t$ is homeomorphic to the PL lift $\hat \Xi$ of $\Xi$ and converges to it in the Hausdorff topology as $t \rightarrow 0$. 
\end{thm}

In the present article we complete the proof of this theorem by proving the case of hypersurfaces in $\R^3$.  The case of curves in $\R^2$, which is relatively simple, is already contained in op.cit. In the same article, we also gave examples of Lagrangian lifts of non-smooth tropical curves and sketched a proof of how to lift tropical curves in higher codimension to Lagrangian submanifolds in $(\C^*)^n$.  The piecewise linear (PL) Lagrangian lift $\hat \Xi$ of a tropical hypersurface $\Xi$ in $\R^n$ is a topological, closed $n$-dimensional submanifold of $(\C^*)^n$, Lagrangian on the smooth points, such that the following diagram commutes
\begin{equation} \label{diagr_lift}
    \begin{tikzcd}
         \hat \Xi  \arrow[d]  \arrow[r] & (\C^*)^n  \arrow[d]\\
         \Xi \arrow[r] & \R^n
    \end{tikzcd}
\end{equation}
where the vertical arrows are given by the $\Log$ map
\[ \Log: (z_1, \ldots, z_n) \mapsto (\log|z_1|, \ldots, \log|z_n|). \] 
This is similar to other piecewise linear objects associated to tropical subvarieties in the context of complex geometry, e.g. the complexified non-archimedean amoeba in \cite{mikh_pants} also called phase tropical hypersurfaces (see for instance \cite{kerr_zha_phase_trop} or \cite{nisse_sottile_nonAch_coam}). 

In \cite{lag_pants} we also gave many new constructions of Lagrangian surfaces in two dimensional toric varieties, including some monotone Lagrangian tori. Shortly after our article appeared on the arXiv, Mikhalkin released \cite{mikh_trop_to_lag} with a different proof of the result for the case of tropical curves in $\R^n$. He also gave many other interesting examples, including a proof of Givental's result on the existence of embedded Lagrangian non-orientable surfaces diffeomorphic to the sum of $2k+1$ Klein bottles, with $k \geq 1$. In addition, in the case of lifts of tropical curves in three dimensional toric varieties, he gives an interpretation of the order of the first homology group of the lift in terms of the multiplicity of the tropical curve. This should have interesting applications in the counting of special Lagrangian submanifolds and homological mirror symmetry. We also learned about the work of Mak and Ruddat who give a construction of Lagrangian submanifolds in the mirror quintic, lifting tropical curves in the boundary of the moment polytope of the ambient toric variety. They also give similar applications to the counting problem of special Lagrangian submanifolds. 

We point out that in the case of tropical curves in $\R^2$ an alternative method of proof of Theorem \ref{main_thm} is to use the hyperk\"ahler trick in $(\C^*)^2$, turning complex submanifolds to Lagrangian. This reduces the Lagrangian problem to the complex case, where one can appeal to ``tropical to complex'' correspondence results. This method was used for instance by Mikhalkin in \cite{mikh_trop_to_lag}. The same idea does not apply in the case of tropical hypersurfaces in $\R^3$. This is where our idea of introducing Lagrangian pairs of pants as main building blocks becomes essential. 

In Sections 2-4 we recall the main ideas of \cite{lag_pants}, such as the definition of Lagrangian pair of pants, of the PL lift $\hat \Xi$ and we summarize the most useful properties. Section 5 contains some technical results in preparation for the proof of Theorem \ref{main_thm} given in Section 6. 
\subsection{Examples in toric varieties and Calabi-Yau manifolds?} In the last section we discuss some (possible) generalizations and examples, extending those described in \cite{lag_pants} and \cite{mikh_trop_to_lag} in the case of tropical curves. In particular we discuss how the same tropical hypersurface can be lifted in different ways, by twisting with local sections. Moreover we give some examples of lifts of non smooth tropical hypersurfaces. Finally we discuss the problem of constructing examples in three dimensional toric varieties. Unfortunately the step from $(\C^*)^3$ to toric varieties is not as straight forward as in the case of tropical curves. A complete construction requires a more detailed analysis of the interaction of the Lagrangian lifts with the toric boundary, which we postpone to future work. Nevertheless we give some interesting candidate examples: a candidate Lagrangian monotone embedding of $S^1 \times S^2$ in $\PP^3$ (see Example \ref{monotone_ex}), which generalizes the examples in \cite{lag_pants} and a candidate non-orientable Lagrangian in $\C^3$ which generalizes Mikhalkin's construction of non-orientable Lagrangian surfaces in $\C^2$. In \S \ref{CY's} we sketch how one could use the ideas in this article to construct interesting Lagrangian submanifolds in the symplectic Calabi-Yau manifolds with singular Lagrangian torus fibrations which come from our work with Casta\~no-Bernard \cite{CB-M} and the work of Gross \cite{TMS}, \cite{GHJ}. In particular in Example \ref{sphere_quintic} we give a candidate construction of $105$ Lagrangian submanifolds (spheres?) in a symplectic manifold homeomorphic (and conjecturally symplectomorphic) to the quintic threefold in $\PP^4$.

\subsection{Notation.} \label{notation} Given a set of vectors $u_1, \ldots, u_k$ in a vectors space $V$, the cone generated by these vectors is the set
\[ \cone \{ u_1, \ldots, u_k \} = \left \{ \sum_{j=1}^{k} t_j u_j \, | \, t_j \in \R_{\geq 0} \right \}. \]
Given a subset $A$ of an affine space, we will denote the convex hull of $A$ by
\[ \conv A. \]
Given a subset $W$ of an affine space, the notation 
\[ \inter W \]
stands for the relative interior of $W$. Namely, we consider the smallest affine subspace containing $W$, then $\inter W$ will be the topological interior relative to this affine subspace. This for examples applies to faces of polyhedra or cones. 

\subsection*{Acknowledgments} 
I wish to thank Ricardo Casta\~no-Bernard, Mark Gross and Grigory Mikhalkin for useful discussions on this topic. For this project I was partially supported by the grant FIRB 2012 ``Moduli spaces and their applications'' and by the national research project ``Geometria delle variet\`a proiettive'' PRIN 2010-11. I am a member of the INdAM group GNSAGA.

\section{Lagrangian PL lifts of tropical hypersurfaces}
This section reports notions and results from \cite{lag_pants}.
\subsection{The set-up} \label{setup} Let $M \cong \Z^{n+1}$ be a lattice of rank $n+1$ and let $N = \Hom(M, \Z)$ be its dual lattice. We define $M_{\R} := M \otimes_{\Z} \R$ and similarly $N_{\R}$. Since $M_{\R}$ is the dual of $N_{\R}$ the space $M_{\R} \oplus N_{\R}$ has a natural symplectic form
\[ \omega( m \oplus n, m' \oplus n') = \inn{m}{n'} - \inn{m'}{n} \]
where $\inn{\cdot}{\cdot}$ is the duality pairing. We will consider the $n+1$-dimensional torus
\[ T = N_{\R} / N \]
whose cotangent bundle is 
\[ T^*T = M_{\R} \times N_{\R}/N. \]
Then $\omega$ is the standard symplectic form on $T^*T$. The projection 
\[f: T^{*}T \rightarrow M_{\R} \]
 is a Lagrangian torus fibration. 

We will often identify $M_{\R}$ with $\R^{n+1}$ by choosing a basis $\{u_1, \ldots, u_{n+1} \}$ of $M$ and denote the corresponding coordinates in $M_{\R}$ by $x = (x_1, \ldots, x_{n+1})$. Similarly we also identify $N_{\R}$ with $\R^{n+1}$ by choosing a basis $\{ u_1^*, \ldots, u_{n+1}^* \}$ of $N_{\R}$ such that 
\begin{equation} \label{dual_basis}
   \inn{u^*_{j}}{u_k} = \frac{1}{\pi} \delta_{jk} 
\end{equation}
and denote the corresponding coordinates by $y = (y_1, \ldots, y_{n+1})$. In particular $N$ is identified with $\pi \Z^{n+1}$ and thus 
\begin{equation} \label{toro}
T = \R^{n+1}/ \pi \Z^{n+1}.
\end{equation}
We denote by $[y]$ the element of $T$ represented by $y$. 
The symplectic form $\omega$ becomes 
\begin{equation} \label{st_sympl}
       \omega = \frac{1}{\pi}\sum_{i=1}^{n+1} dx_i \wedge dy_i.
\end{equation}
We also have that $T^*T$ is symplectomorphic to $(\C^*)^n$ with the symplectic form 
\[      \omega = \frac{i}{4\pi}\sum_{k=1}^{n+1} \frac {dz_k \wedge d \bar z_k}{|z_k|^2}, \]
via the symplectomorphism $(x_k,y_k) \mapsto z_k = e^{x_k+i2y_k}$.



\subsection{Tropical hypersurfaces} \label{trop_hyper}  A subset $P \subset N_{\R}$ is a convex lattice polytope if it is the convex hull of a finite set of points in $N$. A subdivision of $P$ in smaller lattice polytopes $P_1, \ldots, P_k$ is called {\it regular} if there exists a convex piecewise affine function $\nu: P \rightarrow \R$, such that $\nu$ is integral, i.e. $\nu(P \cap N) \subset \Z$, and the $P_i$'s coincide with the domains of affiness of $\nu$. With a slight abuse of notation the pair $(P, \nu)$ will also denote the set of simplices in the decomposition, i.e. all the $P_k$'s and all of their faces. Therefore we will write 
$e \in (P, \nu)$ to indicate that $e$ is a simplex in the decomposition. Inclusion of faces will be denoted by 
\[ f \preceq e. \]
We say that the subdivision is {\it unimodal} if all the $P_i$'s are elementary simplices.  

The discrete Legendre transform of $\nu$ is the function $\check{\nu}: M_{\R} \rightarrow \R$:
\begin{equation} \label{cknu}
\check{\nu}(m) = \min \{ \inn{v}{m} + \nu(v), \ v \in P \cap N\}. 
\end{equation}

Also $\check \nu$ gives a decomposition of $M_{\R}$ in the convex polyhedra given by its domains of affiness. As above, the pair  $(M_{\R}, \check \nu)$ will also denote the set of all polyhedra in the subdivision and their faces. 

\begin{defi} The {\it tropical hypersurface} associated to the pair $(P, \nu)$ is the subset $\Xi \subset M_{\R}$ given by the points where $\check{\nu}$ fails to be smooth. We say that $\Xi$ is {\it smooth} if the subdivision of $P$ induced by $\nu$ is unimodal.
\end{defi}

\begin{ex} \label{ex_strop_hyp} Let $M= \Z^{n+1}$ and let $P$ be the standard simplex (i.e. the convex hull of the origin and the canonical basis of $\Z^{n+1}$). Let $\nu$ be the zero function. Then 
\[ \check{\nu} = \min \{0, x_1, \ldots, x_{n+1} \}. \] 
The corresponding tropical hypersurface is called the {\it standard tropical hyperplane} which we denote by $\Gamma$. 
\end{ex}

The subdivision $(M_{\R}, \check \nu)$ is dual to the subdivision $(P, \nu)$. In particular there is an inclusion reversing bijection between faces of $(P, \nu)$ and faces of $(M_{\R}, \check \nu)$, which we denote by 
\[ e \mapsto \check e. \]
We have that $\dim \check e = n+1- \dim e$.

\subsection{The tropical hyperplane} \label{stTrop} Let $\{u_1, \ldots, u_{n+1} \}$ be a basis of $M$ inducing coordinates $x=(x_1, \ldots, x_{n+1})$ on $M_{\R}$. Let $\Gamma$ be the standard tropical hyperplane (see Example \ref{ex_strop_hyp}). It can be described as the union of the following cones. Let 
\begin{equation} \label{u0}
     u_0 = - \sum_{j=1}^{n+1} u_j.
\end{equation}
Given a proper subset $J \subsetneq \{0, \ldots, n+1 \}$ let $|J|$ be its cardinality and let 
\[ \Gamma_{J} = \cone \{ u_j, \, j \in J \}. \] 
For convenience let us also define
\[ \Gamma_{\emptyset} = \{ 0 \},\]
which is the vertex of $\Gamma$. We have that
\[ \Gamma = \bigcup_{0 \leq |J| \leq n} \Gamma_J. \]


\subsection{Lagrangian coamoebas} \label{stCoam}
Let $\{ u_1^*, \ldots, u_{n+1}^* \}$ be the basis of $N_{\R}$ satisfying \eqref{dual_basis}. Thus the torus $T$ is as in \eqref{toro}. 
Consider the points
\[ p_0 = 0 \quad \text{and} \quad p_k = \frac{\pi}{2}u^*_k, \  \ (k=1, \ldots, n+1). \]  
Denote by $C^+$ the set of points $[y] \in T$ which are represented either by a vertex or by an interior point of the $(n+1)$-dimensional simplex with vertices the points $p_0, \ldots, p_{n+1}$. Let $C^-$ be the image of $C^+$ with respect to the involution $[y] \mapsto [-y]$. 
The {\it (standard) $(n+1)$-dimensional Lagrangian coamoeba} is the set $C = C^+ \cup C^-$ (see Figure \ref{co_am}). Notice that this definition makes sense also when $n=0$, in which case $C = \R/ \pi \Z$. The points $[p_0], \ldots, [p_{n+1}]$ are called the {\it vertices} of $C$. 

For any subset $J  \subsetneq \{0, \ldots, n+1\}$, denote by $E^+_J$ the set of points $[y] \in T$ which are represented either by a vertex or by a point in the relative interior of the $(n+1-|J|)$-dimensional simplex with vertices the points  $\{p_k \}_{k \notin J}$. We let $E_J^-$ be the image of $E^+_J$ via the involution $[y] \mapsto [-y]$. 
We define the {\it $J$-th face} of $C$ to be the set $E_J = E_J^+ \cup E_J^-$.   Clearly $E_J$ is homeomorphic to an $(n+1-|J|)$-dimensional Lagrangian coamoeba.  If $J = \{ j \}$ then we denote $E_J$ by $E_j$ and we call it the {\it $j$-th facet} of $C$. The closures $\bar E_J$ of the $J$-th faces satisfy
\[ \begin{split}
              \bar E_j   & = \left \{ [y] \in \bar C \, | \, \inn{u_j}{y} = 0 \mod  \Z \right \}, \quad \text{when} \ j=1, \ldots, n+1 \\
              \bar E_0  & = \left \{ [y] \in \bar C \, | \, \inn{u_0}{y} = \frac{\pi}{2} \mod \Z \right \}, \\
              \bar E_{J} & = \bigcap_{j \in J} \bar E_j.
   \end{split}
                     \] 
where $u_0$ is the vector defined in \eqref{u0}. For convenience we also define 
\[ E_{\emptyset} = C. \]
Faces of dimension $1$ are called edges. If we denote by $J_k$ the complement of $k$ in $\{0, \ldots, n+1 \}$, then 
\[ p_k = E_{J_k}. \]

\begin{figure}[!ht] 
\begin{center}
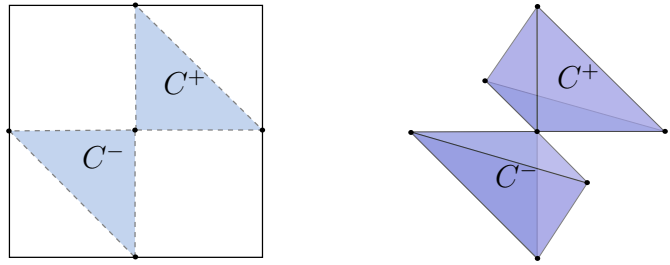
\caption{The $2$ and $3$ dimensional standard coamoebas. They contain their vertices but not their higher dimensional faces.} \label{co_am}
\end{center}
\end{figure}

\subsection{The Lagrangian PL-lift of $\Gamma$} \label{plift} For every $J \subset \{ 0, \ldots, n+1 \}$ with $0 \leq |J| \leq n$, consider the following $n+1$ dimensional subsets of $M_{\R} \times T$:
\begin{equation} \label{ends_lift}
         \hat{\Gamma}_{J} = \Gamma_J \times E_{J}.
\end{equation}
The piecewise linear lift (or PL-lift) of $\Gamma$ is defined to be 
\begin{equation} \label{plLift}
  \hat \Gamma = \bigcup_{0 \leq |J| \leq n} \hat{\Gamma}_J.
\end{equation}
We have that $\hat \Gamma$ is a topological manifold and its smooth part is Lagrangian.

\subsection{Symmetries of $\Gamma$ and $C$}
\label{symmetries}
For every $k=1, \ldots, n+1$ let $R_k$ be the unique affine automorphism of $T$ which maps $C^+$ to itself, exchanges the vertices $p_0$ and $p_k$ and fixes all other vertices. Define $G$ to be the group generated by the maps $R_k$. We have that $G$ acts on the coamoeba $C$. The elements $R_k$ permute the faces of $C$ according to the following rule. Let $R_k$ act on the set of indices $\{0, \ldots, n+1 \}$ as the transposition which exchanges $0$ and $k$ and extend this action to the set of subsets $J \subseteq \{0, \ldots, n+1 \}$ . Then clearly
\[ R_k E_J = E_{R_kJ}. \]

Dually let us define the group acting on $\Gamma$. If $u_0, \ldots, u_{n+1}$ are the vectors in $M_{\R}$ as in \S \ref{stTrop}, let $R^*_k$ be the unique linear map which exchanges $u_0$ and $u_k$ and fixes all other $u_j$'s.   More explicitly
\[ R^*_k(x) = (x_1 - x_k, \, \ldots, \, x_{k-1} - x_k, \, -x_k, \, x_{k+1}-x_k, \, \ldots, \, x_{n+1}-x_k ).\]
We have that $R^*_k$ permutes  the cones $\Gamma_J$ according to the rule
\begin{equation} \label{action_cones}
      R^*_k \Gamma_J = \Gamma_{R_k J}.
\end{equation}
Denote by $G^*$ the group generated by the transformations $R^*_k$. Then $G^*$ acts on $\Gamma$. It is easy to see that 
\[ R_k: y \longmapsto (R^*_k)^t(y) + p_k \]
where $(R^*_k)^t$ is the transpose of $R^*_k$. Therefore we can combine the actions of $G$ and $G^*$ to get an action on the PL-pair of pants $\hat \Gamma$ via the following affine symplectic automorphisms of $T^*T$:
\begin{equation} \label{glob_symm}
 \mathcal R_k(x,y) = (R^*_k x, R_k y).
\end{equation}
Let $\mathcal G$ be the group generated by the $\mathcal R_k$'s. Then $\mathcal G$ acts on $\hat \Gamma$.

\subsection{Lagrangian piecewise linear lifts of tropical hypersurfaces} \label{LagPLift} Let $\Xi$ be a smooth tropical hypersurface in $M_{\R}$ given by a pair $(P, \nu)$ as in \S \ref{trop_hyper}. Given a $k$-dimensional face $e \in (P, \nu)$, with $k=1, \ldots, n+1$, let $\check e$ be the dual $(n+1)-k$ dimensional face of $\Xi$. We will use the involution $\iota$ of $M_{\R} \times N_{\R} / N$ given by $\iota: (x, [y]) \mapsto (x, [-y])$.
Define the following subsets of $N_{\R}/N$:
\[ \begin{split}
            & \bar{C}^{+}_{e}  = \{ [y] \in N_{\R} / N \, | \, 2(y-k) \in e \ \text{for some} \ k \in N \},\\
            & \bar{C}^{-}_{e} = \iota \left(  \bar{C}^{+}_{e} \right), \\
            & \bar{C}_{e}  = \bar{C}^{+}_{e} \cup \bar{C}^{-}_{e}.
   \end{split} \]
A point $[y] \in N_{\R} / N$ is a {\it vertex} of $\bar{C}_e$ if $2(y-k)$ is a vertex of $e$ for some $k \in N$. We define $C^+_e$ (resp. $C^{-}_e$ and $C_{e}$) to be the set of points $[y]$ which are either vertices or relative interior points of $\bar{C}^+_e$ (resp. $\bar{C}^{-}_e$ and $\bar{C}_{e}$). Clearly if $f \preceq e$ is a face of $e$, then $C_f$ is a face of $C_e$. When we view $C_f$ as a face of $C_e$ we denote it $C_{e,f}$.

Now define the Lagrangian lift of $\check e$ to be 
\[ \hat e = \check e \times C_{e}. \]
Clearly $\hat e$ is $(n+1)$-dimensional, moreover the tangent space of $C_{e}$ is the orthogonal complement of the tangent space to $\check e$, thus the interior of $\hat e$ is a Lagrangian submanifold of $M_{\R} \times N_{\R} / N$. We define the Lagrangian $PL$-lift of $\Xi$ to be 
\[ \hat \Xi = \bigcup_{e} \, \hat e \]
where the union is over all faces in $(P, \nu)$ of dimensions $k=1, \ldots, n+1$. It can be shown that $\hat \Xi$ is an $(n+1)$-dimensional topological submanifold of $M_{\R} \times N_{\R} / N$.

\begin{ex} Let $P = \conv \{ (0,0), (1,2), (2,1) \}$ with the unimodal subdivision induced by the piecewise affine function $\nu$ such that $\nu(0,0) = 1$ and $\nu(2,1) = \nu(1,2) = \nu(1,1) = 0$. Let $\Xi$ be the associated tropical curve. If $\check e$ is an edge of $\Xi$, then $\hat e$ is a cylinder. When $e$ is a vertex,  the sets $C_e$ are drawn in Figure \ref{pl_lift}.
\begin{figure}[!ht] 
\begin{center}
\includegraphics{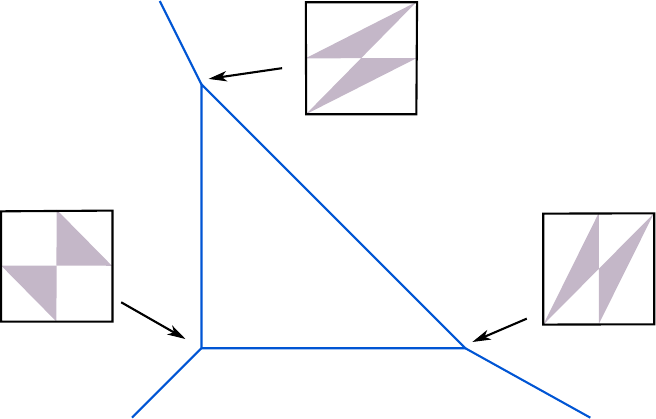}
\caption{} \label{pl_lift}
\end{center}
\end{figure}
\end{ex}
 
Given a $k$-dimensional polyhedron $\check e$ of $\Xi$, define the star-neighborhood of $\check e$ to be the union of the polyhedra of $\Xi$ which contain $\check e$, i.e.
\begin{equation} \label{star_neigh}
    \Xi_{\check e} = \bigcup_{f \preceq e, \ \dim f \geq 1} \ \check f.
\end{equation}
Similarly define its lift
 \[ \hat \Xi_{\check e} = \bigcup_{f \preceq e \ \dim f \geq 1 } \hat f. \]

\section{Lagrangian pairs of pants}
We report the definition of a Lagrangian pair of pants from \cite{lag_pants} and recall the main properties, without proofs.
\subsection{The definition} \label{theconstr}
 \begin{defi} \label{blup} The {\it real blow up of the coamoeba} $C$ at its vertices is the smooth manifold $\tilde C$ defined as follows. If $p$ is one of the vertices $p_0, \ldots, p_{n+1}$ of $C$ and $U_p \subset C$ a small neighborhood of $p$ let 
 \[ \tilde U_p = \{ (q, \ell) \in U_p \times \RP^{n} \, | \, q-p \in \ell \}, \]
 where $\ell$ is a line through the origin in $\R^{n+1}$, which we think as a point of $\RP^{n+1}$. The real blow up of $C$ at $p$ is formed by gluing $\tilde U_p$ to $C-p$ via the projection map $\tilde U_p \rightarrow U_p$. We define $\tilde C$ to be the blow up of $C$ at all vertices. We denote by $\pi: \tilde C \rightarrow C$ the natural projection. Clearly we can identify $\tilde C - \cup_{j=0}^{n+1} \pi^{-1}(p_j)$ with $C - \{p_0, \ldots, p_{n+1} \}$ via $\pi$.   Notice that when $E_J$ is a face of $C$ then it is also a coamoeba inside a smaller dimensional torus, therefore we can define its blow-up which we denote by $\tilde E_J$. Let $G$ be the group acting on $C$ defined in \S \ref{symmetries}. It is easy to see that this action lifts to an action on $\tilde C$.
 \end{defi}

Define the following function $F$ on  $C$: 
\begin{equation} \label{Fglob} 
   F(y)=  \begin{cases} 
                                                \left( \cos \left( \sum_{j=1}^{n+1}y_j \right) \prod_{j=1}^{n+1} \sin y_j  \right)^{\frac{1}{n+1}} \quad \text{on} \ C^+, \\
                                                \ \\
                                               (-1)^{n} \left( \cos \left( \sum_{j=1}^{n+1}y_j \right) \prod_{j=1}^{n+1} \sin y_j  \right)^{\frac{1}{n+1}} \quad \text{on} \ C^-.
                                          \end{cases}
\end{equation}
We have that $F$ is well defined on $C$ and vanishes on the boundary of $C$.  Moreover if $\iota: [y] \mapsto [-y]$ is the involution of the torus then we have that $F$ is $G$ invariant and satisfies $F(\iota(y)) = -F(y)$.
The graph of $dF$ over $C - \{p_0, \ldots, p_{n+1} \}$ inside $T^*T$, i.e. the graph of the map 
\[ \h = (F_{y_1}, \ldots, F_{y_{n+1}}), \]
where $F_{y_j}$ denotes the partial derivative of $F$ with respect to $y_j$, is a Lagrangian submanifold. We have the following

\begin{lem} \label{smoothExt} Let $F: C \rightarrow \R$ be as in \eqref{Fglob}. Then $F$ and the map $\h: C - \{p_0, \ldots, p_{n+1} \} \rightarrow M_{\R}$ extend smoothly to $\tilde C$ and the map $\Phi: \tilde C \rightarrow T^*T$ given by 
\begin{equation} \label{phi:graph}
  \Phi(q) = (\mathbf{h}(q), \pi(q) ).
\end{equation}
is a Lagrangian embedding of $\tilde C$.
\end{lem}

Whenever $F$ is a function on $C$ satisfying the above lemma, we say that the map $\Phi$ is the {\it graph of an exact one form over $\tilde C$}. 
 
\begin{defi} \label{LpPants}  We call the submanifold $L=\Phi(\tilde C)$ the {\it standard Lagrangian pair of pants} in $T^*T$. Given $\lambda >0$, let $\Phi_{\lambda}$ be the embedding constructed from $\h_{\lambda} = (\lambda F_{y_1}, \ldots, \lambda F_{y_{n+1}})$ via \eqref{phi:graph}. Then,  if $\lambda \neq 1$, we call $\Phi_{\lambda}(\tilde C)$ a {\it $\lambda$-rescaled Lagrangian pair of pants}
\end{defi}

We have that $L$ has the following symmetries
\begin{lem} \label{eq_action} Given a transformation $R_k$ as in \S \ref{symmetries} and the involution $\iota$ of the torus, the map $\h: \tilde C \rightarrow M_{\R}$ defined using $F$ satisfies
\[ \h(R_k(y)) = R^*_k \h(y) \quad \text{and} \quad \h(\iota(y)) = \h(y) \]
In particular the group $\mathcal G$ and the involution act on a Lagrangian pair of pants.
\end{lem}

\subsection{Some properties} \label{Imh} 

For every $k \in \{0, \ldots, n+1 \}$ define 
\[\mathcal H_{k} = \left \{ t_k u_0 + \sum_{l \neq k} t_l u_l \, | \, t_l \geq 0 \ \forall l \ \text{and} \ t_1t_2 \ldots t_{n+1} \leq  \frac{1}{(n+1)^{n+1}} \right\}. \]
Recall that we defined $J_{k}$ to be the complement of $k$ in $\{0, \ldots, n+1 \}$ (see \S \ref{stCoam}). Then 
\[ \mathcal H_{k} \subset \Gamma_{J_k} \quad \text{and}  \quad \mathcal H_{k} = R^*_{k} \mathcal H_{0}. \]

Let 
\begin{equation} \label{the amoeba}
        \mathcal H = \bigcup_{l=0}^{n+1} \mathcal H_l,
\end{equation}
see Figure \ref{lagrangian_amoeba}. Let $\mathcal S_0$ be the hypersurface
\begin{equation} \label{bound_amoeba}
 \mathcal S_0: \quad  (n+1)^{n+1} x_1 \ldots x_{n+1} = 1 \ \text{and} \  x_j > 0, \forall j.
\end{equation}
and let 
\begin{equation} \label{bound_amoeba_k}
   \mathcal S_k = R^*_{k} \mathcal S_{0}.
\end{equation}
Then the boundary of $\mathcal H$ is 
\[ \partial \mathcal H = \bigcup_{l=0}^{n+1} \mathcal S_l. \]

\begin{figure}[!ht] 
\begin{center}
\includegraphics{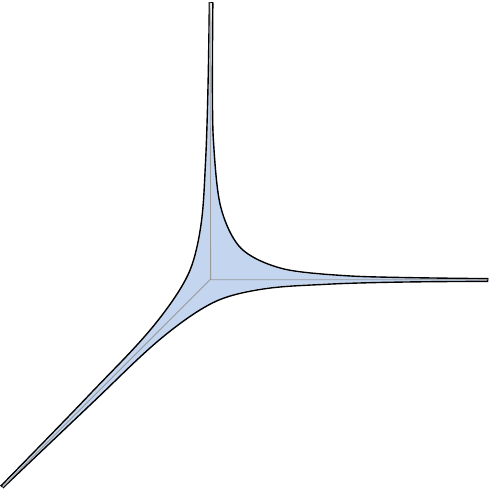}
\caption{The set $\mathcal H$ in the case $n=1$} \label{lagrangian_amoeba}
\end{center}
\end{figure}

\begin{prop} \label{imh} Assume $n=1$ or $2$. The image of $\h: \tilde C \rightarrow M_{\R}$ is $\mathcal H$ and the hypersurface $\mathcal S_k$ is the image of the set $\pi^{-1}(p_k)$. Moreover $\h$ defines a diffeomorphism between $\inter C^+$ and $\inter \mathcal H$. 
\end{prop}

This statement must be true for all values of $n$, but unfortunately we have been able to write a complete proof only in these dimensions. 

\begin{cor} \label{HessFneg} Assuming $n=1$ or $2$. Let $F$ be as in \eqref{Fglob}, then the Hessian of $F$, restricted to $\inter C^+$, is negative definite.
\end{cor}

We expect also this to be true for all $n$. Let us give a more detailed description of the map $\h$.


\begin{defi} For every pair of vertices $p_k$ and $p_j$ of $C^+$, let $\delta_{jk}$ be the hyperplane that contains all vertices different from $p_k$ and $p_j$ and passes through the middle point of the edge from $p_k$ to $p_j$. This hyperplane cuts $C^+$ in two halves. We denote by $\Delta_{jk}$ the half which contains $p_k$.
\end{defi}

Clearly, the set of hyperplanes $\delta_{jk}$ cuts $C^+$ into the first barycentric subdivision of $C^+$. 
We have the following inequalities defining $\Delta_{jk}$
\begin{equation} \label{inqDelta1} 
     \Delta_{j0} =  \left \{ y \in C^+ \, | \, 2y_j + \sum_{k\neq j} y_k \leq \frac{\pi}{2} \right \} 
\end{equation}
and when $j,k \neq 0$
\begin{equation} \label{inqDelta2} 
 \Delta_{jk} = \left \{ y \in C^+ \, | \,  y_k-y_j \geq 0 \right \}.
\end{equation}
For every face $E_J^+$ of $C^+$ let $\mathcal W_{J}^+$ denote its star neighborhood, i.e. the union of simplices of the barycentric subdivision whose closures contain the barycenter of $E_J^+$. We have that 
\begin{equation} \label{facenbh}
       \mathcal W_{J}^+ = \bigcap_{k \notin J, j \in J} \Delta_{jk}.
\end{equation}
As usual we denote by $\mathcal W_{J}^-$ the image of $\mathcal W_{J}^+$ with respect to $\iota$ and 
\[ \mathcal W_{J} = \mathcal W_{J}^- \cup \mathcal W_{J}^+ \quad \text{and} \quad  \tilde{\mathcal W}_{J}= \pi^{-1}(\mathcal W_{J}). \]

We have a dual structure for $\mathcal H$. 

\begin{defi}
For every $j,k =0, \ldots, n+1$ with $j \neq k$ let 
\[ d_{jk} =  \spn_{\R} \{ u_l \, | \, l \neq j,k \} \]
It is a codimension $1$ vector subspace which divides $M_{\R}$ in two halves. Denote by $D_{jk}$ the half containing $u_j$. 
\end{defi}

We have the following inequalities defining $D_{jk}$
\[ D_{j0}=  \left \{ x \in M_{\R} \, | \, x_j \geq 0 \right \} \]
and when $j,k \neq 0$
\[ D_{jk} = \left \{ x \in M_{\R}  \, | \,  x_j-x_k \geq 0 \right \}. \]
Let 
\begin{equation} \label{h_nbhoods}
  \mathcal V_{J} = \bigcap_{j \in J, k \notin J} D_{jk}
\end{equation}
When $1 \leq |J| \leq n$, $\mathcal V_{J}$ contains the face $\Gamma_J$ of $\Gamma$ and can be regarded as a neighborhood of it, analogous to the star neighborhood $\mathcal W_{J}$ of the face $E_{J}$. Moreover 
\[ \mathcal V_{J_k} \cap \mathcal H = \mathcal H_k.\]

We have the following useful facts:

\begin{lem} \label{tnbhd} We have that $R^*_l( \mathcal V_{J}) = \mathcal V_{R_lJ}$ and $R_l(\mathcal W_{J}) = \mathcal W_{R_lJ}$.
\end{lem}

\begin{lem} \label{hNbhd}
\[ \h( \tilde{\mathcal W}_{J}) = \mathcal V_{J} \cap \mathcal H \]
\end{lem}

The following lemma describes the behavior of $\h$ near the boundary of $C^+$.
\begin{lem} \label{hnear_faces} 
Let $E_J$ be a face of $C$ of codimension $1 \leq |J| \leq n$ and let $\{ q_\ell \}$ be a sequence of points of $C$ which converges to a point in $\inter{E_J}$. Then we have the following behavior of $\h$. If $p_0$ is a vertex of $E_J$ (i.e. $0 \notin J$) then
\[ \begin{split}
          \lim h_{j}(q_\ell) &= + \infty \quad \forall j \in J,  \\
          \lim h_{k}(q_{\ell}) & = 0 \quad \forall k \notin J \cup \{ 0 \}.
   \end{split}
\]
If $p_0$ is not one of the vertices of $E_J$ (i.e. $0 \in J$), then for all $i \notin J$ we have 
\[ \begin{split}
          \lim h_{i}(q_\ell) &= - \infty,  \\
          \lim h_{j}(q_{\ell})- h_{i}(q_{\ell}) & = + \infty \quad \forall j \in J - \{ 0 \}, \\
          \lim h_{k}(q_{\ell})- h_{i}(q_{\ell}) & = 0 \quad \forall k \notin J. 
   \end{split}
\]
\end{lem}

\begin{cor} \label{h_near_bndry} If $\{ q_k \}$ is a sequence of points of $C$ which converges to a point on the boundary of $C$, then either $\{ \h(q_k) \}$ converges to a point on the boundary of $\mathcal H$ or $\lim_{k \rightarrow +\infty} || h(q_k)|| = + \infty$. 
\end{cor}

We also have

\begin{prop}  \label{pairpants_PLft} The Lagrangian pair of pants $\Phi(\tilde C)$ is homeomorphic to the $PL$-lift $\hat \Gamma$ of $\Gamma$. 
\end{prop}

\section{Projections to faces and Legendre transform} \label{projections} 
Projections onto faces were introduced in \cite{lag_pants}, where the most important result, Proposition \ref{faceproj}, was proved. The only new input is the definition of compatible system of projections.
 
\subsection{The projections}

\begin{defi} \label{projFace1} Given a face $E_J$ of $C$ of codimension $1 \leq|J| \leq n$, let $L \subseteq N_{\R}$ be a vector subspace of dimension $|J|$ which is transversal to $E_J$. Let $U_{J, L}$ be the set of points $y \in \inter C$ such that there exists a $y' \in \inter E_J$ such that $y-y' \in L$. If such a $y'$ exists, it is unique by transversality. Thus we can define the projection 
\[ 
\begin{split}
 \y_{J, L}: U_{J, L} & \rightarrow \inter E_{J} \\
                        y & \mapsto y'.
 \end{split}
 \]
Recall that $\{ p_k \}_{k \notin J}$ is the set of vertices of $E_J$. Define 
\[ \tilde U_{J,L} = \pi^{-1}(U_{J,L} \cup \{ p_k \}_{k \notin J})  \subseteq \tilde C\]
Then $\y_{J,L}$ extends to a map $\y_{J,L}: \tilde U_{J,L} \rightarrow \tilde E_{J}$. 
\end{defi}

Dually we give the following definition.  

\begin{defi} \label{projFace2} Let $\Gamma_J$ be a face of $\Gamma$. Recall that we denoted by $V_J$ the smallest subspace containing $\Gamma_J$. Let $L$ be as in Definition \ref{projFace1}. Define 
\[ L^{\perp} = \{ x \in M_{\R} \, | \, \inn{x}{y} = 0 \ \forall y \in L \} \]
Then $L^{\perp}$ has dimension $n+1-|J|$ and it is transversal to $V_J$. It thus defines the projection   $\x_{J,L}: M_{\R}  \rightarrow V_J$, dual to $\y_{J,L}$, whose fibres are parallel to $L^{\perp}$.  
\end{defi}

Given a face $E_J$ of $C$, let $T_J$ be the smallest subtorus of $T$ which contains $E_J$. By construction $V_J \times T_J$ is a Lagrangian submanifold of $M_{\R} \times T$.  Given $L$ and $L^{\perp}$ as in Definitions  \ref{projFace1} and  \ref{projFace2}, the space $(V_J \times T_J) \times (L^{\perp} \times L)$ is naturally a covering of $M_{\R} \times T$ and thus induces from the latter a symplectic form. We have the following

\begin{lem} \label{cotangent} The choice of a vector subspace $L$ as in Definition \ref{projFace1} induces a natural (linear) symplectomorphism between the cotangent bundle of $V_J \times T_J$ and $(V_J \times T_J) \times (L^{\perp} \times L)$.
\end{lem}
 
\begin{proof} This is just linear algebra. In fact $L^{\perp} \times L$ can be naturally identified with a cotangent fibre of $V_J \times T_J$ by sending the pair $(\ell', \ell) \in L^{\perp} \times L$ to the linear form $(v, w) \mapsto \inn{\ell}{v} - \inn{\ell'}{w}$, where $v$ is a tangent vector in $V_J$ and $w$ in $T_J$. The signs in this identification are chosen in order to match the symplectic forms. 
\end{proof}
 Given $L$, $U_{J,L}$ and $\tilde U_{J,L}$ as above,  define   $\h_{J,L}: \tilde U_{J,L} \rightarrow V_{J}$ to be the map  
\begin{equation*}  
                  \h_{J,L} = \x_{J,L} \circ \h  
\end{equation*}
and $\gb_{J,L}: \tilde U_{J,L} \rightarrow  V_{J} \times \tilde{E}_J$ to be 
\begin{equation} \label{mapg}
                  \gb_{J,L} = ( \y_{J,L}, \, \h_{J,L} ).
\end{equation}

\subsection{Projections and Legendre transform}

\begin{prop} \label{faceproj} Assume $n=1$ or $2$. The map $\gb_{J,L}: \tilde U_{J,L} \rightarrow \tilde{E_J} \times V_{J}$ is a diffeomorphism onto the open subset $Z_{J,L} = \gb_{J,L}(\tilde U_{J,L}) \subseteq  \tilde{E_J} \times V_{J}$. Moreover, via the identification of the cotangent bundle of $V_J \times T_J$ with (a covering of) $M_{\R} \times T$ given in Lemma \ref{cotangent}, $\Phi( \tilde U_{J,L})$ is the graph of an exact one form over  $Z_{J,L}$ obtained as the differential of a Legendre transform of $F$.
\end{prop}

\begin{cor} \label{submers} The map $\h_{J,L}: \tilde U_{J,L} \rightarrow  V_{J}$ is a submersion. The fibres of $\h_{J,L}$ can be identified with open subsets of $\tilde E_{J}$ via the map $\y_{J,L}$. 
\end{cor}

\subsection{Compatible systems of projections} 
\begin{defi} A {\it compatible system of projections} over $C$ is given by a choice of transversal subspaces $L_{J}$ as in Definition \ref{projFace1} for every face $E_J$ of $C$ with the property that if $E_{J_2} \subset E_{J_1}$ then $L_{J_1} \subset L_{J_2}$. In particular this implies that 
\begin{equation} \label{comp_proj}
  \y_{J_2, L_{J_2}}  \circ \y_{J_1, L_{J_1}} =  \y_{J_2, L_{J_2}} \quad \text{and} \quad \x_{J_1, L_{J_1}} \circ \x_{J_2, L_{J_2}} = \x_{J_1, L_{J_1}}.
\end{equation}
For simplicity of notations, once a compatible system of projections is fixed, we will write $\x_{J}$ instead of $\x_{J,L_J}$ and similarly in all other occurrences of this suffix. 
\end{defi}

\begin{ex} Given an inner product on $N_{\R}$ let $L_{J}$ be the orthogonal complement of $T_J$, then clearly this choice for every $E_J$ forms a compatible system of projections.
\end{ex}

Clearly the fact that we have a compatible system of projections implies that whenever $E_{J_2} \subset E_{J_1}$ then the following diagram commutes
\begin{equation} \label{compbndl}
    \begin{tikzcd}
         \tilde U_{J_1} \cap \tilde U_{J_2}  \arrow[rd, "\h_{J_1}"']  \arrow[r,"\gb_{J_2}"] & \tilde E_{J_2} \times V_{J_2}  \arrow[d]\\
         & V_{J_1}
    \end{tikzcd}
\end{equation}
where the vertical arrow is the projection to $V_{J_2}$ followed by the composition with $\x_{J_1}$. In other words, $\h_{J_2}$ restricted to a fibre of $h_{J_1}$ is a submersion over the fibre of $\x_{J_1}$ intersected with $V_{J_2}$. 

\section{Trimming Lagrangian pairs of pants} \label{trimming_Lag}
Our goal is to use Lagrangian pairs of pants as local models for the smoothing of the PL-lifts of tropical hypersurfaces. For this purpose we need to trim off some parts at infinity. We will discuss the cases $n=1$ and $2$. In this section we consider the $\lambda$-rescaled Lagrangian pair of pants $\Phi_{\lambda}(\tilde C)$ as in Definition \ref{LpPants}.  Since we are fixing the rescaling factor, in order to avoid cumbersome notation, we will drop the suffix $\lambda$ from our notations, i.e. we will denote the maps by $\Phi$ and $\h$. We will also continue to denote by $\mathcal H$ the image of $\h$ and by $\mathcal S_k$ the surfaces forming the boundary of $\mathcal H$. So, for instance, the surface $\mathcal S_0$ is now defined by the equation
\begin{equation} \label{bound_amoeba_rescaled}
 \mathcal S_0: \quad  (n+1)^{n+1} x_1 \ldots x_{n+1} = \lambda^{n+1} \ \text{and} \  x_j > 0, \forall j.
\end{equation}
We also fix a compatible system of projections $\{ \y_J \}$ and $\{ \x_J \}$.  

Let $\Gamma_J$ be a cone of $\Gamma$ and $r_J$ a point in the interior of $\Gamma_J$. We have 
\[ r_J = \sum_{j \in J} s_j u_j \] 
for some positive coordinates $s_j$. We will consider the following subsets of $\Gamma_J$ 
\begin{equation} \label{qsets}
 Q_{r_J} = \left \{ u = \sum_{j \in J} t_j u_j \in \Gamma_J \, | \, t_j \geq s_j \ \  \forall j \in J \right \}.
\end{equation}
Define also the following numbers
\[ r_J^+ = \max \{ s_j \}_{j \in J} \quad \text{and} \quad r_J^- = \min \{ s_j \}_{j \in J}. \]

When $\Gamma_J$ is one dimensional, there is a canonical identification of $\Gamma_J$ with $\R_{\geq 0}$ given by $u=t u_j \mapsto t$, therefore we will identify points of $\Gamma_J$ with their coordinate. In particular when $\Gamma_J$ is one dimensional $r_J = r^+_J= r^-_J$. 

\subsection{Trimming the ends over $n$-dimensional cones}  \label{Trimming1} First consider the case when $\Gamma_J$ has dimension $n$.
We want to understand the preimage of $Q_{r_J}$ by $\h_{J}$.  We will estimate its location inside $C$ and prove that for $r_J^-$ large enough, $\h_{J}$ defines a circle bundle (with fibre $\tilde E_J$) over $Q_{r_J}$. These will be the ``ends at infinity'' which we will trim off. 

The idea is the following. Let $\mathcal H_k$ and $\mathcal S_k$ be the sets defined in \S \ref{Imh} and rescaled by $\lambda$ (see discussion above). The cone $\Gamma_J$ is contained in two $n+1$ dimensional cones $\Gamma_{J^+}$ and $\Gamma_{J^-}$ corresponding to the two elements $k^+$ and $k^-$ of $\{0, \ldots, n+1 \}$ which are not in $J$. For every point $x \in Q_{r_J}$ consider the line $\x_{J}^{-1}(x)$. When $r_J$ is large enough the connected component of $\x_{J}^{-1}(x) \cap \mathcal H$ containing $x$ is a closed segment whose endpoints are intersection points of $\x_{J}^{-1}(x)$ with the hypersurfaces $\mathcal S_{k^+}$ and $\mathcal S_{k^-}$, contained respectively in the cones $\Gamma_{J^-}$ and $\Gamma_{J^+}$. The union of all these segments, as $x$ moves in $Q_{r_J}$, together with the map $\x_J$, forms a fibre bundle over $Q_{r_J}$ with fibre the segments. Notice that the preimage of each segment via $\h$ is a circle. This circle is precisely a fibre of $\h_J$.  This situation is evident in the case $n=1$. For instance Figure \ref{amoeba_projections} depicts what happens in the case $J=\{ 1 \}$: for all $r_J > \bar R$ and all $x \in Q_{r_J}$ the line $\x_{J}^{-1}(x)$ intersects both $\mathcal S_0$ and $\mathcal S_2$. The connected component of $\x_{J}^{-1}(x) \cap \mathcal H$ containing $x$ is marked with a thicker continuous line. 

\begin{figure}[!ht] 
\begin{center}
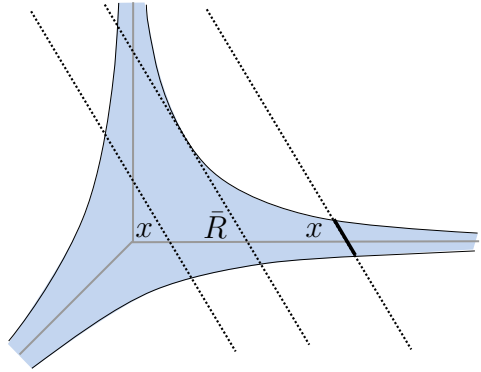
\caption{The sets $\x_J^{-1}(x) \cap \mathcal H$. For the rightmost $x$, $\x_{J}^{-1}(x)$ intersects both $\mathcal S_0$ and $\mathcal S_2$. This does not happen for the leftmost $x$. The behavior changes after $\bar R$.} \label{amoeba_projections}
\end{center}
\end{figure}
\todo{fix Figure \ref{amoeba_projections} (notation changed)}
Since this picture is intuitively clear, we will now only state without proof a few technical lemmas which quantifying more precisely the sentence ``for $r^-_J$ large enough'' and give some estimates on the size of the segments described above. 

\begin{lem} \label{interSdim2} Let $n=1$ and $J = \{ 1 \}$. There exist positive constants $\bar R_J$ and $K_J$, depending only on the projection $\x_J$, such that if  $r_J >  \bar R_J \lambda $, then for every $x = (x_1, 0) \in Q_{r_J}$, $\x_{J}^{-1}(x)$ intersects $\mathcal S_0$ transversely in either one or two points. If $x' \in \x_{J}^{-1}(x) \cap \mathcal S_0$ is the point closest to $x$, then $x'=(x'_1, x'_2)$ with 
 \[ 0 < x'_2 < \frac{K_J}{r_J} \lambda^2  \quad \text{and} \quad x'_1>  r_J - \frac{K_J}{r_J} \lambda^2. \]
\end{lem}

In the case $n=2$ we have the following:
  
 \begin{lem} \label{intersectS} Let $J= \{ 1, 2 \}$.  There exist positive constants $\bar R_J$ and $K_J$, depending only on the projection $\x_J$ (and not on $\lambda$), such that if  $r_J^- >  \bar R_J \lambda$, then for every $x = (x_1, x_2, 0) \in Q_{r_J}$, $\x_{J}^{-1}(x)$ intersects $\mathcal S_0$ transversely in either one or two points. If $x^+ \in \x_{J}^{-1}(x) \cap \mathcal S_0$ is the point closest to $x$, then $x^+=(x^+_1, x^+_2, x^+_3)$ with 
 \begin{equation}  \label{estimate3}  
         0 < x^+_3 < \frac{K_J}{(r^-_J)^2} \lambda^3 \quad \text{and} \quad x^+_j >  r^-_J - \frac{K_J}{(r^-_J)^2} \lambda^3 \ \text{for} \ j=1,2. 
 \end{equation}
 \end{lem}

Let  $\Gamma_{J^+}$ and $\Gamma_{J^-}$ be the two $n+1$-dimensional cones containing $\Gamma_J$. More precisely, if $k^{\pm}$ are the two elements of $\{0, \ldots, n+1 \}$ which are not in $J$, then $J^{\pm} = J \cup \{k^{\pm} \}$. The hypersurfaces $\mathcal S_{k^-}$ and $\mathcal S_{k^+}$ are contained in $\Gamma_{J^+}$ and $\Gamma_{J^-}$ respectively. 
We have the following 
\begin{cor} \label{intersectS2} Let $n=1$ or $2$. For any $\Gamma_J$ of dimension $n$, there exists a positive constant $\bar R_J$, depending only on the projection $\x_J$, such that for all $r_J \in \Gamma_J$ with $r_J^- > \bar R_J \lambda$ and all $x \in Q_{r_J}$, the line $\x_{J}^{-1}(x)$ intersects $\mathcal S_{k^-}$ and $\mathcal S_{k^+}$ transversely in either one or two points. 
\end{cor}

Assume that $\bar R_J$ is as in the last corollary and that $r_J$ is such that $r_J^- > \bar R_J \lambda$. For every $x \in Q_{r_J}$ let $x^{\pm} \in \x_J^{-1}(x) \cap \mathcal S_{k^{\pm}}$ be the intersection point which is closest to $x$. Clearly the segment joining $x^+$ and $x^-$ is entirely contained in $\mathcal H$ and contains $x$. Denote such a segment by $I_{J,x}$ and define
\begin{equation} \label{fibrato}
 \mathcal H_{r_J} = \bigcup_{x \in Q_{r_J}} I_{J,x}.
\end{equation}

\begin{cor} \label{hj_inVJ} Let $n=1$ or $2$ and let $\Gamma_J$ be a cone of dimension $n$. There exists a constant $\bar R_J$, depending only on the projection $\x_J$ and satisfying Corollary \ref{intersectS2}, such that if $r_J \in \Gamma_J$ satisfies $r^-_J > \bar R_J \lambda$, then
\[ \mathcal H_{r_J} \subset \inter \mathcal V_J, \]
where the set on the righthand side is defined in \eqref{h_nbhoods}. \todo{ma questo corollario mi serve davvero?}
\end{cor} 
\begin{proof} We do the case $n=2$ since the case $n=1$ is similar and easier. Assume $J= \{1,2 \}$. Let $\bar R_J$ and $K_J$ be as in Lemma \ref{intersectS} and $r_J$ such that $r^-_J > \bar R_J \lambda$.  Given $x \in Q_{r_J}$,  let $x^+ \in \x_{J}^{-1}(x) \cap \mathcal S_0$ be the end point of $I_{J,x}$. If $\bar R_J$ is chosen so that 
\[ \bar R_J > \frac{2K_J}{\bar R_J^2} \]
then the same is true for $r^-_J$. Then inequalities \eqref{estimate3} imply
\[ x^+_3 < x^+_j \quad \text{for} \ j=1,2, \]
i.e. $x^+ \in \mathcal V_J$. Obviously, also all points $x'$ on the segment from $x^+$ to $x$ lie in $\mathcal V_J$. Similarly if $x^-$ is the other end point of $I_{J,x}$. Choose $\bar R_J$ to be the constant which works for both end points.
\end{proof}
By construction $\x_J: \mathcal H_{r_J} \rightarrow Q_{r_J}$ is a fibre bundle with fibre $I_{J,x}$. We also have that $\h^{-1}(I_{J, x})$ is a circle, therefore $\h_J: \h^{-1}(\mathcal H_{r_J}) \rightarrow Q_{r_J}$ is a circle bundle. We would like to prove that $\gb_{J}: \h^{-1}(\mathcal H_{r_J}) \rightarrow \tilde E_J \times Q_{r_J}$ is diffeomorphism, thus providing a trivialization of the fibre bundle, but we first need to show that $\gb_{J}$ is well defined on $ \h^{-1}(\mathcal H_{r_J})$, i.e. that the latter is contained in $\tilde U_{J}$, the domain of the projection $\y_J$.  For this purpose we define certain neighborhoods of an edge $E_J$ of $C$. We first work inside $C^+$ and then extend to $C$ by symmetry. Let $b_J$ be the barycenter of $E_J^+$.   Given $\epsilon \in (0,1)$ and a vertex $p_j$ not in $E_J$ define the point
\begin{equation} \label{qnearface}
          q_{j,\epsilon} = \epsilon p_{j} + (1- \epsilon) b_J.
\end{equation} 
In case $n=1$ there is only one vertex $p_j$ not in $E_J$ thus we let
\begin{equation} \label{wsets}
          \mathcal W^+_{J, \epsilon} = \conv ( E^+_J \cup q_{j, \epsilon}) \cap C^+.
\end{equation}
If $n=2$, let $J = \{j_1, j_2 \}$. Then the points $p_{j_1}$ and $p_{j_2}$ are not a vertices of $E_J$.
Define
\[ \mathcal W^+_{J, \epsilon} = \conv( E^+_J \cup \{ q_{j_1, \epsilon}, p_{j_2} \}) \cap  \conv ( E^+_J \cup \{ q_{j_2, \epsilon}, p_{j_1} \})  \cap C^+ \]
\todo{check that $\conv$ is defined} Notice that when $\epsilon =1/2$, $\mathcal W^+_{J, \epsilon}$ coincides with $\mathcal W^+_{J}$ defined \eqref{facenbh}, thus $\mathcal W^+_{J, \epsilon}$ is a deformation of $\mathcal W^+_{J}$ which comes closer to $E^+_J$ as $\epsilon$ becomes small. Define $\mathcal W_{J, \epsilon}$ as usual using the involution and $\tilde{\mathcal W}_{J, \epsilon}$ by blowing up. 

\begin{lem} \label{preimSmall} Given an edge $E_J$, for every $\epsilon \in (0,1/2)$ there exists $\bar R_J > 0$, depending only on $\x_J$ and $\epsilon$ and satisfying Corollaries \ref{intersectS2} and \ref{hj_inVJ}, such that for all $r_J$ with $r^-_J > \bar R_J \lambda$
\[ \h^{-1}( \mathcal H_{r_J}) \subseteq  \tilde{\mathcal W}_{J, \epsilon}. \]
\end{lem}

\begin{proof}  We prove it for $n=2$, the case $n=1$ is similar. We can assume $J=\{1,2 \}$. We do the case $\lambda =1$, the general case follows easily. We have that $\mathcal H_{r_J} \subset \mathcal V_J$. Therefore $\h^{-1}(\mathcal H_{r_J}) \subset \mathcal W_J$. 
Let $x' \in \mathcal H_{r_J}$, i.e. $x' \in I_{J,x}$ for some $x \in Q_{r_J}$. By symmetry we can assume $x' \in \mathcal H_0$. By continuity we can also assume that $x'$ is not one of the end points of $I_{J,x}$, so that there is a unique $y \in C^+$ such that $\h(y) = x'$. We have $y \in \mathcal W_{J} \cap \mathcal W_{J_0}$. Inequalities \eqref{estimate3} imply 
\begin{equation} \label{stima_hj}
       h_j(y) >  r^-_J - \frac{K}{(r^-_J)^2} \ \text{for} \ j=1,2.
\end{equation}
Since $y \in \mathcal W_{J} \cap \mathcal W_{J_0}$, we have
\begin{equation} \label{y3y1}
     \begin{split} 
          0 & < 2y_j + \sum_{k \neq j} y_k  < \pi/2 \quad \forall j=1,2,3  \\
          y_3 & > y_j  \quad \text{for} \ j=1,2.
     \end{split}
\end{equation}
Using the symmetries we can also assume that $y_1 \geq y_2$. Then we have that for some constant $C$
\[ \begin{split} 
        h_1(y) & = \frac{\cos \left( 2y_1 + y_2 + y_3 \right) \sin y_2 \sin {y_3}}{\left[ \cos \left( y_1+y_2+y_3 \right) \sin y_1\sin y_2 \sin y_3 \right]^{\frac{2}{3}}} \leq C \frac{\sin y_2 \sin{y_3}}{\left[\sin y_1\sin y_2 \sin y_3 \right]^{\frac{2}{3}}} = \\ 
        \ & \ \\
        \ & = C \frac{(\sin y_2)^{1/3} \sin{y_3}}{\left[\sin y_1 \sin y_3 \right]^{\frac{2}{3}}} \leq C \frac{(\sin y_1)^{1/3} \sin{y_3}}{\left[\sin y_1 \sin y_3 \right]^{\frac{2}{3}}} \leq C \frac{\sin y_3}{\sin y_1} \leq 2C \frac{y_3}{y_1},
   \end{split}  
       \]  
 where in the first inequality we used the fact that we are on $\mathcal W_J \cap \mathcal W_{J_0}$, so that we can bound the factors involving the cosine; in the third inequality we used \eqref{y3y1} so that we can replace $\sin y_3$ in the denominator with $\sin y_1$. 
Then \eqref{stima_hj} implies 
\[ y_1 \leq \frac{2C}{R'} y_3 \]
where 
\[ R' = r^-_J - \frac{K}{(r^-_J)^2}. \]
This implies that $y \in \mathcal W_{J, \epsilon}$ with $\epsilon= \frac{C}{R'+C}$, which tends to $0$ as $r^-_J \rightarrow + \infty$.
\end{proof}

\begin{cor} \label{fibreBndl1} Let $E_J$ be an edge. There exists a constant $\bar R_J >0$, depending only on $\x_{J}$ and satisfying Corollaries \ref{intersectS2} and \ref{hj_inVJ}, such that for all $r_J$ with $r^-_J > \bar R_J \lambda$, we have
\[ \h^{-1}(\mathcal H_{r_J}) \subset \tilde U_{J}. \]
Moreover $\gb_J: \h^{-1}(\mathcal H_{r_J}) \rightarrow \tilde E_J \times Q_{r_J}$ is a diffeomorphism and $\Phi(\h^{-1}(\mathcal H_{r_J}))$ is the graph of $dG$ over $\tilde E_J \times Q_{r_J}$, where $G$ is the Legendre transform of $F$ defined in Proposition \ref{faceproj}.
\end{cor}
\begin{proof}
It can be seen that, for $\epsilon$ small enough, $\tilde{\mathcal W}_{J,\epsilon} \subset \tilde U_{J}$. Thus the first part of the Corollary follows from Lemma \ref{preimSmall}. To prove that $\gb_J$ is a diffeomorphism we need to prove surjectivity, but this follows from the fact that $\y_{J}$ restricted to a fibre $\h^{-1}(I_{J,x})$ is a one to one map between a pair of circles, thus it must be surjective. The last claim follows from Proposition \ref{faceproj}.
\end{proof}

\begin{rem} \label{level_set} In case $n=1$ also the inverse of Lemma \ref{preimSmall} holds, namely for any $r_J$ such that $\h^{-1}( \mathcal H_{r_J}) \subset \tilde U_J$, there exists an $\epsilon$ such that $\tilde{\mathcal W}_{J,\epsilon} \subset \h^{-1}( \mathcal H_{r_J})$. Indeed consider the boundary $\partial \tilde{\mathcal W}_{J,\epsilon}$ of $\tilde{\mathcal W}_{J,\epsilon}$, i.e. the union of the two segments joining $q_{j, \epsilon}$ to the vertices of $E_J$. It is a compact set. Then $\h_J$ restricted to $\partial \tilde{\mathcal W}_{J,\epsilon}$ tends uniformly to $\infty$ as $\epsilon \rightarrow 0$, therefore for small enough $\epsilon$, $\h_J(\partial \tilde{\mathcal W}_{J,\epsilon}) \subset Q_{r_J}$. Since $\h_J$ has no critical points on the fibres of $\y_J$ we must also have that $\h_J(\tilde{\mathcal W}_{J,\epsilon}) \subset Q_{r_J}$. The situation is depicted in Figure \ref{amoeba_level_sets}. 
\begin{figure}[!ht] 
\begin{center}
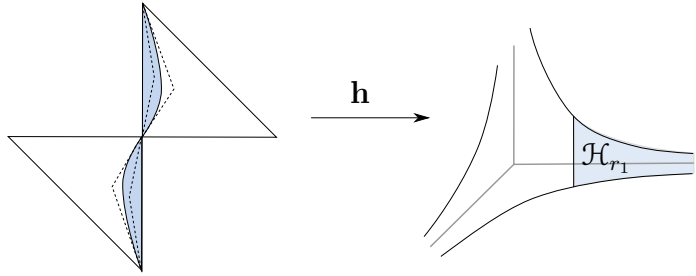
\caption{The shaded area in the co-amoeba is the preimage of $\mathcal H_{r_1}$. The dashed lines represent two neighborhoods of type $\tilde{\mathcal W}_{J, \epsilon}$.} \label{amoeba_level_sets}
\end{center}
\end{figure}
\end{rem}

We are now ready to do the first trimming of the Lagrangian pair of pants. For every $J$, with $|J|=2$, let $\bar R_J$ satisfy Corollary \ref{fibreBndl1} and choose some $r_J \in \Gamma_J$ such that $r^-_J > \bar R_J \lambda$. Then the sets $\mathcal H_{r_J}$ are pairwise disjoint (by Corollary \ref{hj_inVJ}). Define the set
\begin{equation} \label{trim1}
   \mathcal H^{[1]} = \mathcal H - \bigcup_{|J|=2} \mathcal H_{r_J}. 
\end{equation}

\subsection{Trimming the ends over $1$-dimensional cones} \label{Trimming2}
We consider a three dimensional $\lambda$-rescaled Lagrangian pair of pants $\Phi(\tilde C)$ whose set $\mathcal H$ has been trimmed over $2$-dimensional cones, as in the previous subsection, to form the set $\mathcal H^{[1]}$. Given a two dimensional face $E_{J}$ of $C$ and the restriction of $\h_{J}$ to $\h^{-1}(\mathcal H^{[1]})$. The goal is to study the fibres of this map, i.e. given a point $x \in V_J$ we want to understand $\h_{J}^{-1}(x) \cap \h^{-1}(\mathcal H^{[1]})$. We are particularly interested in the case when $x \in Q_{r_J}$ for $r_J$ large enough. In this case the connected component of $\x^{-1}_J(x) \cap  \mathcal H^{[1]}$ containing $x$ is homeomorphic to a two dimensional hyperplane amoeba (i.e. to the two dimensional version of $\mathcal H$) as in Figure \ref{amoeba_slices}. This is rather intuitive but we will sketch a proof below. Therefore the preimage of this set with respect to $\h$ is homeomorphic to a two dimensional pair of pants. We will consider the union of all such connected components of $\x^{-1}_J(x) \cap  \mathcal H^{[1]}$ as $x$ varies in $Q_{r_J}$ and show that its preimage with respect to $\h$ is contained in $\tilde U_J$ and thus $\h_J$ restricted to this set is a fibrebundle with fibre a two dimensional pair of pants. We will also study the image of these fibres with respect to $\y_J$ inside $\tilde E_J$.

Given a one dimensional cone $\Gamma_J$ and a point $x \in \Gamma_J$ define $I_{J,x}$ to be the connected component of $\x_{J}^{-1}(x) \cap \mathcal H^{[1]}$ which contains $x$. 

It is convenient to define
\[ R^+_J = \max \{ r^+_{J'} \}_{J \subset J'} \quad \text{and} \quad R^-_J = \min \{ r^{-}_{J'} \}_{J \subset J'}. \]
We also assume that the points $r_{J'}$ have been chosen so that 
\begin{equation} \label{harmless}
                   R^+_J \geq \frac{\lambda^3}{(R^{-}_J)^2}.
\end{equation} 

\begin{lem} \label{xjfibre} Let $\Gamma_J$ be a one dimensional cone. There exists a constant $K_{J}$ depending only on $x_J$, such that if $r_J \in \Gamma_J$ satisfies
\begin{equation} \label{bound_prop}
                r_J > K_J R^+_{J} 
\end{equation}
then for all $x \in Q_{r_J}$, $I_{J,x}$ is homeomorphic to the two dimensional version of $\mathcal H$ (defined in \eqref{the amoeba},  with $n=1$, see Figure \ref{amoeba_slices}). 
\end{lem}

\begin{proof}
We can assume that $J = \{ 1 \}$, so that points of $\Gamma_J$ are of the form $x = (x_1, 0, 0)$ with $x_1 > 0$. 
Observe that $x \in \Gamma_J$ belongs to $\mathcal H_0$, $\mathcal H_2$ and $\mathcal H_3$.  To determine the shape of $I_{J,x}$ we have to consider a two dimensional cone $\Gamma_{J'}$ containing $\Gamma_J$ and see what happens to $\x_J^{-1}(x) \cap \mathcal H$ when we remove its intersection with $\mathcal H_{r_{J'}}$. We can assume that $J'=\{ 1, 2 \}$. Since we have chosen a compatible system of projections, the projection $\x_J$ factors through $\x_{J'}$. So if $\x_{J'}$ is given by
\[ \x_{J'}(x'_1, x'_2, x'_3) = (x'_1 - a x'_3, x'_2 - b x'_3) \]
Then $\x_J$ can be written as
\[ \x_{J}(x'_1, x'_2, x'_3) = x'_1 - a x'_3 - c(x'_2 - b x'_3),\]
Now let $r_{J'}=(s_1, s_2, 0)$ be the point chosen to define $\mathcal H_{r_{J'}}$. Choose $r_J$ such that 
\[ r_J > (1+|c|)r^+_{J'}. \]
Then for all $x=(x_1, 0, 0) \in Q_{r_J}$ there exists a unique $x^{''}=(x^{''}_1, x^{''}_2, 0) \in \x_{J}^{-1}(x) \cap \Gamma_{J'}$ such that $x^{''}_2 = s_2$. Moreover we have
\[ x^{''}_1 = x_1 + cs_2 \geq r_J - |c|r^+_{J'} > r^+_{J'} \geq s_1.\] 
\begin{figure}[!ht] 
\begin{center}
\includegraphics{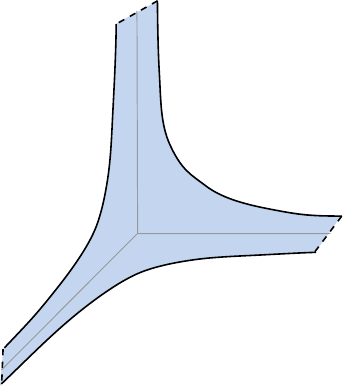}
\caption{The set $I_{J,x}$ when $\Gamma_J$ is one dimensional} \label{amoeba_slices}
\end{center}
\end{figure}
Therefore, by construction, the line $\x_{J'}^{-1}(x^{''})$ intersects $\mathcal S_0$ and $\mathcal S_3$ transversely, i.e. the segment $I_{J',x''}$ is well defined.  Moreover  $I_{J',x''}$ is contained in the boundary of $I_{J,x}$. Repeating this argument for all three of the two dimensional cones containing $\Gamma_J$ we find a $K_J$, depending only on the projection $\x_J$, such that if $r_J$ satisfies \eqref{bound_prop},  then for all $x \in Q_{r_J}$, $I_{J,x}$ must have the shape depicted in Figure \ref{amoeba_slices}, where the short dashed lines represent the three segments of the type $I_{J',x''}$ just described. This set is clearly homeomorphic to the two dimensional version of $\mathcal H$. 
\end{proof}

Given a one dimensional cone $\Gamma_J$ and $r_J$ as in the previous Lemma, let us define $\mathcal H_{r_J}$ as in \eqref{fibrato}. We want to estimate the location of $\h^{-1}(\mathcal H_{r_J})$ in side $\tilde C$. For this purpose, let us define special neighborhoods of a two dimensional face $\tilde E_J$ of $\tilde C$. Clearly $|J| =1$, i.e. $J = \{ j \}$. Let $b_J$ be the barycenter of the two dimensional face $E^+_J$ and consider the vertex $p_j$ which is the unique vertex not in $E_J$. Define the point $q_{j,\epsilon}$ on the segment between $b_J$ and $p_j$ as in \eqref{qnearface}. Let 
\[ \mathcal W^+_{J,\epsilon} = \conv ( E^+_J \cup q_{j,\epsilon}) \cap C^+ \]
and define $\mathcal W_{J, \epsilon}$ and $\tilde{\mathcal W}_{J, \epsilon}$ by symmetry and blow up as usual. Clearly we have that $\tilde{\mathcal W}_J =  \tilde{\mathcal W}_{J, 1/2}$. 

\begin{lem} \label{sizeI_x} Let $J=\{ 1 \}$ and assume that $r_J$ satisfies \eqref{bound_prop} so that $I_{J,x}$ is homeomorphic to the two dimensional version of $\mathcal H$. Then there exists a positive constant $C_J$, depending only on the projections, such that every $x'=(x'_1, x'_2, x'_3) \in I_{J,x}$ satisfies 
 \begin{equation} \label{xprime_ineq}
     \begin{split}
           & x'_1 \geq r_J-C_JR^+_J- \frac{C_J \lambda^3}{(R^-_J)^2} \\ 
           & |x'_j| \leq R^+_J +  \frac{C_J \lambda^3}{(R^-_J)^2} \quad \text{for} \ j=2,3.
     \end{split}
 \end{equation}
\end{lem}

We skip the proof, which is just an application of Lemma \ref{intersectS}. 

\begin{cor} \label{HrjVj} Let $E_J$ be a two dimensional face. Then there exists a positive constant $K'_J$, larger than the constant $K_J$ of Lemma \ref{xjfibre} and depending only on the projections, such that if  $r_J > K'_J R^+_{J}$ then 
\[  \mathcal H_{r_J} \subset  \mathcal \inter \mathcal V_{J}, \]
where the set on the righthand side is defined in \eqref{h_nbhoods}.
\end{cor}

The proof follows easily from inequalities \eqref{xprime_ineq} and condition \eqref{harmless}. 

\begin{lem} \label{preimSmall2} Let $E_J$ be a two dimensional face and let $\epsilon \in (0,1/2)$. Then there exists a positive constant $C'_J$, depending only on the projections, such that if 
\begin{equation} \label{prop_eps}
 r_J >  \frac{C'_J}{\epsilon} R^+_{J},
\end{equation}
then $r_J$ satisfies Lemma \ref{xjfibre} and the following holds
\[ \h^{-1}( \mathcal H_{r_J}) \subset  \tilde{\mathcal W}_{J, \epsilon}. \]
\end{lem}

\begin{proof}
We can assume $J={1}$ and let $x =(x_1, 0,0) \in Q_{r_J}$. Let $x'=(x'_1, x'_2, x'_3) \in I_{J,x}$. By imposing that at least $C'_J > K'_J/2$, we can assume that Corollary \ref{HrjVj} holds. Therefore, using symmetry, we can assume that $x' \in \mathcal V_J \cap  \mathcal H_0$.  Then $x'$ satisfies the inequalities \eqref{xprime_ineq}.  By continuity we can assume that $x' \in \inter \mathcal H$. Let $y$ be the unique $y \in \inter C^+$ such that $\h(y) = x'$. Since $x' \in \mathcal V_J \cap  \mathcal H_0$ we have $y \in \mathcal W^+_{J_0} \cap \mathcal W^+_{J} $ (see Lemma \ref{hNbhd}). In particular for all $j=1,2,3$
\begin{equation} \label{yjs_inq}
  0< y_1 \leq y_j \quad \text{and} \quad 0 < 2y_j + \sum_{k \neq j} y_k  \leq \pi/2 
\end{equation} 
Moreover we can also assume
\begin{equation} \label{y3geqy2}
          y_3 \geq y_2.
\end{equation}
For simplicity denote
\[ R= R^+_J +  \frac{C_J \lambda^3}{(R^-_J)^2}     \quad    \text{and}     \quad M = r_J-C_JR^+_J- \frac{C_J \lambda^3}{(R^-_J)^2}. \]
Then inequalities \eqref{xprime_ineq} imply
\begin{equation} \label{h1M} 
          \h_1(y) = \frac{ \lambda \cos(2y_1 +y_2+y_3) \sin y_2 \sin y_3}{\left( \cos \left( \sum_j y_j \right) \sin y_1 \sin y_2 \sin y_3 \right)^{2/3}} \geq M
\end{equation}
and 
\begin{equation} \label{h1h2MR}
        \frac{\h_1(y)}{\h_2(y)} = \frac{\cos(2y_1 +y_2+y_3) \sin y_2 }{\cos(2y_2 + y_1 +y_3) \sin y_1} \geq \frac{M}{R}.
\end{equation}
Therefore, using \eqref{yjs_inq} and \eqref{y3geqy2}, \eqref{h1M} implies
\begin{equation} \label{y1y2uno} 
    y_1 \leq \frac{\pi}{2} \sin y_1 \leq c \left( \frac{\lambda}{M} \right)^{3/2}
\end{equation}
for some constant $c$. On the other hand  \eqref{h1h2MR} implies 
\begin{equation} \label{y1y2due} 
          y_1 \leq \frac{\pi}{2} \sin y_1 \leq \frac{\pi R\sin y_2 }{2 M\cos(2y_2 + y_1 +y_3)} \leq \frac{\pi R \, y_2 }{2 M \cos(2y_2 + y_1 +y_3)}.
\end{equation}
When
\[ 0 \leq 2y_2+y_1+y_3 \leq \frac{\pi}{3} \]
\eqref{y1y2due} implies
\begin{equation} \label{y1y2tre} 
          y_1 \leq \frac{\pi R }{M} \, y_2.
\end{equation}
On the other hand when
\[  \frac{\pi}{3} \leq 2y_2+y_1+y_3 \leq \frac{\pi}{2} \]
we have that \eqref{yjs_inq} and \eqref{y3geqy2} imply 
\[ y_2 \geq \frac{\pi}{24}\]
therefore, by choosing $r_J$ so that 
\begin{equation} \label{Mlarge1}
  \frac{c \lambda^{3/2}}{\sqrt{M}} \leq \frac{\pi^2 R }{24}
\end{equation}
we have that \eqref{y1y2uno} implies
\[  y_1 \leq  c \left( \frac{\lambda}{M} \right) ^{3/2} \leq  \frac{\pi R }{M} \, y_2.\]
This implies that $y \in \mathcal W^+_{J, \epsilon}$ if for some constant $c'$
\[ \frac{M}{R}  \geq \frac{c'}{\epsilon}. \]
It can be easily seen that, if \eqref{harmless} holds, we can suitably choose $C'_J$ so that if $r_J$ satisfies \eqref{prop_eps}, then both \eqref{Mlarge1} and the latter inequality hold. Thus $y \in \mathcal W^+_{J, \epsilon}$.
\end{proof}

We then have 

\begin{cor} \label{fibreBndl2} Let $E_J$ be a two dimensional face. There exists a constant $K''_J$, depending only on the projections and larger than the constant $K'_J$ of Corollary \ref{HrjVj}, such that if 
\begin{equation}  \label{ineq_fbrbndl}
         r_J> K''_J R^+_{J}
\end{equation}
then
\[ \h^{-1}(\mathcal H_{r_J}) \subset \tilde U_{J}. \]
In particular $\gb_J: \h^{-1}(\mathcal H_{r_J}) \rightarrow Q_{r_J} \times \tilde E_J $ is a diffeomorphism onto its image and $\Phi(\h^{-1}(\mathcal H_{r_J}))$ is the graph of $dG$, where $G$ is the Legendre transform of $F$ defined in Proposition \ref{faceproj}.
\end{cor}

\begin{proof} Choose $\epsilon$ such that $\tilde{\mathcal W}_{J,\epsilon} \subset \tilde U_J$ and apply Lemma \ref{preimSmall2}.
\end{proof}

\begin{cor} \label{fibreBndl3}
If $r_J$ is as in Corollary \ref{fibreBndl2}, then $\h_J:  \h^{-1}(\mathcal H_{r_J}) \rightarrow Q_{r_J}$ is a fibre bundle whose fibre $\h^{-1}(I_{J,x})$ is homeomorphic to a two dimensional pair of pants.
\end{cor}

This follows directly from the previous results. 

\medskip

Let us now discuss the compatibility between the fibre bundle structures given in Corollaries \ref{fibreBndl1} and \ref{fibreBndl3}. First of all let us see how the total spaces of these fibre bundles may overlap. Suppose then that $\Gamma_{J_1}$ and $\Gamma_{J_2}$ are respectively one and two dimensional cones of $\Gamma$ such that $\Gamma_{J_1} \subset \Gamma_{J_2}$. Let $r_{J_2} \in \Gamma_{J_2}$ be the point chosen to define $\mathcal H^{[1]}$ and let $r_{J_1}$ satisfy Corollary \ref{fibreBndl2}. Now choose a second point $r'_{J_2} \in \Gamma_{J_2}$ such that
\[ Q_{r_{J_2}} \subset Q_{r'_{J_2}}\]  
and $r'^{-}_{J_2} > \bar R_{J_2} \lambda$, where $\bar R_{J_2}$ is as in Corollary \ref{fibreBndl1}. We have that $\mathcal H_{r_{J_1}}$ and $\mathcal H_{r'_{J_2}}$ have non-empty intersection. Moreover we have that by construction
\[ \x_{J_2} \left( \mathcal H_{r_{J_1}} \cap \mathcal H_{r'_{J_2}} \right) = ( Q_{r'_{J_2}} - Q_{r_{J_2}}) \cap \x^{-1}_{J_1}(Q_{r_{J_1}}) \]
Now, the restriction of the commuting diagram \eqref{compbndl} gives the following

\begin{cor} \label{compbndl2}
The following diagram commutes
\begin{equation*} 
    \begin{tikzcd}
         \h^{-1} \left( \mathcal H_{r_{J_1}} \cap \mathcal H_{r'_{J_2}} \right)   \arrow[rd, "\h_{J_1}"']  \arrow[r,"\gb_{J_2}"] & \tilde E_{J_2} \times \left( ( Q_{r'_{J_2}} - Q_{r_{J_2}}) \cap \x^{-1}_{J_1}(Q_{r_{J_1}})  \right) \arrow[d]\\
         & Q_{r_{J_1}}
    \end{tikzcd}
\end{equation*}
where the horizontal arrow is a diffeomorphism and the vertical one is projection to  $( Q_{r'_{J_2}} - Q_{r_{J_2}}) \cap \x^{-1}_{J_1}(Q_{r_{J_1}})$  composed with $\x_{J_1}$.
\end{cor}

In particular the above implies that, if we consider a fibre $\h^{-1}(I_{J_1, x})$ of $\h_{J_1}$, then the end of the leg of this fibre corresponding to $\Gamma_{J_2}$, i.e. the set $\h^{-1}(I_{J_1, x} \cap \mathcal H_{r'_{J_2}})$, has a fibre bundle structure over a segment, with fibre a circle, induced by $\h_{J_2}$.  

\begin{defi} \label{trimming_param} Given a $\lambda$-rescaled Lagrangian pair of pants $\Phi(\tilde C)$, we say that a collection of points $\{ r_J \in \inter \Gamma_J \}_{1 \leq |J| \leq n}$ is a {\bf good set of trimming parameters} if for all $J$ with $|J| =2$ we have $r^-_J > \bar R_J \lambda$, where $\bar R_J$ is as in Corollary \ref{fibreBndl1} and for all $J$ with $|J| =1$ we have that $r_J$ satisfies \eqref{ineq_fbrbndl} so that 
\[ \h^{-1}(\mathcal H_{r_J}) \subset \tilde U_{J}. \]
\end{defi}

Given a good set of trimming parameters we define $\mathcal H^{[1]}$ as in \eqref{trim1}.  Then,  Corollary \ref{HrjVj} implies that for all $J$ with $|J| =1$, the sets $\mathcal H_{r_J}$ are pairwise disjoint. We can thus define
\[ \mathcal H^{[2]} = \mathcal H^{[1]} - \bigcup_{|J|=1} \mathcal H_{r_J}. \]
Notice that $\Phi^{-1}(\mathcal H^{[2]})$ is diffeomorphic to $\tilde C$. 

We have the following useful lemma:

\begin{lem} \label{resc_lem} Let $\epsilon_1, \epsilon_2 \in (0,1/2)$ be such that for every $J$ with $|J| =j$, $\tilde{\mathcal W}_{J, \epsilon_j} \subset \tilde U_{J}$. Let $\{ r_J \in \inter \Gamma_J \}_{1 \leq |J| \leq 2}$ be a collection of points such that for all $J$ with $|J| =1$, $r_J$ satisfies \eqref{prop_eps} for $\epsilon = \epsilon_1$. Then there exists a $\lambda>0$ such that this collection is a good set of trimming parameters for the $\lambda$-rescaled Lagrangian pair of pants. Moreover for all $J$ with $|J| =j$
\[ \h_{\lambda}^{-1}(\mathcal H_{r_J}) \subset \tilde{\mathcal W}_{J, \epsilon_j}. \]
\end{lem}
\begin{proof} Let $\bar R_J$ be the constants satisfying Lemma \ref{preimSmall} for $\epsilon = \epsilon_2$. Then there exists $\lambda$ such that for all $J$ with $|J|=2$, $r^-_{J} >  \bar R_J \lambda$. Then Lemma \ref{preimSmall2} and the hypothesis guarantee that the given collection is a good set of trimming parameters for the Lagrangian pair of pants. 
\end{proof}

\subsection{Estimating the fibres over the ends of $1$-dimensional cones} \label{est_fibr}
We consider a three dimensional $\lambda$-rescaled Lagrangian pair of pants $\Phi(\tilde C)$. Given a two dimensional face $E_{J}$, we establish a result which allows some control on the image of the map $\gb_J: \tilde U_{J} \rightarrow  \tilde E_J \times V_{J}$. For this purpose we introduce some special subsets of $\tilde E_J$. 
 \begin{figure}[!ht] 
\begin{center}
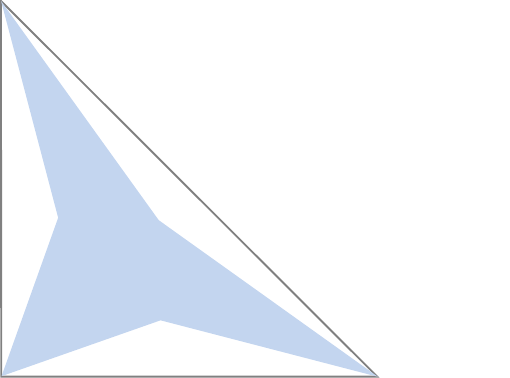
\caption{} \label{co_am_like}
\end{center}
\end{figure}
Consider $E_{J}$ as a $2$-dimensional Lagrangian coamoeba. Given some $\epsilon \in (0,1/2)$, for every one dimensional face $E_{J'}$ of $E_J$ define the subset $\mathcal W^+_{J,J', \epsilon} \subset E^+_{J}$ exactly as we defined the sets $\mathcal W^+_{J, \epsilon}$ in \eqref{wsets}, but where everything is done inside $E^+_J$ instead of $C^+$.  Then let $\mathcal W_{J,J', \epsilon}$ and $\tilde{\mathcal W}_{J,J', \epsilon}$ be as usual. Now define 
\[ \mathcal A^+_{J, \epsilon} = E_J^{+} - \bigcup_{J \subset J', |J'| = 2 } \inter( \mathcal W^+_{J,J', \epsilon}), \]
see Figure \ref{co_am_like}. Notice that vertices are included in $\mathcal A^+_{J, \epsilon}$. Let $\mathcal A_{J, \epsilon}$ and $\tilde{\mathcal A}_{J, \epsilon}$ be as usual. We have that $\tilde{\mathcal A}_{J, \epsilon}$ is a compact subset of $\tilde E_{J}$ and its interior is homeomorphic to $\tilde E_J$. 
 Let $\epsilon_1 \in (0, 1/2)$ be such that 
 \[ \tilde{\mathcal W}_{J, \epsilon_1} \subset \tilde U_{J}. \]
Then we have the following

\begin{lem} \label{size_fibr}  Let $\epsilon, \epsilon_1 \in (0, 1/2)$ and $\tilde{\mathcal A}_{J, \epsilon}$ be as above and let $K$ be a neighborhood of the origin in $L_{J}^{\perp}$. Then there exists a $\bar R \in \Gamma_J$ such that for all $r_J > \bar R \lambda$ we have
\[ \tilde{\mathcal A}_{J, \epsilon} \times Q_{r_J} \subseteq \gb_J(\tilde{\mathcal W}_{J, \epsilon_1} ). \]
Moreover, if we identify $V_J \times L_{J}^{\perp}$ with $M_{\R}$ via $(v, r) \mapsto r+v$, we have for all $r \in Q_{r_J}$ 
\[ \h(\gb_J^{-1}(\tilde{\mathcal A}_{J, \epsilon} \times \{r \}) ) \subseteq \{ r \} \times K. \]
\end{lem}

\begin{proof} It is enough to prove the case $\lambda =1$. Consider the boundary $\partial \tilde{\mathcal W}_{J, \epsilon_1} $ of $\tilde{\mathcal W}_{J, \epsilon_1}$, then it is easy to see that 
\[ \mathcal A'_{	\epsilon} = \y_{J}^{-1}(\tilde{\mathcal A}_{J, \epsilon}) \cap \partial \tilde{\mathcal W}_{J, \epsilon_1} \]
is a compact subset of $\tilde C$.  Let
\[ \bar R > \max_{\mathcal A'_{\epsilon}} \h_{J}. \]
It is now easy to see that for any $r > \bar R$ and any $y' \in \tilde{\mathcal A}_{J, \epsilon}$, there exists $y \in \tilde{\mathcal W}_{J, \epsilon_1}$ such that $\gb_{J}(y) = (r, y')$. 
Indeed let 
\[ y_0 = \y_{J}^{-1}(y') \cap \partial \tilde{\mathcal W}_{J, \epsilon_2}. \] 
Then, since $\h_{J}(y_0) < \bar R < r$ and $\lim_{y \rightarrow y'} \h_{J}(y) = + \infty$, there exists a $y \in \y_{J}^{-1}(y')$ such that $\h_{J}(y) = r$ (recall that $\y_{J}^{-1}(y')$ is one dimensional). Thus $\gb_{J}(y) = (r, y')$. This proves that $\{ r \} \times \tilde{\mathcal A}_{J, \epsilon}$ is in the image of $\gb_{J}$ and hence the first part of the statement if we take $r_J > \bar R$. 

To prove the last inclusion, we can assume $J= \{ 1 \}$. As $r \rightarrow \infty$ we have that $\gb_J^{-1}(\{ r \} \times \tilde{\mathcal A}_{J, \epsilon})$ approaches the face $\tilde E_J$. Thus, by Lemma \ref{hnear_faces}, the components $h_2$ and $h_3$ of $\h$ restricted to $\gb_J^{-1}(\{ r \} \times \tilde{\mathcal A}_{J, \epsilon})$ converge to $0$ as $r \rightarrow +\infty$. By compactness of $\tilde{\mathcal A}_{J, \epsilon}$, this convergence is uniform. Thus by taking a larger $\bar R$ also the last inclusion of the lemma holds. 
\end{proof}

\section{Lagrangian lifts of smooth tropical hypersurfaces} \label{main_results}
In this section we finally prove Theorem \ref{main_thm} for the case of tropical hypersurfaces in $\R^3$. Let $\dim N_{\R} = 3$ and fix a smooth tropical hypersurface $\Xi \subset M_{\R}$ given by a pair $(P, \nu)$ as in \S \ref{trop_hyper}. Let $\hat \Xi$ be the Lagrangian $PL$-lift of $\Xi$ inside $M_{\R} \times N_{\R}/N$ (see \S \ref{LagPLift}). We will use the following notation: given two point $q, q' \in M_{\R}$ we will denote
\[ [q,q'] = \conv \{ q, q' \}, \]
i.e. the closed segment from $q$ to $q'$. 
  
 \subsection{Compatible systems of projections} \label{global_proj} For every $k$-dimensional face $e \in (P, \nu)$, with $k=1,2$, define the following subspaces
 \begin{itemize}
 \item $N^e_{\R}$ is the $k$-dimensional vector subspace of $N_{\R}$ parallel to $e$;
 \item $T_e \subset T$ is the smallest affine subtorus of $T$ which contains $C_e$;
 \item $M^{\check e}_{\R}$ is the $(3-k)$-dimensional vector subspace  of $M_{\R}$ parallel to $\check e$;
 \item $V_{\check e}$ is the smallest affine subspace of $M_{\R}$ which contains $\check e$.
 \end{itemize} 
Obviously $N^e_{\R}$ is of the form $N^{e}_{\R} = N^e \otimes \R$, where $N^e = N^{e}_{\R} \cap N$ and similarly $M^{\check e}_{\R} = M^{\check e} \otimes \R$, where $M^{\check e} = M^{\check e}_{\R} \cap M$. 

\medskip

Choose a $(3-k)$-dimensional vector subspace $L_e \subset N_{\R}$ which is transverse to $N^e_{\R}$. This defines a unique projection $\y_e: N_{\R} \rightarrow N^{e}_{\R}$ such that $\ker \y_e = L_e$. We say that the collection of these choices forms a {\it compatible system of projections} for $(P, \nu)$ if, whenever $f \preceq e$, then $L_e \subset L_f$. This implies that $\y_f \circ \y_e = \y_f$. 
We will use the same notation to denote the projection onto $T_e$, which is well defined on suitable open neighborhoods of $T_e$, as $\y_e([y']) = [y]$ where $[y] \in T_e$ is such that $y-y' \in L_e$. 

Dually the $k$-dimensional vector subspace $L_e^{\perp}$ is transverse to $M^{\check e}_{\R}$ and it defines the projection $\x_e: M_{\R} \rightarrow V_{\check e}$ such that $\x_e(x')= x$ where $x-x' \in L_{e}^{\perp}$. Compatibility of projections implies that if $f \preceq e$ then $L^{\perp}_f \subset L^{\perp}_e$ and $\x_e \circ \x_f = \x_e$. It is easy to construct a compatible system of projections, for instance one can introduce an inner product on $N_{\R}$ and define $L_e$ to be the orthogonal complement of $N_e$. 

\medskip
 
 As in Lemma \ref{cotangent}, the choice of $L_e$ induces a natural linear symplectomorphism between the cotangent bundle of $V_{\check e} \times T_e$ and $(V_{\check e} \times T_e) \times (L^{\perp}_e \times L_e)$. Moreover the latter is naturally a covering of $M_{\R} \times N_{\R}/N$ via 
\begin{equation} \label{cotan_map}     
\begin{split}
         (V_{\check e} \times T_e) \times (L^{\perp}_e \times L_e) & \longrightarrow M_{\R} \times N_{\R}/N \\
              \left((q, y), (v,w) \right) & \mapsto (q+v, [y+w])
        \end{split} 
\end{equation}
which is a local symplectomorphism.
 
\begin{rem} \label{edge_prod_sympl} Notice that $L^{\perp}_{e}$ and $L_{e}$ can be naturally identified with the cotangent fibres of $T_{e}$ and $V_{\check{e}}$ respectively, thus $(V_{\check e} \times T_e) \times (L^{\perp}_e \times L_e)$ can also be viewed as $T^*V_{\check{e}} \times T^*T_{e}$, i.e. as $( L_e \times V_{\check e}) \times (L^{\perp}_e \times T_e)$. Indeed the symplectic form induced on $(V_{\check e} \times T_e) \times (L^{\perp}_e \times L_e)$ as a covering of $M_{\R} \times N_{\R}/N$ coincides with the symplectic form $(- \omega') \oplus \omega''$ where $\omega'$ and $\omega''$ \todo{check} are the canonical symplectic forms on $T^*V_{\check e}$ and $T^*T_e$ respectively (see Lemma \ref{cotangent}). 

\end{rem}

 \subsection{Tangent tropical hyperplanes, coamoebas and projections} \label{tgt_coa_proj} Let $e \in (P, \nu)$ be of dimension $3$. Recall definition \eqref{star_neigh} of the star-neighborhood $\Xi_{\check e}$. Define the tangent tropical hyperplane $\Gamma_e \subseteq M_{\R}$ to be the cone of this set with center $\check e$, i.e. 
 \[ \Gamma_e = \{ \check e + t(v-\check e) \in M_{\R} \, | \, v \in \Xi_{\check e} \ \text{and} \ t \in \R_{\geq 0} \}. \]
Notice that $\check e$ is the vertex of $\Gamma_e$. 
 
Now let $e  \in (P, \nu)$ be $k$-dimensional, with $k=1, 2$. As we saw in the previous subsection a covering of $M_{\R} \times N_{\R}/N$ can be written as $(V_{\check e} \times L_{e}) \times (L^{\perp}_{e} \times T_{e})$. 
Given the natural identification of $V_{\check{e}} \times L^{\perp}_{e}$ with $M_{\R}$, fix a point $q \in \inter(\check{e})$ and define
\begin{equation} \label{trop_along_edg}
     \Gamma_e = \{ t v \in L^{\perp}_{e} \, | \, q + v \in \Xi_{\check e}  \ \text{and} \ t \in \R_{\geq 0} \}.
\end{equation}
Obviously $\Gamma_e$ is independent of $q$. The choice of a point on $T_{e}$ uniquely identifies $T_e$ with $N_{\R}^{e}/ N^{e}$ (see \S \ref{global_proj} for notation). On the other hand since $L^{\perp}_{e}$ is naturally a cotangent fibre of $T_e$, it inherits from $T_e$ an integral structure, thus it can be written as $L^{\perp}_{e} = M^{e} \otimes \R = M^{e}_{\R}$ where $M^{e}$ is the dual lattice of $N^{e}$. Thus we have an identification 
\begin{equation} \label{edge_2torus}
 L^{\perp}_{e} \times T_{e} \cong M^{e}_{\R} \times N^{e}_{\R} / N^{e}.
\end{equation}

\medskip 

For any $e \in (P, \nu)$ with $\dim e = k \geq 1$, there is a one to one correspondence between $\ell$-dimensional cones of $\Gamma_e$ and $(3-k+\ell)$-dimensional polyhedra $\check f$ containing $\check e$.  Let us denote this correspondence by 
\[ \check f \mapsto \Gamma_{e,f}.\]
The cone $\Gamma_{e,f}$ is dual to the face $C_{e,f}$ of $C_e$. \todo{notation $C_{e,f}$ (or $C_e$) defined?}
 Notice that the smallest affine subspace containing $\Gamma_{e,f}$ is $V_{\check f}$ when $\dim e = 3$ or  $L^{\perp}_{e} \cap M^{\check f}_{\R}$ when $\dim \check e=1, 2$, where $M^{\check f}_{\R}$ is as in \S \ref{global_proj}. 
 
Given a $3$ dimensional $e \in (P, \nu)$ and the corresponding coamoeba $C_e$, we have a compatible system of projections $\{ \y_{e,f} \}_{f \preceq e}$ on the faces of $\tilde C_e$, where $\y_{e,f}$ is the projection onto $\tilde C_{e,f}$ induced by $\y_{f}$. Denote by $\tilde U_{e,f}$ the open subset of $\tilde C_e$ where $\y_{e,f}$ is well defined (see Definition \ref{projFace1}). 
Dually we have the projections $\{ \x_{e,f} \}_{f \preceq e}$ onto the cones $\Gamma_{e,f}$ of $\Gamma_e$, induced by the projections $\x_f$.

Similarly, led $d \in (P, \nu)$ be two dimensional and let $f$ be an edge of $d$. Then $N^f_{\R} \subset N^{d}_{\R}$ and the restriction of $\y_f$ to $N^{d}_{\R}$ induces a projection $\y_{d,f}: N^{d}_{\R} \rightarrow N^{f}_{\R}$ whose kernel is $L_{d,f} = N^{d}_{\R} \cap L_{f}$. The dual of $N^{d}_{\R}$ is identified with $L^{\perp}_{d}$ and, by compatibility of projections, the restriction of $\x_f$ to $L^{\perp}_{d}$ gives a projection $\x_{d,f}: L^{\perp}_{d} \rightarrow L^{\perp}_{d} \cap M^{\check f}_{\R}$ whose kernel is $L^{\perp}_{f}$. Clearly $\x_{d,f}$ is dual to $\y_{d,f}$, thus the collections $\{ \y_{d,f} \}_{f \preceq d}$ and $\{ \x_{d,f} \}_{f \preceq d}$ give a compatible system of projections onto the edges of $\tilde C_d$ and cones of $\Gamma_d$. We let $\tilde U_{d,f}$ be the subsets of $\tilde C_d$ where $\y_{d,f}$ is well defined as a projection onto $\tilde C_{d,f}$.

 \subsection{Local coordinates} \label{loc_coo} Given a $3$ dimensional (resp. of dimension $k=1, 2$) face $e$ of $(P, \nu)$ and the tangent tropical hyperplane $\Gamma_e$ at $\check e$, we can choose a basis $\{ u_1, \ldots, u_{3} \}$ of $M$ (resp. $\{ u_1, \ldots, u_{k} \}$ of $M^e$) such that each $u_j$ is an integral primitive generator of a one dimensional cone of $\Gamma_e$. This basis and the choice of $\check e$ as the origin defines affine coordinates $x=(x_1, \ldots, x_{3})$ on $M_{\R}$ (resp. $M^{e}_{\R}$) which identify $\Gamma_e$ with the standard tropical hyperplane $\Gamma$. Dually, let $\{u^*_1, \ldots, u^*_{3} \}$ be a basis of $N_{\R}$ (resp. $\{u^*_1, \ldots, u^*_{k} \}$ of $N^e_{\R}$) satisfying \eqref{dual_basis}. Then this basis and the choice of a vertex of the coamoeba $C_e$ as the origin of $T$ (resp. of $T_e$) defines coordinates $y=(y_1, \ldots, y_{3})$ such that $C_e$ is identified with the standard Lagrangian coamoeba $C$. It is clear that such a choice of coordinates is unique up to a transformation in the group $G^*$ and in its dual $G$.  For every $f \preceq e$ there is a unique face $E_{J_{e,f}}$ of $C$ which, in these coordinates, corresponds to $C_{e,f}$. Moreover $\Gamma_{J_{e,f}}$ corresponds  $\Gamma_{e,f}$. 

In the previous sections we defined some useful subsets of $\Gamma$ and $\tilde C$ related to their cones and faces, such as the subsets $\tilde{\mathcal W}_{J, \epsilon}$ or $\tilde{\mathcal A}_{J, \epsilon}$ of $\tilde C$. Via the above coordinates, all of these correspond to subsets of  $\Gamma_e$ or $\tilde C_e$. In order to simplify notation, when $f \preceq e$, we will do the following relabeling
\[ \tilde{\mathcal W}^{e,f}_{\epsilon} := \tilde{\mathcal W}_{J_{e,f}, \epsilon} \]
and similarly for the other subsets. 

\subsection{Inner polyhedrons} \label{inner_pol}
Let $\check f$ be a polyhedron of $\Xi$ of dimension either $1$ or $2$. Choose (and fix) a point $b_{\check f}$ in the relative interior of $\check f$ (e.g. the barycenter of $\check f$ if bounded). 
\begin{figure}[!ht] 
\begin{center}
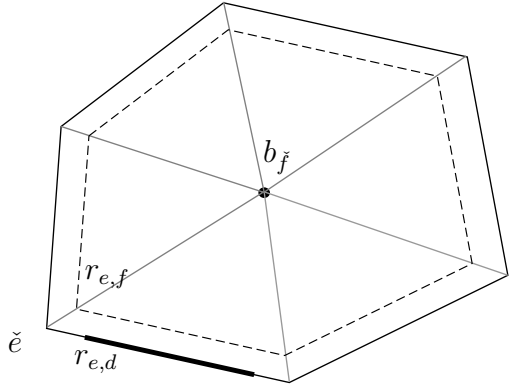
\caption{Enclosed by dashed lines is the inner polyhedron $\rho_f$ of $\check f$. The thicker black line represents the inner polyhedron $\rho_d$ of an edge $\check d$ of $\check f$.} \label{ref_points}
\end{center}
\end{figure}
Now consider, inside $\check f$, a polyhedron which is a rescaling of $\check f$ with center $b_{\check f}$. We call it an {\it inner polyhedron} of $\check f$ and denote it by $\rho_f$. In Figure \ref{ref_points}, $\rho_f$ is drawn in dashed lines when $\check f$ is two dimensional, while the inner polyhedron of an edge $\check d$ of $\check f$ is drawn as a thick black line. Given a face $\check e$ of $\check f$ (i.e. an edge or vertex), let $r_{e,f}$ be the face of $\rho_f$ corresponding to $\check e$. 

We will need three collections of inner polyhedrons $\{ \rho_f \}, \{ \rho'_f \}$ and $\{ \rho''_f \}$ satisfying the following strict inclusions
\begin{equation} \label{inner_pols_incl}
 \rho'_f \subset \rho''_f \subset \rho_f. 
\end{equation}
We will denote by $r_{e,f}$, $r'_{e,f}$ and $r''_{e,f}$ the corresponding faces.
We choose inner polyhedrons so that they satisfy the following property
\begin{enumerate}
   \item  For any two dimensional $d \in (P, \nu)$, any edge $f \preceq d$ and any $q \in \rho_d$, the affine plane $ q  \times L^{\perp}_{d}$ intersects the interior of the edges $r_{d,f}$, $r'_{d,f}$ and $r''_{d,f}$ in a point which we denote respectively by $(q, p_{d,f})$, $(q, p'_{d,f})$ and $(q, p''_{d,f})$.  Obviously $p_{d,f}$, $p'_{d,f}$ and $p''_{d,f}$  are independent of $q$ and they lie in the interior of the cone $\Gamma_{d,f}$ of $\Gamma_d$.
\end{enumerate}

When $\check f$ is two dimensional, we can use this data to subdivide it as in Figure \ref{innerpol_subdiv}. The elements of this subdivision are: the inner polyhedron $\rho_f$, a parallelogram $Y_{d,f}$ for each edge $\check d$ of $\check f$ and a polyhedral (non-convex) shape $Y_{e,f}$ for each vertex $\check e$ of $\check f$. 
For instance $Y_{d,f}$ is constructed as follows: one of its edges is the inner polyhedron $\rho_{d}$, the opposite edge is obtained by translating $\rho_d$ by the vector $p_{d,f}$ defined above. By property (1) above, the latter edge is contained in $r_{d,f}$. When $\check e$ is a vertex the definition of $Y_{e,f}$ follows similarly.

\begin{figure}[!ht] 
\begin{center}
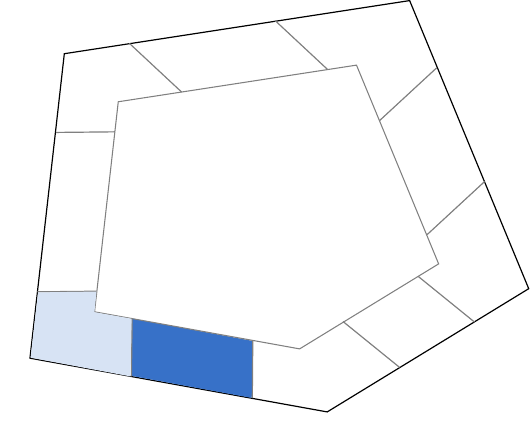
\caption{} \label{innerpol_subdiv}
\end{center}
\end{figure}

For every $e \in (P, \nu)$ of dimension $3$ or $2$ we define
\begin{equation} \label{Y_sets}
 Y_{e} = \bigcup_{f \preceq e, \dim f=1} Y_{e,f}.
\end{equation}

We will denote by $Y'_{e,f}$ and $Y''_{e,f}$ the elements of the subdivision induced by the collections $\{ \rho'_{f} \}$ and $\{ \rho''_{f} \}$ respectively and by $Y'_e$ and $Y''_e$ their corresponding union as in \eqref{Y_sets}.

For every three dimensional $e$ and every edge $f$ of $e$, the reference point $r'_{e,f}$ is on the cone $\Gamma_{e,f}$ of $\Gamma_e$ therefore we can use it to define the sets $Q_{r'_{e,f}}$ as in \eqref{qsets}. Denote 
\begin{equation} \label{trunc_Gamma_e}
      \Gamma_e^{[1]} = \Gamma_e - \bigcup_{f \preceq e, \dim f=1} Q_{r'_{e,f}}.
\end{equation}
Similarly if $d$ is two dimensional and $f$ is an edge of $d$, the points $p'_{d,f}$ are on the cone $\Gamma_{d,f}$ of $\Gamma_d$. Thus they define subsets $Q_{p'_{d,f}}$ of $\Gamma_{d,f}$. We define 
\begin{equation} \label{trunc_Gamma_d}
          \Gamma_d^{[1]} = \Gamma_d -  \bigcup_{f \preceq d, \dim f=1} Q_{p'_{d,f}}.
\end{equation}

\subsection{Neighborhoods} \label{intorni} For every vertex $\check e$ of $\Xi$, let $B_{e}$ be a small convex open neighborhood  of the set $Y'_e$ in $M_{\R}$. For every two dimensional $d$, let $B_{d}$ be a bounded convex open subset  inside $L^{\perp}_{d}$ which contains $\Gamma_d^{[1]}$. 
For every one dimensional $f$, fix an open interval $B_f \subset L^{\perp}_f$ containing the origin. 
When $d$ is two dimensional and $f \preceq d$, then $B_f \subset L^{\perp}_d$. Define 
 \[ B_{d,f} =  B_f \times Q_{p_{d,f}}, \]
which is a convex neighborhood of the set $Q_{p_{d,f}} \subset \Gamma_{d,f}$ (see Figure \ref{neighborhoods_d}). 
We require that the inner polyhedrons and these neighborhoods satisfy the following properties
\begin{enumerate}
\item the subsets $\{ B_{e} \}_{\dim e=3}$ are pairwise disjoint;
\item if $j=1$ or $2$ the subsets  $\{ \rho_f \times B_{f} \}_{\dim f=j}$ are pairwise disjoint;
\item when $\dim e=3$ and $\dim f=1$ or $2$, then $B_{e} \cap (\rho_f \times B_f) \neq \emptyset$ if and only if $f \preceq e$;
\item when $\dim d=2$ and $\dim f=1$ then $(\rho_d \times B_d) \cap (\rho_f \times B_f) \neq \emptyset$ if and only if $f \preceq d$. 
\item for all $(d,f)$ with $\dim d =2$ and $f$ and edge of $d$ \todo{is this needed?}
\[ [p_{d,f}, p'_{d,f}] \times B_f \subset B_d; \] 
\item for all $(e,d)$ with $\dim e= 3$ and $d$ a two dimensional face of $e$ \todo{i don't think we need this condition}
\[ [r_{e,d}, r'_{e,d}] \times B_d\subset B_e; \]
\item for all $(e,d)$ with $\dim e= 3$ and $d$ a two dimensional face of $e$
\begin{equation} \label{kd_gamma}
        (Q_{r_{e,d}} \times B_d) \cap \Gamma^{[1]}_e = Q_{r_{e,d}} \times \Gamma_d^{[1]}, 
  \end{equation}
  where $\Gamma^{[1]}_e$ and  $\Gamma_d^{[1]}$ are defined in \eqref{trunc_Gamma_e} \eqref{trunc_Gamma_d} respectively.
\item for all $(e,f)$ with $\dim e= 3$ and $f$ and edge of $e$
\[ Q_{r_{e,f}} \times B_f \subset \mathcal V_{e,f}. \]
 \end{enumerate}

It is easy to see that the inner polyhedrons and the neighborhoods can be chosen so that conditions $(1)-(8)$ hold. 
Condition (8) also implies that for all edges $f$ of $e$,
 \begin{equation} \label{bief_gamma}
        (Q_{r_{e,f}} \times B_f) \cap \Gamma_e = Q_{r_{e,f}}. 
  \end{equation}
Moreover it also implies that for all $(d,f)$ with $\dim d =2$ and $f$ an edge of $d$
\begin{equation} \label{bidf_vdf}
B_{d,f} \subset \mathcal V_{d,f} 
\end{equation}
and
\begin{equation} \label{bidf_gammad}
       B_{d,f} \cap \Gamma_d = Q_{p_{d,f}}.
\end{equation}
\begin{figure}[!ht] 
\begin{center}
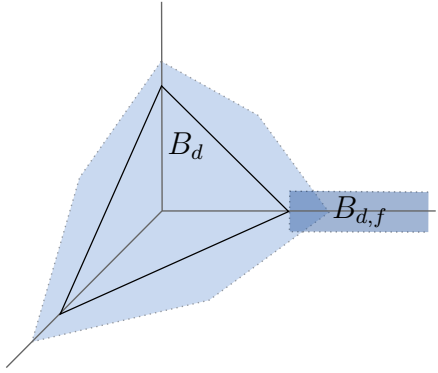
\caption{The neighborhoods $B_{d}$ and $B_{d,f}$. The triangle is the convex hull of the points $p_{d,f}$. } \label{neighborhoods_d}
\end{center}
\end{figure}
\subsection{Fixing the inner polyhedrons} \label{fix_poly}  Consider the pairs $(e,d)$ where $\dim e=3$ and $d \preceq e$ is a two dimensional face. Choose $\epsilon_1 \in (0,1/2)$ so that for all such pairs 
\begin{equation} \label{eps1_incl}
     \tilde W^{e,d}_{\epsilon_1} \subset \tilde U_{e,d}. 
\end{equation}
In the tropical hyperplane $\Gamma_{e}$, for every $f \preceq e$ with $\dim f=1$ consider the points  $r'_{e,f} \in \Gamma_{e,f}$ and for every $d  \preceq e$ with $\dim d =2$ consider the points $r_{e,d} \in \Gamma_{e,d}$. 
We choose the size of the inner polyhedrons so that these collections of points satisfy \eqref{prop_eps} with $r_{J} = r_{e,d}$, $r_{J'} = r'_{e,f}$ and $\epsilon =\epsilon_1$. This can be easily achieved by taking $\rho'_{f}$ sufficiently close to $\check f$ and leaving $\rho_d$ fixed. Notice that conditions (1)-(8) of the previous section still hold, perhaps after taking the segment $B_f$ smaller when $f$ is an edge. 

\subsection{Preparing the local model along edges} \label{resc_edg}
Consider the pairs $(d,f)$ where $\dim d =2$ and $f$ is an edge of $d$. Choose an $\epsilon_3> 0$ such that for all such pairs, inside $\tilde C_d$ we have
\[ \tilde{\mathcal W}^{d,f}_{\epsilon_3} \subset \tilde U_{d,f} \]
(see \S \ref{loc_coo} and \S \ref{tgt_coa_proj} for notation). Given a three dimensional $e \in (P, \nu)$ containing $d$ and viewing $\tilde C_d$ as a face of $\tilde C_e$, we  assume that $\epsilon_3$ is small enough so that the following property holds
\begin{equation} \label{eps3_cond}
\y_{e,d}^{-1}(\tilde{\mathcal W}^{d,f}_{\epsilon_3}) \cap \tilde{\mathcal W}^{e,d}_{\epsilon_1} \subset \tilde{\mathcal W}^{e,f},
\end{equation}
where the latter set corresponds to $\tilde{\mathcal W}_{J_{e,f}}$ as defined in \S \ref{Imh}.

For each edge $f$ of $d$ let $p_{d,f} \in \Gamma_{d,f}$ be the point defined in \S \ref{inner_pol} . If $\bar R_{J_{d,f}}$ is the constant given in Lemma \ref{preimSmall} for $\epsilon = \epsilon_3$, let $\lambda$ be such that for all edges $f$ of $d$,  $p_{d,f} > \bar R_{J_{d,f}} \lambda$. Then we can consider the $\lambda$-rescaled Lagrangian pair of pants $\overline{\Phi}_d: \tilde C_d \rightarrow L^{\perp}_d \times T_d$. More precisely, via local coordinates, we can consider the function $F$ (given in \eqref{Fglob}) as being defined on $\tilde C_d$ and then rescaled by $\lambda$. Then $\overline \Phi_d$ is defined as the graph of $dF$ (in the sense of \eqref{phi:graph}):
\begin{equation} \label{edge_loc_mod_r}
\begin{split}
\overline \Phi_{d}: \tilde C_{d} & \rightarrow L^{\perp}_{d} \times T_{d}   \\ 
                                            q & \mapsto ((dF)_q, q ).
\end{split}
\end{equation}
where $L^{\perp}_{d}$ is identified with the cotangent fibre of $T_{d}$. Let the associated map $\overline{\h}_d$ be given by composition of $\overline{\Phi}_d$ with the projection on $L^{\perp}_d$ and denote its image by $\mathcal H_d$. 

To get the three dimensional model we consider $F$ as a function on $\rho_d \times \tilde{C}_d$, where $\rho_d \subset \check d$ is the inner polyhedron, and define the local model along the edge as:
\begin{equation} \label{edge_loc_mod1}
\begin{split}
\Phi_{d}:  \rho_d \times \tilde C_{d} & \rightarrow (L_{d} \times L^{\perp}_{d} ) \times (V_{\check d} \times T_{d})   \\ 
                                            q & \mapsto ((dF)_q, q ).
\end{split}
\end{equation}
We denote by $\h_d$ the left composition of $\Phi_d$ with projection onto $L^{\perp}_{d} \times V_{\check d}$. Clearly its image is just $\mathcal H_d \times \rho_d$. Here the righthand space is identified with (a covering of) $M_{\R} \times N_{\R}/N$ via \eqref{cotan_map}.   We define
\[ \overline{\h}_{d,f} = \x_{d,f} \circ \overline{\h}_{d}, \]
corresponding to $\h_{J_{d,f}}$ in local coordinates.  Similarly we name by $\overline{\gb}_{d,f}$ the map corresponding to $\gb_{J_{d,f}}$ (see \eqref{mapg}).  

By the above choice of the rescaling factor $\lambda$, we have that Lemma \ref{preimSmall} holds for $\epsilon= \epsilon_3$ and $r_{J_{d,f}} = p_{d,f}$. In particular we can define the subsets $\mathcal H_{p_{d,f}} \subset \mathcal H_d$, which fibre over $Q_{p_{d,f}}$ with fibres the segments $I^{d,f}_q$. Moreover
\begin{equation} \label{edge_ref_pt}
\overline{\h}^{-1}_d(\mathcal H_{p'_{d,f}}) \subset \overline{\h}^{-1}_d(\mathcal H_{p_{d,f}}) \subset \tilde{\mathcal W}^{d,f}_{\epsilon_3} \subset \tilde U_{d,f}.
\end{equation}
Thus also Corollary \ref{fibreBndl1} holds. Given the points $p_{d,f}, p'_{d,f}$ and $p''_{d,f}$ as in \S \ref{inner_pol}, denote the following subsets of $L^{\perp}_d$
\begin{equation} \label{edge_H1}
\begin{split}
        H_d & = \mathcal H_d - \bigcup_{f \preceq d, \dim f=1} \mathcal H_{p_{d,f}}, \\
        H'_d & = \mathcal H_d - \bigcup_{f \preceq d, \dim f=1} \mathcal H_{p'_{d,f}}, \\
        H''_d & = \mathcal H_d - \bigcup_{f \preceq d, \dim f=1} \mathcal H_{p''_{d,f}}.
\end{split}
\end{equation}
Obviously we have 
\[ H_d \subset H''_d \subset H'_d. \]
Recall the neighborhoods $B_d$ and $B_{d,f}$ defined \S \ref{intorni}. After eventually rescaling with a smaller $\lambda$, we can also assume 
\begin{equation} \label{ends_in_nbhd}
     \mathcal H_{d} \cap B_{d,f} = \mathcal H_{p_{d,f}} 
\end{equation}
and therefore by property (5) of \S  \ref{intorni} and the convexity of $B_d$
\begin{equation} \label{body_in_nbhd}
    H'_d \subset B_d.
\end{equation}

Following Remark \ref{level_set} we can choose an $\epsilon'_2$, independent of $(d,f)$, such that 
\begin{equation} \label{e2}
   \tilde{\mathcal W}^{d,f}_{\epsilon'_2} \subset \overline{\h}^{-1}_d(\mathcal H_{p'_{d,f}}).
 \end{equation}
 
We will also need the following definition. 

\begin{figure}[!ht] 
\begin{center}
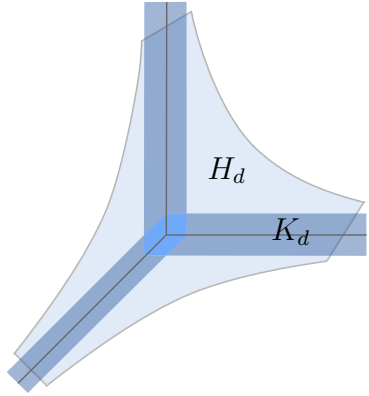
\caption{The sets $H_d$ and $K_d$. The small hexagon in the center is $K'_d$} \label{k_and_h}
\end{center}
\end{figure}

\begin{propdefi} \label{constrain_sets} Given the tangent tropical line $\Gamma_d$, the vectors $\{ u_0, u_1, u_2 \}$ generating its one dimensional cones as in \S \ref{loc_coo} and $t \in \R_{>0}$, define the hexagon
\[  K'_{d}   = \conv \{tu_0, -tu_0, tu_1, -tu_1, tu_2, -tu_2 \}. \]
For every edge $f$ of $d$ let
\[ K_{d,f}  = \Gamma_{d,f} + K'_{d}  \]
and define 
\[  K_d  = \bigcup_{f \preceq d, \dim f=1} K_{d,f}, \]
see Figure \ref{k_and_h}. For sufficiently small $t$, these sets have the following properties
\begin{itemize}
\item [a)] $ K'_d \subset H_d;$
\item [b)] for every edge $f$ of $d$ the boundary points of $I^{d,f}_{p'_{d,f}}$ are outside $K_d$;
\item [c)] for every $p \in K'_d$ and $q \in H_d$ the segment from $p$ to $q$ lies inside $H_d$;
\item [d)] if a point $q$ lies on the segment between points $p_1 \in K_{d,f} \cap H_d$ and $p_2 \in \mathcal H_{p_{d,f}}$ and satisfies $\x_{f} (q) < p_{d,f}$ then $q \in H_d$.  
\end{itemize}
Properties $(c)$ and $(d)$ hold also if we replace $H_d$ with $H'_d$ or $H''_d$ and $p_{d,f}$ with $p'_{d,f}$ or $p''_{d,f}$ respectively. 
\end{propdefi}

\begin{proof} These are easy geometric consequences of the definitions. 
\end{proof}

\subsection{Preparing the local models over vertices.} \label{resc_trim_vert} Let $e \in (P, \nu)$ be three dimensional. Given $\epsilon'_2$ as in \S \ref{resc_edg}, satisfying \eqref{e2}, choose an $\epsilon_{2} < \epsilon'_2$ such that for all edges $f$ of $e$ we have 
\begin{equation} \label{eps2_incl}
    \tilde{\mathcal W}^{e,f}_{\epsilon_{2}} \subset \tilde U_{e,f}
\end{equation}
and for all two dimensional faces $d$ of $e$ the following holds
\begin{equation} \label{trim_neigh}
          \forall \ \text{edges} \ f \preceq d, \quad \y_{e,d}(\tilde{\mathcal{W}}^{e,f}_{\epsilon_2} \cap \tilde{\mathcal W}^{e,d}_{\epsilon_1}) \subset \tilde {\mathcal W}^{d,f}_{\epsilon'_2}.
\end{equation}
where $\epsilon_1$ was chosen in \S \ref{fix_poly}. 

We have that by \eqref{eps1_incl}, \eqref{eps2_incl} and the criterion in \S \ref{fix_poly}, the numbers $\epsilon_1$ and $\epsilon_2$ and the collection of points  $\{ r_{e,f} \in \Gamma_{e,f} \}_{f \preceq e}$ satisfy the hypothesis of Lemma \ref{resc_lem}. Therefore there exists a $\lambda$ such that the collection $\{ r_{e,f} \in \Gamma_{e,f} \}_{f \preceq e}$ is a good set of trimming parameters for a $\lambda$-rescaled Lagrangian pair of pants. More precisely, via the coordinates identifying $C_e$ with $C$ fixed in \S \ref{loc_coo}, we can consider the function $F$ of \eqref{Fglob}, rescaled by $\lambda$, as a function on $\check e \times \tilde C_e$ and . The local model at the vertex $\check e$ is given by the graph of $dF$ (in the sense of \eqref{phi:graph}):
\begin{equation} \label{local_vert}
   \begin{split} 
     \Phi_e: \check e \times \tilde C_e & \rightarrow M_{\R} \times N_{\R} / N  \\
                                           q              & \mapsto (\check e + (dF)_q, q). 
   \end{split}    
\end{equation}
We will also denote by $\h_e:  \check e \times \tilde C_e \rightarrow M_{\R} $ the left composition of $\Phi_e$ with the projection onto $M_{\R}$ and by $\mathcal H_{e}$ the image of $\h_e$. Of course, in the local coordinates of \S \ref{loc_coo}, $\Phi_e$, $\h_e$ and $\mathcal H_e$ coincide with $\Phi$, $\h$ and $\mathcal H$ (rescaled by $\lambda$). For every face $f$ of $e$ we also denote
\[ \h_{e,f} = \x_{e,f} \circ \h_{e}. \]
In local coordinates $\h_{e,f}$ coincides with $\h_{J_{e,f}}$. Similarly we name by $\gb_{e,f}$ the map corresponding to $\gb_{J_{e,f}}$ (see \eqref{mapg}).

Notice that, by the criterion in   \S \ref{fix_poly}, also the collection
 \begin{equation} \label{trimming_coll}
       \{ r'_{e,f} \in \Gamma_{e,f}, \dim f =1 \} \cup \{ r_{e,d} \in \inter \Gamma_{e,d}, \dim d =2 \}
\end{equation}
forms a good set of trimming parameters. 

For every edge $f$ of $e$ the points $r_{e,f} \in \Gamma_{e,f}$ satisfy Corollary \ref{intersectS2} and for every $x \in Q_{r_{e,f}}$ we can define the segments $I^{e,f}_{x}$ (see \S \ref{Trimming1}) and the subsets $\mathcal H_{r_{e,f}} \subset \mathcal H_e$ as in \eqref{fibrato}. Moreover, by construction, we have
\begin{equation} \label{vert_resc_concl}
   \h^{-1}_e(\mathcal H_{r'_{e,f}}) \subset \h^{-1}_e(\mathcal H_{r_{e,f}}) \subset \tilde{\mathcal W}^{e,f}_{\epsilon_{2}} \subset \tilde U_{e,f},
\end{equation}
so that Corollary \ref{fibreBndl1} also holds for the points $r_{e,f}$ and $r'_{e,f}$.  By eventually rescaling by a smaller $\lambda$ we can also assume that, for any edge $f$ and two dimensional face $d$ of $e$, the following conditions are met 
\begin{enumerate}
\item if  $B_f \subset L^{\perp}_f$ is the set defined in \S \ref{intorni}, then
\begin{equation} \label{schiacciato}
     (B_f \times Q_{r_{e,f}}) \cap \mathcal H_e = \mathcal H_{r_{e,f}} ,
\end{equation} 
\item given the set $K'_d$ as in Proposition-Definition \ref{constrain_sets}, the statement of Lemma \ref{size_fibr} holds for $\epsilon_1$ chosen as in \S \ref{fix_poly}, $\epsilon=\epsilon'_2$, $K = K'_d$ and $r_J = r_{e,d}$. In particular we have that 
\begin{equation} \label{AinW} 
                    \tilde{\mathcal A}^{e,d}_{\epsilon'_2} \times Q_{r_{e,d}} \subseteq \gb_{e,d}(\tilde W^{e,d}_{\epsilon_1}).
\end{equation}
Moreover, for any $r \in Q_{r_{e,d}}$
\begin{equation} \label{hA_small}
 \h_e(\gb_{e,d}^{-1}(\tilde{\mathcal A}^{e,d}_{\epsilon'_2} \times \{r \}) ) \subseteq \{ r \} \times K'_d. 
\end{equation}

\end{enumerate}

We can define the first trimming of $\mathcal H_e$ by 
\begin{equation} \label{vert_first_trim}
     \mathcal H^{[1]}_{e} = \mathcal H_e - \bigcup_{f \preceq e, \, \dim f=1} \mathcal H_{r'_{e,f}}.
\end{equation}
Moreover, for all two dimensional faces $d$ of $e$, we have that $r_{e,d}$ satisfies Corollary \ref{fibreBndl2}. In particular for all $x \in Q_{r_{e,d}}$ we have the fibres $I^{e,d}_x \subset \mathcal H_e^{[1]}$, whose preimages under $\h_e$ are two dimensional pairs of pants. We also have the subsets $\mathcal H_{r_{e,d}}$ which satisfy
\begin{equation} \label{trim_condition}
         \h_e^{-1}( \mathcal H_{r'_{e,d}})  \subset \h_e^{-1}( \mathcal H_{r_{e,d}}) \subset \tilde{\mathcal W}^{e,d}_{\epsilon_1} \subset \tilde U_{e,d}.
\end{equation}
We then define the second trimming 
\begin{equation} \label{vert_first_trim_2}
     \mathcal H^{[2]}_{e} = \mathcal H^{[1]}_e - \bigcup_{d \preceq e, \, \dim d=2} \mathcal H_{r'_{e,d}}.
\end{equation}
By eventually rescaling with a smaller $\lambda$, we can also assume that:
\begin{enumerate}  \setcounter{enumi}{2}
\item for every three dimensional $e$
\begin{equation} \label{h2_in_nbhd}
       \mathcal H^{[2]}_e \subset B_{e}. 
\end{equation} 
\item for every two dimensional face $d$ of $e$
\begin{equation} \label{Bd_inters_H}
(B_{d} \times Q_{r_{e,d}}) \cap \mathcal H_e^{[1]} = \mathcal H_{r_{e,d}} 
\end{equation}
\item if $K_d \subset L^{\perp}_{d}$ is as in Proposition-Definition \ref{constrain_sets} and $H'_d$ as in \eqref{edge_H1}, then for every $x \in Q_{r_{e,d}}$
\begin{equation} \label{ieder_inH1}
  I^{e,d}_x \subset  (K_{d} \cap H'_d) \times \{ x \}.
\end{equation}
\end{enumerate}
Indeed we have that $\mathcal H^{[1]}_e \cap \Gamma_e = \Gamma^{[1]}_e$ and $\mathcal H^{[2]}_e \cap \Gamma_e = Y'_{e}$ and $\mathcal H_{r_{e,d}} \cap \Gamma^{[1]}_e = \Gamma^{[1]}_d \times Q_{r_{e,d}}$ thus by a sufficiently small $\lambda$ we can assume $\mathcal H^{[1]}_e$, $\mathcal H^{[2]}_e$ and $\mathcal H_{r_{e,d}}$ to be arbitrarily close to $\Gamma^{[1]}_e$, $Y'_{e}$ and $\Gamma^{[1]}_d \times Q_{r_{e,d}}$ respectively. Thus \eqref{h2_in_nbhd} follows from the fact that $B_e$ is a neighborhood of $Y'_e$. Equality \eqref{Bd_inters_H} follows from \eqref{kd_gamma}, while \eqref{ieder_inH1} follows form the fact that $K_d \cap H'_d  \subset B_d$ and that $((K_d \cap H'_d) \times Q_{r_{e,d}}) \cap \Gamma^{[1]}_e = \Gamma^{[1]}_d \times Q_{r_{e,d}}$.

Notice that for all $r \in Q_{r_{e,d}}$ we have 
\begin{equation} \label{hAepsilon}
          \h_e(\gb_{e,d}^{-1}(\tilde{\mathcal A}^{e,d}_{\epsilon'_2} \times \{r \}) ) \subset I^{e,d}_{r}.
\end{equation}
Indeed \eqref{AinW}, \eqref{trim_neigh}, \eqref{vert_resc_concl} ensure if $y \in \gb_{e,d}^{-1}(\tilde{\mathcal A}^{e,d}_{\epsilon'_2} \times \{r \})$ then $\h_e(y) \in \mathcal H^{[1]}_e$. Therefore the inclusion follows from \eqref{hA_small}, \eqref{body_in_nbhd}, part (a) of Proposition-Definition \ref{constrain_sets} and \eqref{Bd_inters_H}.

\begin{lem} \label{hdomains}
Let $\overline{\h}_d: \tilde C_d \rightarrow L^{\perp}_d$ be the map from \S \ref{resc_edg} and let $H'_d$ be as in \eqref{edge_H1}. Then, for all $r \in Q_{r_{e,d}}$, we have 
\[ \overline{\h}^{-1}_d(H'_d) \subset  \tilde{\mathcal A}^{e,d}_{\epsilon'_2} \subset \y_{e,d}(\h^{-1}_{e}(I^{e,d}_{r})). \]
\end{lem}
\begin{proof}
These inclusions follow from \eqref{e2} and \eqref{hAepsilon}.
\end{proof}

\medskip

We can now give a provisional definition of the trimmed local model. 
\begin{defi}[Provisional] \label{defi_loc_mod_vert}
The local model at a vertex $\check e$ is given by \eqref{local_vert}, but now $\Phi_e$ is rescaled as explained in this subsection and its domain is restricted to the subset 
\[ \mathcal Z_e = \h^{-1}_e(\mathcal H^{[2]}_e) \subset \check{e}  \times \tilde C_{e}. \]
 We denote this local model by $(\Phi_e, \mathcal Z_e)$.  
\end{defi} 

\subsection{Gluing the local models} \label{gluing_mdls} We are ready to do the first gluing: given the local model over the vertex $\check e$ we glue its ends over one dimensional cones to the local model over the edge corresponding to that cone. So let $d$ be a two dimensional face of $e$. From the local model over $\check e$, we have that the end of $\mathcal H^{[2]}_{e}$ over the cone $\Gamma_{e,d}$ is given by the subset $\mathcal H_{r_{e,d}} - \mathcal H_{r'_{e,d}}$. For simplicity of notation, let us denote 
\[ \mathcal H_{[r_{e,d}, r'_{e,d})} :=  \mathcal H_{r_{e,d}} - \mathcal H_{r'_{e,d}}. \]
As we already recalled,  Corollaries \ref{fibreBndl2}, \ref{fibreBndl3}, \ref{compbndl2} hold for the constant $r_{e,d}$. Hence we have a fibre bundle
\[ \h_{e,d}: \h^{-1}_e(\mathcal H_{[r_{e,d}, r'_{e,d})}) \rightarrow  [r_{e,d}, r'_{e,d})\]
with fibre homeomorphic to a two dimensional pair of pants.  Then we have  
\[ \gb_{e,d}: \h^{-1}_e(\mathcal H_{[r_{e,d}, r'_{e,d})}) \rightarrow  [r_{e,d}, r'_{e,d}) \times \tilde C_{d} \]
which is a diffeomorphism onto its image. Let us denote this image by
\[ \mathcal Z^0_{e,d} := \gb_{e,d}( \h^{-1}_e(\mathcal H_{[r_{e,d}, r'_{e,d})})) \subset \rho_d \times \tilde C_d.\]
Then $\Phi_e(\h^{-1}_e(\mathcal H_{[r_{e,d}, r'_{e,d})}))$ is the graph of $dG_{e,d}$ for some function $G_{e,d}$ over $\mathcal Z^0_{e,d}$. 

Let us now look at the local model over the edge from \S \ref{resc_edg}. Recall that $\Phi_d$ is defined in \eqref{edge_loc_mod1} as the graph of $dF$, for suitable $F$. The idea is to interpolate the two maps via a partition of unity. 

Recall also that we defined a third inner polyhedron $\rho''_d$ which is nested between $\rho_d$ and $\rho'_d$, and defines a point $r''_{e,d} \in [r_{e,d}, r'_{e,d}]$. Choose some $\bar r_{e,d} \in [ r_{e,d}, r'_{e,d}]$ so that  
\[ r_{e,d} < r''_{e,d} < \bar r _{e,d}< r'_{e,d}. \]
Define 
\[ \mathcal Z^{\infty}_{e,d} := (\bar r_{e,d}, r'_{e,d}) \times \tilde C_{d} \]
and consider the following open subset of $[ r_{e,d}, r'_{e,d}) \times \tilde C_{d}$ 
\[ \mathcal Z_{e,d} = \mathcal Z^{0}_{e,d} \cup \mathcal Z^{\infty}_{e,d}. \]
Let $\eta: [r_{e,d}, r'_{e,d}) \rightarrow \R$ be some smooth, non-increasing function such that
\[ \eta(t) = \begin{cases}
                         1 \quad t \in [r_{e,d}, r''_{e,d}],          \\
                         0 \quad t \in [\bar r_{e,d}, r'_{e,d}). 
                 \end{cases}
 \]
On the open subset $\mathcal Z_{e,d}$ of $\rho_d \times \tilde C_{d}$ define the following function 
 \[ F_{e,d}(t, y)= \begin{cases}
                 \eta(t) G_{e,d}(t,y) + (1-\eta(t)) F(y) \quad \text{on} \ \mathcal Z^{0}_{e,d}, \\
                 F(y) \quad \text{on} \ \mathcal Z^{\infty}_{e,d},
                 \end{cases} \]
  where $G_{e,d}$ is the function coming from the local model over $\check e$ and $F$ is the function from the local model over $\check d$. Clearly $F_{e,d}$ is well defined and smooth.

 \begin{defi} \label{gluing_models} Let
\[ \mathcal Z_{d} = (\rho'_{d} \times \tilde C_{d}) \cup \left( \bigcup_{d \preceq e} \mathcal Z_{e,d} \right) \]
and let $F_d: \mathcal Z_{d} \rightarrow \R$ be the function which coincides with $F$ on $\rho'_{d} \times \tilde C_{d}$ and with $F_{e,d}$ on $\mathcal Z_{e,d}$. Clearly $F_d$ is smooth. Given the identification of the cotangent bundle of $\rho_d \times C_d$ with $(L_d \times L_d^{\perp}) \times (\rho_d \times C_d)$ of Lemma \ref{cotangent}, we redefine the new local model along the edge as 
  \[ \begin{split}
              \Phi_d: \mathcal Z_d &\rightarrow (L_d \times L_d^{\perp}) \times \mathcal Z_d \\
                                q    & \mapsto ( (dF_d)_q, q).
     \end{split}
               \]
\end{defi}
The point of this definition is that we have the equality
\begin{equation} \label{overlap} 
 \Phi_e(\h_{e}^{-1}(\mathcal H_{[r_{e,d}, r''_{e,d}]})) = \Phi_d(\mathcal Z_{d} \cap ([r_{e,d}, r''_{e,d}] \times \tilde C_d) )
\end{equation}
and thus the local models over $\check e$ and over $\check d$ may be glued along this set. 

\subsection{Trimming the new local models over the edges}
The image of the new local model over the edge defined in Definition \ref{gluing_models} is still too big, since its image goes off to infinity. Before deciding where to trim, we have to prove that the new local model continues to have all the nice properties with respect to projections onto faces. 

Let us define by $\h_d$ the following composition
\[ \h_d: \quad \mathcal Z_d \stackrel{\Phi_d}{\longrightarrow} (L_d \times L_d^{\perp}) \times \mathcal Z_d \longrightarrow L_d^{\perp} \times \rho_d \]
where the rightmost map is just the standard projection. If $f$ is an edge of $d$, define
\[ \h_{d,f}= \x_f \circ \h_{d}. \]
Given the projection $\y_{d,f}$ onto the edge $\tilde C_{f}$ of $\tilde C_{d}$, well defined on $\tilde U_{d,f}$, we have:

\begin{lem} \label{glued_proj}
The following map is a diffeomorphism onto its image
\[ \begin{split}
  	\gb_{d,f}: \quad \quad \mathcal Z_d \cap (\rho_d  \times \tilde U_{d,f}) & \rightarrow V_{\check f} \times \tilde C_f \\
	                  (t, y) & \mapsto (\h_{d,f}(t,y), \y_{d,f}(y)).
	\end{split} \]
This implies that $\Phi_d(\mathcal Z_d \cap (\rho_d \times \tilde U_{d,f})$ is the graph of the differential of a Legendre transform of $F_d$. 
\end{lem}

\begin{proof} This is analogous to Proposition \ref{faceproj}. Clearly the Lemma holds when $\gb_{d,f}$ is restricted to $\rho'_d \times \tilde C_d$, where the local model coincides with the one in \S \ref{resc_edg}. So we restrict to $\mathcal Z_{e,d}$. Let us describe $\h_{d,f}$ in more detail. For every $t \in \rho_d$, define the slice
\begin{equation} \label{zt0}
 Z_t = \mathcal Z_d \cap (\{ t \} \times \tilde C_{d}).
\end{equation} 
and let $F_{d,t}: Z_t \rightarrow \R$ be the restriction of $F_d$ to the slice. Now let
\[ \h_{d,t}: Z_t \rightarrow L^{\perp}_{d} \]
be the differential of $F_{d,t}$ (recall that $L^{\perp}_{d}$ is the cotangent fibre of $T_d$) and let 
\[ \h_{d,f,t} = \x_{d,f} \circ \h_{d,t}. \]
It is easy to show that 
\begin{equation} \label{acca_d}
   \begin{split}
                    \h_d(y,t) &= (\h_{d,t}(y), t) \\
                     \h_{d,f}(t, y) &= (\h_{d,f,t}(y), t)
   \end{split} 
\end{equation}
and therefore that
\begin{equation} \label{g_edge_model}
      \gb_{d,f}(t, y) = (\h_{d,f,t}(y), t, \y_{d,f}(y)).
\end{equation}
Define 
\begin{equation} \label{g_edge_slice}
   \begin{split}
  	\gb_{d,f,t}: \quad  Z_t \cap (\{ t \} \times \tilde U_{d,f} )& \rightarrow \Gamma_{d,f} \times \tilde C_f \\
	                  y \ \ \ \ \  \ \ \ \ & \mapsto (\h_{d,f,t}(y), \y_{d,f}(y)).
	\end{split}
\end{equation}
In particular $\gb'_{d,f}$ is a diffeomorphism if and only if $\gb_{d,f,t}$ is a diffeomorphism for all $t$. Clearly there is nothing to prove when $\eta(t) = 0$ or $1$, since in this case $F_d$ coincides with $G_{e,d}$ or $F$ and the result follows from Proposition \ref{faceproj}. For other values of $\eta$, $F_{d,t}$ is an interpolation between $G_{e,d}$ and $F$.  Given local coordinates $y=(y_1, y_2)$ on $Z_t$, we know that the hessian (in the $y$ coordinates) of both functions is negative definite by construction, therefore also the hessian of $F_{d,t}$ must be negative definite. In particular also the hessian of $F_{d,t}$ restricted to a fibre of $\y_{d,f}$ is negative definite. It follows that $\gb_{d,f,t}$ is a local diffeomorphism and that $\h_{d,f,t}$ restricted to a fibre of $\y_{d,f}$ is injective (compare also with Proposition \ref{faceproj}). Hence $\gb_{d,f}$ is a diffeomorphism. The proof of the last statement follows as in Proposition \ref{faceproj}.
\end{proof}


Recall the definition of the reference points $p_{d,f}$ and $p'_{d,f}$ in $\Gamma_{d,f}$ given in \S \ref{inner_pol}. We have the following 

\begin{lem} \label{trim_new_const} The set $([p_{d,f}, p'_{d,f}) \times \rho_d) \times \tilde C_f$ is in the image of the map $\gb_{d,f}$ defined in Lemma \ref{glued_proj}
\end{lem} 

\begin{proof} Consider the description \eqref{g_edge_model} of $\gb_{d,f}$ and the map $\gb_{d,f,t}$ in \eqref{g_edge_slice}. Let $Z_t \subset \tilde C_d$ be as in \eqref{zt0} and let 
\[ x_t = (q, t)  \in [p_{d,f}, p'_{d,f}) \times \rho_d. \]
We have to show that $\{ x_t \} \times \tilde C_f$ is in the image of $\gb_{d,f,t}$.  When $t \in \rho'_d$, then  $Z_t = \{ t \} \times \tilde C_{d}$ and $\gb_{d,f,t}$ coincides with the map $\overline{\gb}_{d,f}$ from \S \ref{resc_edg}. Therefore the claim follows from \eqref{edge_ref_pt} and Corollary \ref{fibreBndl1}.  

Otherwise assume $t \in [ r_{e,d}, r'_{e,d})$. In this case $F_d$ interpolates $G_{e,d}$ and $F$. Let $G_{e,d,t}$ be the restriction of $G_{e,d}$ to $Z_t$ and denote by $\h_{d,t}^+$ and $\h_{d,f}^{-}$ the differentials (with respect to the $y$ coordinates) respectively of $G_{e,d,t}$ and $F$ and let $\h_{d,f,t}^{\pm} = \x_{d,f} \circ \h_{d,t}^{\pm}$. Then we have that 
\begin{equation} \label{diff_inter}
   \begin{split}
          \h_{d,t} & = \eta \h_{d,t}^{+} + (1-\eta) \h_{d,t}^{-} \\
          \h_{d,f, t} & = \eta \h_{d,f,t}^{+} + (1-\eta) \h_{d,f,t}^{-}
   \end{split}
\end{equation}
Observe that $\h_{d,t}^{-}$ coincides with the map $\overline{\h}_d$ from \S \ref{resc_edg}. From Lemma \ref{hdomains} we have
\begin{equation} \label{zt1}
      (\h_{d,t}^{-})^{-1}(H'_d) \subseteq  \tilde {\mathcal A}^{e,d}_{\epsilon'_2} \subseteq \y_{e,d}( \h_{e}^{-1}( I_{t}^{e,d}) ) \subseteq Z_t.
\end{equation}
Given the segment $I^{d,f}_{q} \subseteq H'_{d}$ (see \S \ref{resc_edg}) and the segment $I^{e,f}_{x_t} \subset I_{t}^{e,d}$ define the following curves in $Z_t$
\begin{equation} \label{curve_edge}
\gamma^-_{d,f,t} :=(\h_{d,t}^{-})^{-1}(I^{d,f}_{q}) \quad \text{and} \quad  \gamma^+_{d,f,t} := \y_{e,d}  (\h_e^{-1}( I^{e,f}_{x_t})).
\end{equation}
We have
\begin{equation} \label{position_curve1}
         \gamma^-_{d,f,t} \subset \tilde {\mathcal A}^{e,d}_{\epsilon'_2}. 
\end{equation} 
Moreover, since $I^{e,f}_{x_t} \subset \mathcal H_{r_{e,f}} \cap \mathcal H_{r_{e,d}}$, by \eqref{vert_resc_concl},  \eqref{trim_condition} and \eqref{trim_neigh} we have
\begin{equation} \label{position_curve2}
      \gamma^+_{d,f,t} \subset  \tilde{\mathcal W}^{d,f}_{\epsilon'_2}.
\end{equation} 
Notice that since $\gb_{e,f}$ maps the curve $\h^{-1}_e(I^{e,f}_{x_t})$ one to one onto $x_t \times \tilde C_f$, we have that $\y_{d,f}$ maps the curve $\gamma^+_{d,f,t}$ one to one onto $\tilde C_f$. Similarly $\y_{d,f}$ maps $\gamma^-_{d,f,t}$ one to one onto $\tilde C_f$. Moreover, by construction,
\begin{equation} \label{im_curves}
          \h_{d,f,t}^{+}(\gamma^{+}_{d,f,t}) = \h_{d,f,t}^{-}(\gamma^{-}_{d,f,t})= q \in [p_{d,f}, p'_{d,f}).
\end{equation}
Now fix a fibre $\y_{d,f}^{-1}(y')$ of $\y_{d,f}$. Let $y^+$ and $y^{-}$ be the unique points where this fibre intersects $\gamma^+_{d,f,t}$ and $\gamma^-_{d,f,t}$ respectively. Now recall that the hessians of $F_{d,t}$, $G_{e,d,t}$ and $F$ restricted to a fibre of $\y_{d,f}$ are all negative definite, in particular $\h_{d,f, t}$, $\h_{d,f,t}^{+}$ and $\h_{d,f,t}^{-}$ are all injective. It is then easy to see that \eqref{diff_inter} and \eqref{im_curves} together with inclusions \eqref{position_curve1} and \eqref{position_curve2} imply that there is a point $y$ on this fibre of $\y_{d,f}$, between $y^+$ and $y^-$, such that 
 \[ \h_{d,f,t}(y)= q. \]
Then $\gb_{d,f,t}(y) = (q,y')$. This concludes the proof.
\end{proof}

\begin{defi}[Provisional] \label{new_edge_trim} Given a two dimensional $d \in (P, \nu)$, let $\Phi_d$ be the map in Definition \ref{gluing_models}. Redefine the trimmed domain $\mathcal Z_d$ of $\Phi_d$ to be the open set of points $(t, y)$ such that $\Phi_d(t,y)$ is defined and for all edges $f$ of $d$ satisfying $y \in \tilde U_{d,f}$, we have 
\[ \h_{d,f,t}(t, y)  < p'_{d,f}. \] 
We denote this local model by $(\Phi_d, \mathcal Z_d)$. For all edges $f$ of $d$ also denote 
\begin{equation} \label{zprim_df_prov}
     \mathcal Z_{d,f} = \{ (t,y) \in \mathcal Z_{d} \, | \, y \in \tilde U_{d,f} \ \text{and} \  \h_{d,f}(y,t) \in [p_{d,f} , p'_{d,f}) \times \rho_d \}.
\end{equation}
\end{defi} 

It is clear from the construction that $\mathcal Z_d$ is homeomorphic to $\rho_d \times \tilde C_d$. It is also clear from Lemma \ref{trim_new_const} that $\gb_{d,f}$ gives a diffeomorphism from $\mathcal Z_{d,f}$ to $([p_{d,f} , p'_{d,f}) \times \rho_d) \times \tilde C_f$.  Moreover for every $t \in \rho_d$ the set
\[ Z_{d,t} = \mathcal Z_d \cap (\{ t \} \times \tilde C_d) \]
is homeomorphic to $\tilde C_d$. Also define
\[ Z_{d,f, t} = \mathcal Z_{d,f} \cap (\{ t \} \times \tilde C_d). \]

 The following lemma controls the size of the image of $\h_{d}$. 

\begin{lem} \label{gluing_size} If $H'_d$ is the set defined in \eqref{edge_H1}, we have 
\[ \h_{d}( \mathcal Z_{d} ) \subseteq H'_d \times \rho_d. \]
Moreover, given the set $B_{d,f} \subset L^{\perp}_d$ defined in \S \ref{intorni}, we have that, for all $t \in \rho_d$, $y \in  Z_{d,t}$ satisfies 
\[ \h_{d,t}(y) \in B_{d,f} \]
if and only if $y \in Z_{d,f, t}$. 
\end{lem} 

\begin{proof} For the first inclusion we have to show that for all $t \in \rho_d$
\[ \h_{d,t}(Z_{d,t}) \subseteq H'_d. \]
By construction, when $t \in \rho'_d$, then $\h_{d,t}$ coincides with the map $\overline{\h}_d$ from \S \ref{resc_edg}. Therefore Definition \ref{new_edge_trim} implies that $\h_{d,t}(Z_{d,t})$ coincides with $H'_d$.

Now let $t \in [r_{e,d}, r'_{e,d})$. We use the description \eqref{diff_inter} of $\h_{d,t}$. Inclusion \eqref{ieder_inH1} implies 
\[ \h^{+}_{d,t}(Z_{d,t}) \subset I^{e,d}_t \subset  K_d  \cap H'_d, \]
where the first inclusion follows from the description of $Z_{d,t}$ given in the proof of Lemma \ref{trim_new_const}. Moreover \eqref{hA_small} implies that 
 \[ \h^{+}_{d,t}(\mathcal A^{e,d}_{\epsilon_2'}) \subseteq K'_{d}. \]
Given $y \in Z_{d,t}$, assume 
\[ \h^{-}_{d,t}(y) \in H'_{d}. \]
Then \eqref{zt1} implies $y \in \mathcal A^{e,d}_{\epsilon_2'}$. Therefore $\h_{d,t}(y)$ is on the segment between $\h^{+}_{d,t}(y) \in K'_d$ and $\h^{-}_{d,t}(y)  \in H'_d$. Property $(c)$ of Proposition-Definition \ref{constrain_sets} implies that $\h_{d,t}(y) \in H'_{d}$.

On the other hand suppose $y \notin H'_{d}$, then for some edge $f$ of $d$
\begin{equation}  \label{hminus_in_h}
        \h^{-}_{d,t}(y) \in \mathcal H_{p'_{d,f}}.
\end{equation}
In particular  \eqref{edge_ref_pt} implies $y \in Z_{d,t} \cap \mathcal W^{d,f}_{\epsilon_3}$ and \eqref{eps3_cond} ensures that 
\begin{equation} \label{hplus_inV}
       \h^{+}_{d,t}(y) \in \mathcal V_{d,f},
\end{equation}
where the latter set is as in \eqref{h_nbhoods} for $J = J_{d,f}$. To prove this, recall that
\[ (t, \h^{+}_{d,t}(y)) = \h_e( \gb_{e,d}^{-1}(t,y)). \]
Let $y' = \gb_{e,d}^{-1}(t,y)$. Since $y \in \mathcal W^{d,f}_{\epsilon_3}$ and by \eqref{trim_condition}, $y' \in \y_{e,d}^{-1}( \mathcal W^{d,f}_{\epsilon_3}) \cap \mathcal W^{e,d}_{\epsilon_1}$. Therefore $y' \in \mathcal W^{e,f}$ and $\h_e(y') \in \mathcal V_{e,f}$ by Lemma \ref{hNbhd}. In particular this implies \eqref{hplus_inV}. Now, \eqref{hplus_inV} together with \eqref{ieder_inH1} implies 
\[  \h^{+}_{d,t}(y) \in K_{d,f} \cap H'_d.\]
The latter, together with \eqref{hminus_in_h} and property (d) of Proposition-Definition \ref{constrain_sets} implies $\h_{d,t}(y) \in H'_{d}$. This concludes the proof of first inclusion. 

Now suppose $y \in Z_{d,f, t}$. This implies $\h^{-}_{d,t}(y) \in \mathcal H_{p_{d,f}}$ and, by the above arguments, $\h^+_{d,t}(y) \in K_{d,f} \cap H'_d$. Using properties $(b)-(d)$ of Proposition-Definition \ref{constrain_sets} and the fact that $\h_{d,f,t}(y) \in [p_{d,f}, p'_{d,f})$, we must have
\[ \h_{d,t}(y) \in \mathcal H_{p_{d,f}} - \mathcal H_{p'_{d,f}} \subseteq B_{d,f}. \]
On the other hand suppose $ \h_{d,t}(y) \in B_{d,f}$. In particular
\[ \h_{d,f,t}(y) > p_{d,f}. \]
It is then enough to prove that $y \in \tilde U_{d,f}$. By the first part of the Lemma and by \eqref{ends_in_nbhd}, we must have $\h_{d,t}(y) \in \mathcal H_{p_{d,f}}- \mathcal H_{p'_{d,f}}$. Then we cannot have $y \in (\h^{-}_{d,t})^{-1}(H_d)$, since if this were true, the same arguments as above would imply $\h_{d,t}(y) \in H_d$, which contradicts $\h_{d,t} \in \mathcal H_{p_{d,f}}$.  On the other hand we cannot have $y \in \mathcal W^{d,f'}_{\epsilon_3}$ for some $f' \neq f$, since this would imply $\h_{d,t}(y) \in \mathcal V_{d,f'}$, while \eqref{bidf_vdf} implies $\h_{d,t} \in \mathcal V_{d,f}$. Therefore we must have $y \in \mathcal W^{d,f}_{\epsilon_3}$. In particular $y \in \tilde U_{d,f}$. \end{proof}

\subsection{The local models over faces}
Given an edge $f \in (P, \nu)$, the goal of this subsection is to define a Lagrangian embedding $\Phi_f: \rho_f \times \tilde C_f  \rightarrow M_{\R} \times N_{\R}/N$ which matches with the previous local models on overlaps. 

We need to trim further the local models over vertices by replacing \eqref{vert_first_trim_2} with
\begin{equation} \label{redfin_h2e}
       \mathcal H^{[2]}_e =  \mathcal H^{[1]}_{e}- \bigcup_{d \preceq e, \, \dim d=2} \mathcal H_{r''_{e,d}}.
\end{equation}
Then $\mathcal Z_e$ is as in Definition \ref{defi_loc_mod_vert}. Define, for a three dimensional $e$ and an edge $f$ of $e$ the sets
\[ \bar Y_e =  \mathcal H^{[2]}_e \cap \Gamma_e \quad \text{and} \quad \bar Y_{e,f} = \bar Y_e \cap \rho_f. \]
Notice that $\bar Y_e$ is obtained from $Y'_e$ by removing its intersections with the sets $\rho''_d \times \Gamma^{[1]}_d$ for all two dimensional faces $d$ of $e$. 

Let us now collect some data on $\rho_f \times \tilde C_f$ induced by local models over edges and vertices contained in $\check f$. For every three dimensional $e$ containing $f$, define the following subset of $\mathcal Z_e$:
\[ \mathcal Z_{e,f} = \{ y \in \mathcal Z_e \cap \tilde U_{e,f} \, | \, \h_{e,f}(y) \in \bar Y_{e,f}  \}. \]
We have that by construction and by Corollary \ref{fibreBndl1}, 
\[ \gb_{e,f}: \mathcal Z_{e,f} \rightarrow \bar Y_{e,f} \times \tilde C_f \]
is a diffeomorphism and $\Phi_e(\mathcal Z_{e,f})$ is the graph of the differential of a function $G$ defined on $\bar Y_{e,f} \times \tilde C_f$. Let us rename this function by $G_{e,f}$. 

Similarly, for every two dimensional $d$ containing $f$, in \eqref{zprim_df_prov} we defined the subset $\mathcal Z_{d,f}$ of $\mathcal Z_d$. Then by Lemmas \ref{glued_proj} and \ref{trim_new_const}, 
\[ \gb_{d,f}: \mathcal Z_{d,f} \rightarrow ((p_{d,f}, p'_{d,f}) \times \rho_d) \times \tilde C_f \]
is a diffeomorphism and $\Phi_d(\mathcal Z_{d,f})$ is the graph of the differential of a function $G$ defined on $ ((p_{d,f}, p'_{d,f}) \times \rho_d) \times \tilde C_f$. Rename this function by $G_{d,f}$. 

Notice that when $d$ is a face of $e$, the domains of definition of the two functions $G_{e,f}$ and $G_{d,f}$ overlap, but we have the following

\begin{lem} When $d$ is a face of $e$ the two functions $G_{e,f}$ and $G_{d,f}$ coincide on the overlap $(\bar Y_{e,f}  \cap ((p_{d,f}, p'_{d,f}) \times \rho_d)) \times \tilde C_f$.
\end{lem}

\begin{proof}
This is just a consequence of the fact that the local models over $\check e$ and $\check d$ coincide on the overlaps (as in \eqref{overlap}) and the two functions are defined via a Legendre transform. 
\end{proof}

As a consequence, if we consider all the functions $G_{e,f}$, where $e$ varies among two and three dimensional faces containing $f$, then these patch together to give a unique smooth function
\[ G_f : (\rho_f - \rho'_f) \times \tilde C_f \rightarrow \R. \]
We now wish to extend $G_f$ to the whole of $\rho_f \times \tilde C_f$ using a partition of unity interpolating $G_f$ with the zero function. Let $\rho''_f$ be the third inner polyhedron satisfying \eqref{inner_pols_incl} and consider a smooth function $\eta: \rho_f \rightarrow \R$ such that $0 \leq \rho \leq 1$ and
\[ \eta(x) = \begin{cases} 
                         1 \quad \text{if} \ x \in \rho_f - \rho''_f \\
                          0 \quad \text{on neighborhood of}  \  \rho'_f
                   \end{cases}
\]
Define $F_f: \rho_f \times \tilde C_f \rightarrow \R$ by 
\[ F_f(x,y) = \begin{cases} 
                         \eta(x) G_f(x,y) \quad \text{if} \ x \in \rho_f - \rho'_f  \\
                          0 \quad \text{if} \ x \in \rho'_f.
          \end{cases}
\]

\begin{defi} \label{loc_mod_face}
Given an edge $f \in (P, \nu)$, let 
\[ \mathcal Z_f = \rho_f \times \tilde C_f. \] 
The local model over $\check f$ is the map
\[ \begin{split}
              \Phi_{f}: \mathcal Z_f & \rightarrow (L_f \times L^{\perp}_f) \times \mathcal Z_f \\
                                                       (x,y) & \mapsto      ( dF_f , \, (x, \, y) ).
    \end{split} \]
We also denote by $\h_f:  \mathcal Z_f \rightarrow L^{\perp}_f \times \rho_f$ the right composition of $\Phi_f$ with the projection onto  $L^{\perp}_f \times \rho_f$.
\end{defi}

\begin{lem} \label{bound_loc_mod_face} Given an edge $f \in (P, \nu)$ and the local model in Definition \ref{loc_mod_face} we have
\[ \h_f(\mathcal Z_f) \subseteq \rho_f \times B_f \]
\end{lem}

\begin{proof} We have that the differential of $F_{f}$ decomposes as the sum $dF_{f} = d_xF_{f} + d_yF_{f}$, i.e. as the sum of the differentials with respect to the $x$ and $y$ coordinates respectively. By the identification of $L_{f}$ and $L^{\perp}_{f}$ with the cotangent fibres of $\rho_f$ and $T_f$ respectively, we have that $d_xF_{f} \in L_f$ and $d_yF_{f} \in L^{\perp}_f$. Then
\[ \h_f (x,y) = ( d_yF_f , x). \]
When $(x,y) \in \rho'_f \times \tilde C_f$, then $F_f =0$, therefore $\h_f(x,y) \in \{ 0 \} \times \rho'_f \subseteq \rho_f \times B_f$. Otherwise, when $x \in \rho_f - \rho'_f$, then 
\[ \h_f (x,y) = (\eta d_yG_f, x). \]
Let 
\[ \h^+_f(x,y) = ( d_yG_f, x). \]
If $x \in \bar Y_{e,f}$ for some three dimensional $e$ containing $f$, then $G_f = G_{e,f}$ and by construction 
\[ \h^+_f(x,y) = \h_e( \gb_{e,f}^{-1}(x,y)) \]
i.e. $\h^+_f(x,y) \in \mathcal H_{r_{e,f}}$. Therefore, by \eqref{schiacciato}, $\h^+_f(x,y) \in B_f \times \rho_f$. In particular, since $\eta(x) \in [0,1]$, also $\h_f(x,y) \in B_f \times \rho_f$.

Similarly, if $(x,y) \in ((p_{d,f}, p'_{d,f}) \times \rho_d)) \times \tilde C_f$ for some two dimensional $d$ containing $f$, then $G_f = G_{d,f}$ and 
\[ \h^+_f(x,y) = \h_d( (\gb_{d,f})^{-1}(x,y)). \]
Therefore $\h^+_f(x,y) \in B_f \times \rho_f$ by Lemma \ref{gluing_size}. In particular also $\h_f(x,y) \in B_f \times \rho_f$.
\end{proof}
\subsection{The last step}
We now wish to glue all the pieces together to form the smooth Lagrangian submanifold $\mathcal L$ lifting $\Xi$. First we need to trim further the local models over vertices and edges. 

\begin{defi}[Final] Given a three dimensional $e \in (P, \nu)$, consider the local model at $\check e$ given by \eqref{local_vert}. Redefine the sets $\mathcal H^{[1]}_e$ and $\mathcal H^{[2]}_e$ as 
\begin{equation} \label{final_def_h2e}
   \begin{split}
       \mathcal H^{[1]}_e &=  \mathcal H_{e}- \bigcup_{f \preceq e, \, \dim f=1} \mathcal H_{r''_{e,f}}, \\
        \mathcal H^{[2]}_e&=  \mathcal H^{[1]}_{e}- \bigcup_{d \preceq e, \, \dim d=2} \mathcal H_{r''_{e,d}}.
   \end{split}
\end{equation}
Then restrict the domain of $\Phi_e$ to $\mathcal Z_e = \h^{-1}_e(\mathcal H^{[2]}_e)$. Notice that we have 
\[ \mathcal H^{[2]}_e \cap \Gamma_e = Y''_e, \]
where the latter set was defined in \S \ref{inner_pol}. For every edge of $e$ we also redefine the subsets $\mathcal Z_{e,f}$ of $\mathcal Z_e$ as
\[ \mathcal Z_{e,f} = \{ y \in \mathcal Z_e \cap \tilde U_{e,f} \, | \, \h_{e,f}(y) \in Y''_{e} \cap \rho_f   \}. \]
Notice that 
\begin{equation} \label{zef_hef}
 \mathcal Z_{e,f} = \h_{e}^{-1}( \mathcal H^{[2]}_e \cap \mathcal H_{r_{e,f}} )
\end{equation}
\end{defi}

Similarly we trim the local models along the edges. 

\begin{defi}[Final] \label{edge_trim_final} Given a two dimensional $d \in (P, \nu)$, let $\Phi_d$ be the map in Definition \ref{gluing_models}. We redefine the trimmed domain $\mathcal Z_d$ of $\Phi_d$ to be the open set of points $(t, y)$ such that $\Phi_d(t,y)$ is defined and for all edges $f$ of $d$ satisfying $y \in \tilde U_{d,f}$, we have 
\[ \h_{d,f,t}(t, y)  < p''_{d,f}. \] 
We denote this local model by $(\Phi_d, \mathcal Z_d)$. For all edges $f$ of $d$ also denote 
\begin{equation} \label{zprim_df}
     \mathcal Z_{d,f} = \{ (t,y) \in \mathcal Z_{d} \, | \, y \in \tilde U_{d,f} \ \text{and} \  \h_{d,f}(t,y) \in [p_{d,f} , p''_{d,f}) \times \rho_d \}.
\end{equation}
\end{defi}

Notice that by construction we have the following overlaps. Given a three dimensional $e$, for every two dimensional face $d$ of $e$ we have
\begin{equation} \label{overlap_vert_edge} 
 \Phi_e(\h_{e}^{-1}(  \mathcal H_{[r_{e,d}, r''_{e,d})})) = \Phi_d(\mathcal Z_{d} \cap ([r_{e,d}, r''_{e,d}) \times \tilde C_d) ),
\end{equation}
while for every edge $f$ of $e$
\begin{equation} \label{overlap_vert_face} 
 \Phi_e(\mathcal Z_{e,f}) = \Phi_f((Y''_{e} \cap \rho_f) \times \tilde C_f ).
\end{equation}
Given a two dimensional $d$ and an edge $f$ of $d$ we have
\begin{equation} \label{overlap_edge_face} 
 \Phi_d(\mathcal Z_{d,f}) = \Phi_f(([p_{d,f}, p''_{d,f}) \times \rho_d) \times \tilde C_f ).
\end{equation}

Let us now glue all the pieces together. 

\begin{defi} 
A Lagrangian smooth lift of $\Xi$ is  defined to be the following subset of $M_{\R} \times N_{\R}/N$ 
\begin{equation} \label{L}
 \mathcal L = \bigcup_{1 \leq \dim e \leq 3} \Phi_e (\mathcal Z_e) 
\end{equation}
\end{defi}

Finally we can prove the following.

\begin{thm} \ $\mathcal L$ is a closed Lagrangian submanifold of $M_{\R} \times N_{\R}/N$ homeomorphic to $\hat \Xi$. 
\end{thm}

\begin{proof} Since local models are graphs, the subsets $\Phi_e(\mathcal Z_e)$ are Lagrangian submanifolds, for all $e$. It is enough to prove that \eqref{overlap_vert_edge}--\eqref{overlap_edge_face} are the only possible intersections between local models. 

Inclusions \eqref{h2_in_nbhd} and \eqref{body_in_nbhd}, Lemmas  \ref{gluing_size}  and \ref{bound_loc_mod_face} and conditions $(1)$ and $(2)$ of \S \ref{intorni} imply that given two distinct simplices $\check e_1$ and $\check e_2$ of the same dimension then 
\[ \Phi_{e_1}(\mathcal Z_{e_1}) \cap \Phi_{e_2}(\mathcal Z_{e_2}) = \emptyset. \]

Let $\check e$ and $\check d$ be such that $\dim \check e < \dim \check d$.  If $\check e$ is not a face of $\check d$, then  inclusions \eqref{h2_in_nbhd} and \eqref{body_in_nbhd},  Lemmas \ref{gluing_size}  and \ref{bound_loc_mod_face} and conditions (3) and (4) of \S \ref{intorni} imply that 
\[ \Phi_e(\mathcal Z_e) \cap \Phi_{d}(\mathcal Z_{d}) = \emptyset. \]

Suppose now that $\check e$ is a vertex of an edge $\check d$, then Lemma \ref{gluing_size}, inclusion \eqref{body_in_nbhd} and \eqref{Bd_inters_H} imply that 
\[ \h_d(\mathcal Z_d) \cap \h_e(\mathcal Z_e) \subseteq \mathcal H_{[r_{e,d}, r''_{e,d})}.\]
Thus \eqref{overlap_vert_edge} implies that the latter inclusion is an equality and that 
\[ \Phi_d(\mathcal Z_d) \cap \Phi_e(\mathcal Z_e) =  \Phi_e(\h_{e}^{-1}( \mathcal H_{[r_{e,d}, r''_{e,d})})) = \Phi_d(\mathcal Z_d \cap ([r_{e,d}, r''_{e,d}) \times \tilde C_d)). \]
Similar arguments, using Lemmas \ref{bound_loc_mod_face} and \ref{gluing_size}, show in the remaining cases that, whenever $\check e \preceq \check d$,  then $\Phi(\mathcal Z_e)$ and $\Phi (\mathcal Z_d)$ intersect as expected. 
This concludes the proof of the fact that $\mathcal L$ is a submanifold. The closure of $\mathcal L$ is a consequence of the construction. 

Let us prove that $\mathcal L$ is homeomorphic to $\hat \Xi$. Let us first describe a decomposition of $\hat \Xi$. Given an $e \in (P, \nu)$, of dimension $2$ or $3$, consider the subset $Y''_e \subset \Xi$ as defined in \S \ref{inner_pol} and let $\hat Y''_e \subset \hat \Xi$ be its PL-lift. Then 
\[   \hat \Xi = \left( \bigcup_{\dim e = 2,3} \hat Y''_e \right) \cup \left( \bigcup_{\dim f = 1} (\rho''_f \times \tilde C_f) \right).
\]
On the other hand we also have the following decomposition of $\mathcal L$
\[
  \mathcal L = \left( \bigcup_{\dim e = 3} \Phi_e \left( \overline{\mathcal Z}_e \right) \right) \cup \left( \bigcup_{\dim d = 1,2} \Phi_{d} \left( \overline{\mathcal Z}_d \cap  (\rho''_d \times \tilde C_d) \right) \right)  ,
\]
where $\overline{\mathcal Z}_e$ and $\overline{\mathcal Z}_d$ denote the closures of those sets inside $\tilde C_e$ and $\rho_d \times \tilde C_d$ respectively. We have that
\[ \overline{\mathcal Z}_e = \h_e^{-1} \left( \overline{ \mathcal H}^{[2]}_e \right) \]
and by construction and by Proposition \ref{pairpants_PLft} we have the homeomorphism 
\[ \Phi_e \left( \overline{\mathcal Z}_e \right) \cong \hat Y''_e.  \]
Similarily
\[ \Phi'_{d} \left( \overline{\mathcal Z}'_d \cap  (\rho''_d \times \tilde C_d) \right) \cong \hat Y''_d\]
when $d$ has dimension $1$ or $2$. It is also clear that one can arrange these homeomorphisms to match on the intersections. 

To construct a family $\mathcal L_t$ which converges to $\hat \Xi$ in the Hausdorff topology one can uniformly scale the local models by some parameter $t$ and then glue everything together as above. 
\end{proof}

\section{On more general examples and applications}
In \cite{lag_pants} we gave various generalizations and examples in the case of Lagrangian lifts of tropical curves. We expect that similar generalizations and examples extend to the case of tropical surfaces, although with some additional subtleties. We briefly comment here these ideas, referring the reader either to \cite{lag_pants} when the details are a straight forward generalization or to future work in the more delicate cases. We will use the same notations as in Section \ref{main_results}.

\subsection{Different lifts of the same tropical hypersurface} 
As we did for curves in \S 5.1 of \cite{lag_pants}, we can twist the Lagrangian lift of a tropical hypersurface by local sections. Let $\check f$ be a polyhedron of $\Xi$ of dimension $k=1, \ldots, n$ and let $C_f \subset N^f_{\R}/N^f$ be the standard coamoeba associated to $f$. Given a smooth section 
\begin{equation} \label{section_edge}
        \sigma_f : \check f \rightarrow \check f \times T
\end{equation}
Define
\begin{equation} \label{general_lift_edge}
       \hat f_{\sigma_f} = C_f \cdot \sigma_f,
\end{equation}
where the righthand side means that for every $x \in \check f$, we consider the set $C_f \cdot \sigma_f(x)$ as a subset of the orbit of $\sigma_f(x)$ under the action of $N^f_{\R}/N^f$ on $T$. Given the quotient
\begin{equation}  \label{quotient_edge}
     \alpha: \check f \times T \rightarrow \check f \times \frac{T}{N^f_{\R}}
\end{equation}
then the righthand side is naturally a symplectic manifold. We have that $\hat f_{\sigma_f}$ is Lagrangian (at its smooth points) if and only if $\alpha \circ \sigma_f$ is a Lagrangian section of the quotient. So we must impose this condition. Now define the twisted PL-lift to be
\[ \hat \Xi_{\sigma} = \left( \bigcup_{\dim e=n+1} \check{e} \times C_e \right) \cup \left( \bigcup_{1 \leq \dim f \leq n} \hat f_{\sigma_f} \right). \] 
In order for this to be a topological manifold we must impose suitable boundary conditions on the sections $\sigma_f$, so that everything matches nicely. The smoothing $\mathcal L_{\sigma}$ of $\Xi_{\sigma}$ can be done by suitably adapting the proof of Section \ref{main_results}. 

\begin{rem} We expect that such lifts should be classified by a sheaf of multivalued piecewise linear integral functions, in the spirit of the Gross-Siebert program \cite{G-Siebert2003}. Some examples of Lagrangian spheres constructed from piecewise linear integral functions were given in \cite{onHmstoric}, where the underlying tropical surface was just a disk. Moreover, we also expect that the difference $\mathcal L - \mathcal L_{\sigma}$ should be, in some sense, related to Lagrangian lifts of lower dimensional tropical varieties. For instance, suppose the lift $\mathcal L_{\sigma}$ is constructed from a  piecewise linear integral functions $\sigma$, then the difference should be related to the tropical subvariety given by the non-smooth locus of $\sigma$. For the relevance of the different lifts of the same tropical variety in homological mirror symmetry see Section 6.3 of \cite{diri_branes} and \cite{onHmstoric}.
\end{rem}
\subsection{Non smooth tropical hypersurfaces}
We expect to be able to lift also non-smooth tropical hypersurfaces, namely those given by not necessarily unimodal subdivisions of $P$. An easy case is when $P \subset N_{\R}$ is an integral $n+1$-dimensional simplex (not elementary), with no subdivision. Indeed let $N' \subset N$ be the smallest sublattice in which $P$ is an elementary integral simplex and let $M' \subset M_{\R}$ be its dual.  Then the associated tropical subvariety $\Xi \subset M_{\R}$ is a standard tropical hyperplane as a tropical subvariety of $M'_{\R}$. Denote the torus
\[ T' = \frac{N_{\R}}{N'}. \]
Inside $T'$ we have the standard Lagrangian coamoeba $C'$ associated to $P$ and $\Xi$. The action of $N'$ on $T$ defines a covering map
\[ \beta: T \rightarrow T'.\]
Then we can define 
\[ C= \beta^{-1}(C') \]
Given the function $F': C' \rightarrow \R$ defined in  \eqref{Fglob}, we let 
\[ F= F' \circ \beta \] 
on $C_e$. We define the Lagrangian lift of $\Xi$ to be the graph of the differential of $F$ extended to the real blow up of $C_e$ at its vertices. 

\begin{ex} \label{singular_ex} 
An interesting case is when $N=\Z^{n+1}$, $\{u_1, \ldots, u_{n+1} \}$ is the standard basis, $u_0$ is defined as in \eqref{u0} and
\[ P= \conv \{ u_0, \ldots, u_{n+1} \}. \]
Then $\beta: T \rightarrow T'$ is a covering of degree $n+2$.  The associated tropical hypersurface $\Xi$ is the fan whose rays are generated by the vectors 
\[ \xi_0 = u_0, \quad \text{and} \quad \xi_j = u_0 + (n+2) u_j \]
and the maximal cones are those spanned by all collections of $n$ rays.  
\end{ex}

\subsection{Lagrangian submanifolds in toric varieties}
We wish to generalize to higher dimensions the examples given in Section 6 of \cite{lag_pants} of Lagrangian submanifolds inside a toric variety which lift tropical curves in the moment polytope. We have not yet worked out all the details, since the construction is not as straight forward as in the case of curves, therefore we will only sketch some examples and point out where the difficulties are. Let $\dim M_{\R} = 3$ and let $\Delta \subset M_{\R}$ be a Delzant polyhedron. Denote by $\partial \Delta$ its boundary and by $\Delta^{\circ}$ its interior. Let $X_{\Delta}$ be the associated toric variety, recall that $\Delta^{\circ} \times T \subset X_{\Delta}$.

Given a tropical hypersurface $\Xi^{\infty} \subset M_{\R}$ and $\mathcal L^{\infty}$ a Lagrangian lift of $\Xi^{\infty}$. Define
\[ \Xi = \Delta \cap \Xi^{\infty}. \]
Then the lift $\mathcal L$ of $\Xi$ inside $X_{\Delta}$ is formed by taking the closure of $\mathcal L^{\infty} \cap (\Delta^{\circ} \times T)$ inside $X_{\Delta}$.  The question is: how nice is $\mathcal L$? When is it a smooth submanifold, with or without boundary? In the case of curves and given certain conditions on how $\Xi^{\infty}$ intersects $\partial \Delta$, it turns out that $\mathcal L$ is automatically a smooth manifold with boundary or, in some nicer cases, a smooth manifold without boundary. Some times $\mathcal L$ is a non-orientable surface (see \cite{mikh_trop_to_lag} or \S 6.2 of \cite{lag_pants}). 

It the case of tropical surfaces, it is not hard to find conditions such that $\mathcal L$ is a smooth manifold with boundary and corners, but it is not obvious how to obtain smooth manifolds without boundary. The problem is understanding the interaction of $\mathcal L$ with the toric boundary of $X_{\Delta}$. 

\begin{ex}\label{wave_front_trop}
This example generalizes Examples 6.2 and 6.3 of \cite{lag_pants} and Mikhalkin's tropical wave fronts (Example 3.3 of \cite{mikh_trop_to_lag}). The polyhedron $\Delta$ is given by an intersection of half spaces
\[ \Delta = \bigcap_{\delta} \left \{ \inn{d_{\delta}}{x} \geq t_{\delta} \right \} \]
where the boundary of each half space contains a two dimensional face $\delta$ of $\Delta$ such that $d_{\delta}$ is its inward integral primitive normal direction. Consider the smaller polyhedron  inside $\Delta$ given by
\[ \Delta_{\epsilon} = \bigcap_{\delta} \left \{ \inn{d_{\delta}}{x} \geq t_{\delta}+ \epsilon \right \} \]
for some small $\epsilon$. For each edge $\tau$ of $\Delta$, let $\tau_{\epsilon}$ be the corresponding edge of $\Delta_{\epsilon}$. Consider the two dimensional polyhedron
 \[ \ell_{\tau} = \conv (\tau \cup \tau_{\epsilon}). \]
Define the tropical surface
\[ \Xi = \partial \Delta_{\epsilon} \cup \left(  \bigcup_{\tau} \ell_{\tau} \right), \]
where the union runs over all edges of $\Delta$.  See Figure \ref{wave_front} for a picture of $\Xi$ in the case $\Delta$ is a standard simplex. It can be easily seen that since $\Delta$ is Delzant, $\Xi$ is smooth and its boundary coincides with the union of the edges of $\Delta$. Each $\ell_{\tau}$ has the following property. Given its tangent space $M^{\ell_{\tau}}_{\R}$, choose a basis $\{v_1, v_2 \}$ of the lattice $M^{\ell_{\tau}}$ such that $v_1$ is tangent to $\tau$. Then we have that for each two dimensional face $\delta$ containing $\tau$ 
\[ \inn{d_{\delta}}{v_2} = 1. \]
 This is analogous to what we called property (P) in \S 6.1 of \cite{lag_pants} or in Mikhalkin's terminology $\ell_{\tau}$ is bisectrice (see Definition 1.12 of \cite{mikh_trop_to_lag}). In particular each vertex of $\Delta$ is the endpoint of an edge of $\Xi$, all of whose adjacent two dimensional polyhedra are bisectrices.  We ask whether one can construct a smooth Lagrangian lift $\mathcal L$. We believe this is true but we do not have a complete proof yet. The bisectrice property of the polyhedra $\ell_{\tau}$ make it possible to construct a lift which is smooth over interior points of the edges $\tau$. The difficulty lies in proving that the lift can be smoothed also over the vertices of $\Delta$. As suggested by Mikhalkin, it would be interesting to follow the dynamics of $\Delta_{\epsilon}$  beyond small values of $\epsilon$, such as described by Kalinin and Shkolnikov in \cite{kalinin_wave_dyn}. Can this dynamic be translated in a smooth family of Lagrangians?
\begin{figure}[!ht] 
\begin{center}
\includegraphics{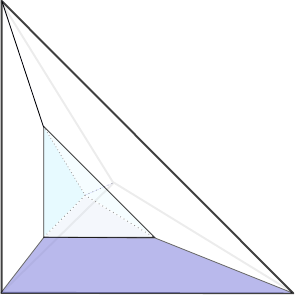}
\caption{The tropical hypersurface $\Xi$. For clarity, only one of the polyhedra $\ell_{\tau}$ is colored.} \label{wave_front}
\end{center}
\end{figure}
\end{ex}

\begin{ex} \label{monotone_ex}
As a limit case of the above example, let $\Delta$ be the polytope of $\PP^3$, i.e. the standard simplex in $\R^3$, and let $q \in \Delta$ be its barycenter. For every edge $\tau$ of $\Delta$ let 
  \[ \ell_{\tau} = \conv ( \tau \cup q ) \]
and define
\[ \Xi = \bigcup_{\tau} \ell_{\tau}. \]
Then $\Xi$ is a tropical hypersurface which, in a neighborhood of the vertex $q$, is as in Example \ref{singular_ex}. Therefore we can use the lift constructed there to find the lift $\mathcal L$ of $\Xi$ inside $\PP^3$. As in the previous example we have not yet proven that one can smooth the lift over the vertices of $\Delta$. We expect $\mathcal L$ to be homeomorphic to $S^1 \times S^2$. This example generalizes the monotone Example 6.5 of \cite{lag_pants}, so it should also be monotone.
\end{ex}

\begin{ex} This example in $\C^3$ generalizes Mikhalkin's examples \cite{mikh_trop_to_lag} of tropical curves representing non-orientable Lagrangian surfaces in $\C^2$. Let
\[ \Delta = (\R_{\geq 0})^3, \]
then $X_{\Delta} = \C^3$. Consider the points
\[ \begin{split}
            & Q_0=(0,0,0), \quad  Q_1 = (3,3,3),  \\
            & P_{1,1}=(4,3,3),\quad P_{2,1} = (7,2,2), \quad P_{3,1} =(12,0,0). 
   \end{split} \]
Let $P_{j,k}$ be the point obtained from $P_{j,1}$ by exchanging the first and the $k$-th coordinate. Define three dimensional polytopes
\[ \begin{split}
          & \Sigma_1= \conv \{Q_1, P_{1,1}, P_{1,2}, P_{1,3} \}, \\
          & \Sigma_2 =  \conv \{P_{1,1}, P_{1,2}, P_{1,3}, P_{2,1}, P_{2,2}, P_{2,3}. \} 
   \end{split}
\]
Clearly $\Sigma_1$ is a standard simplex and $\Sigma_2$ is a truncated simplex.
Define the two dimensional polytopes
\[ \begin{split}
          & \beta_1 = \conv \{ Q_0, \, Q_1, \, P_{1,1}, \, P_{2,1}, \, P_{3,1}  \}  \\
          & \gamma_1 = \conv \{  P_{2,2}, \, P_{3,2}, \, P_{2,3}, \, P_{3,3} \}. 
    \end{split}       
\]
Let $\beta_k$ and $\gamma_k$ be obtained from $\beta_1$ and $\gamma_1$ by the symmetry exchanging the first and $k$-th coordinate. Now let 
\[ \Xi = \partial \Sigma_1 \cup \partial \Sigma_2 \cup \left( \bigcup_{k} \beta_k\right) \cup \left( \bigcup_{k} \gamma_k\right). \]
It can be checked that this is a smooth tropical hypersurface in $\Delta$, see Figure \ref{non_orientable_exmpl}.

\begin{figure}[!ht] 
\begin{center}
\includegraphics{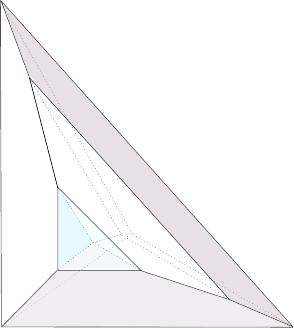}
\caption{We have only colored $\partial \Sigma_1$, $\beta_1$ and $\gamma_2$} \label{non_orientable_exmpl}
\end{center}
\end{figure}

The two dimensional polyhedra which hit the boundary of $\Delta$ are the $\beta_k$'s and $\gamma_k$'s. We have that $\beta_k$ has one edge lying on a coordinate axis of $\Delta$ and it is a bisectrice (see previous example). The $\gamma_k$'s have an edge lying on a coordinate plane $\delta$ of $\Delta$. They have the property that if $\{ v_1, v_2 \}$ is a basis of the lattice $M^{\gamma_k}$ such that $v_1$ is tangent to $\delta$, then  
\begin{equation} \label{even}
       \inn{d_{\delta}}{v_2} = 2,
\end{equation}
where $d_{\delta}$ is the inward, primitive integral normal direction of $\delta$. This is analogous to the condition satisfied by the edges of tropical curves representing non-orientable surfaces (see \S 3.4 of \cite{mikh_trop_to_lag}).  

Therefore, it seems reasonable to expect that such a tropical hypersurface admits a smooth (non-orientable) Lagrangian lift $\mathcal L$ in $\C^3$. Indeed the above properties guarantee that $\mathcal L$ can be constructed so that it is smooth everywhere except over the points $Q_0$, $P_{3,1}$, $P_{3,2}$, $P_{3,3}$ which are the points where an edge of $\Xi$ hits the boundary of $\Delta$. While the point $Q_0$ is of the type already present in Example \ref{wave_front_trop} (i.e. the vertices of $\Delta$), the points $P_{3,k}$ have a different nature. They are the end points of an edge of $\Xi$ which is adjacent to a bisectrice (i.e. $\beta_k$) and two polyhedra satisfying \eqref{even} (i.e. two of the $\gamma_j$'s). 
\end{ex}

\subsection{Lagrangian submanifolds of Calabi-Yau manifolds} \label{CY's}
An interesting generalization of the above constructions would be to find Lagrangian submanifolds inside the symplectic Calabi-Yau manifolds with a Lagrangian torus fibration constructed in \cite{CB-M}, based on Gross's topological torus fibrations \cite{TMS}. Indeed, given a symplectic manifold $(X, \omega)$ with a Lagrangian torus fibration $f: X \rightarrow B$, let $B_0$ be the locus in $B$ of smooth fibres and let $D = B - B_0$ be the discriminant locus. Action coordinates on $B$ define an integral affine structure on $B_0$, i.e. an atlas with change of coordinate maps inside $\Gl_{n}(\Z) \ltimes \R^n$. Therefore $B_0$ is a natural ambient space where tropical subvarieties can be defined. If we also have a Lagrangian section $\sigma: B \rightarrow X$, then the Arnold-Liouville theorem tells us that $X_0 = f^{-1}(B_0)$ is symplectomorphic to $T^*B_0 / \Lambda$, where $\Lambda$ is a lattice of maximal rank in $T^*B_0$. Therefore, locally $X_0$ is like $M_{\R} \times N_{\R}/N$. Hence we can define the Lagrangian PL lift $\hat \Xi$ of a tropical hypersurface $\Xi$ in $B_0$. If $\dim_{\R} B_0 = 3$ then we can also find a smoothing $\mathcal L_0 \subset X_0$ of $\hat \Xi$. Suppose now that $\Xi$ is a tropical hypersurface which has boundary on the discriminant locus $D$. What is the closure $\mathcal L$ of $\mathcal L_0$? When is it a smooth manifold, without boundary? The Lagrangian $3$-torus fibrations  constructed in \cite{CB-M} have prescribed singular fibres modeled on those described \cite{TMS}. Indeed $D$ is a (thickening of a) $3$-valent graph, with two types of singular fibres over the vertices: positive and negative. We believe that it should not be hard to understand when the closure $\mathcal L$ of $\mathcal L_0$ is smooth. Indeed the examples  in  \cite{onHmstoric} of Lagrangian spheres were constructed using this idea.  The following examples are inside a symplectic Calabi-Yau homeomorphic to the quintic threefold in $\PP^4$. 

\begin{ex} \label{sphere_quintic} In \cite{TMS} and \cite{GHJ}, Gross describes a $3$-valent graph $D$ inside a $3$-sphere $B$ and an integral affine structure on $B_0 = B- D$ such that one can compactify $X_0 = T^*B_0 / \Lambda$ to a topological manifold $X$ by adding canonical singular fibres over $D$. Gross proves that $X$ is homeomorphic to a smooth quintic threefold in $\PP^4$. In \cite{CB-M} it is shown that one  can find a symplectic form on $X$ (extending the natural one on $X_0$) so that the fibration is Lagrangian. The $3$-sphere $B$ is identified with the boundary $\partial \Delta$ of the standard simplex in $\R^3$. Let $\Delta^{[2]}$ be the two skeleton of $\partial \Delta$, i.e. the union of two dimensional faces. Then $D \subset \Delta^{[2]}$ and $D$ divides $\Delta^{[2]}$ in $105$ connected components. Each of these components is a smooth tropical hypersurface with boundary on $D$. The components are divided into three different types which are pictured in Figure \ref{vertex_quintic}. 
\begin{figure}[!ht] 
\begin{center}
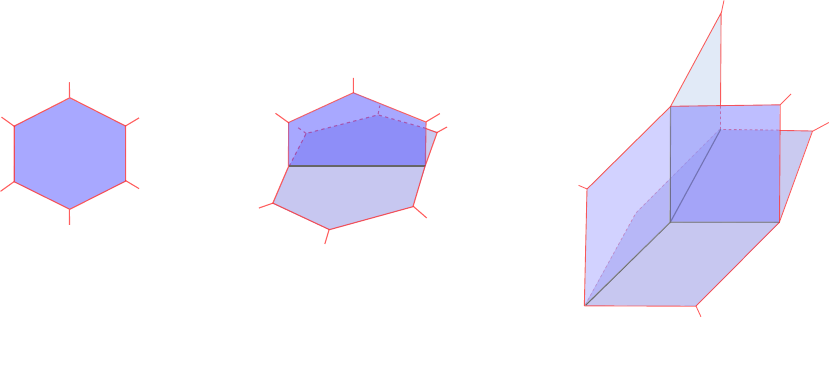
\caption{} \label{vertex_quintic}
\end{center}
\end{figure}
Type $(a)$ are contained in the interior of each $2$-face of $\Delta$ and there are $60$ of these ($6$ in each face). Type $(b)$ are defined along edges of $\Delta$ and there are $40$ of these ($4$ along each edge). Type (c) are defined around vertices of $\Delta$ and there are $5$ of these. In Example 4.10 of \cite{onHmstoric} it is shown how to construct Lagrangian spheres over type (a) components. It should be possible, combining the methods of this article with a detailed analysis of the interaction of $\mathcal L_0$ with the singular fibres, to construct smooth Lagrangian submanifolds (spheres?) over components of type $(b)$ and $(c)$. Similarly we should be able to construct Lagrangian submanifolds over tropical curves with boundary on $D$ using the constructions in \cite{mikh_trop_to_lag} and \cite{lag_pants}, together with a similar analysis of interactions with the singular fibres. For an explicit construction of Lagrangian lifts of tropical curves in the mirror of the quintic, using toric degenerations, see also \cite{mak_ruddat_tropLag}. 
\end{ex}


\vspace{1cm}
\begin{flushleft}
Diego MATESSI \\
Dipartimento di Matematica \\
Universit\`a degli Studi di Milano \\
Via Saldini 50 \\
I-20133 Milano, Italy \\
E-mail address: \email{diego.matessi@unimi.it}
\end{flushleft}

\end{document}

%% file: co_am.pdf_tex

\begingroup
  \makeatletter
  \providecommand\color[2][]{%
    \errmessage{(Inkscape) Color is used for the text in Inkscape, but the package 'color.sty' is not loaded}
    \renewcommand\color[2][]{}%
  }
  \providecommand\transparent[1]{%
    \errmessage{(Inkscape) Transparency is used (non-zero) for the text in Inkscape, but the package 'transparent.sty' is not loaded}
    \renewcommand\transparent[1]{}%
  }
  \providecommand\rotatebox[2]{#2}
  \ifx\svgwidth\undefined
    \setlength{\unitlength}{250.24941406pt}
  \else
    \setlength{\unitlength}{\svgwidth}
  \fi
  \global\let\svgwidth\undefined
  \makeatother
  \begin{picture}(1,0.39085313)%
    \put(0,0){\includegraphics[width=\unitlength]{co_am.pdf}}%
    \put(0.11474143,0.1401978){\color[rgb]{0,0,0}\makebox(0,0)[lb]{\smash{$C^-$}}}%
    \put(0.23767868,0.25220729){\color[rgb]{0,0,0}\makebox(0,0)[lb]{\smash{$C^+$}}}%
    \put(0.83278221,0.26102363){\color[rgb]{0,0,0}\makebox(0,0)[lb]{\smash{$C^+$}}}%
    \put(0.73909128,0.11131129){\color[rgb]{0,0,0}\makebox(0,0)[lb]{\smash{$C^-$}}}%
  \end{picture}%
\endgroup

%% file: amoeba_projections.pdf_tex

\begingroup
  \makeatletter
  \providecommand\color[2][]{%
    \errmessage{(Inkscape) Color is used for the text in Inkscape, but the package 'color.sty' is not loaded}
    \renewcommand\color[2][]{}%
  }
  \providecommand\transparent[1]{%
    \errmessage{(Inkscape) Transparency is used (non-zero) for the text in Inkscape, but the package 'transparent.sty' is not loaded}
    \renewcommand\transparent[1]{}%
  }
  \providecommand\rotatebox[2]{#2}
  \ifx\svgwidth\undefined
    \setlength{\unitlength}{178.62677002pt}
  \else
    \setlength{\unitlength}{\svgwidth}
  \fi
  \global\let\svgwidth\undefined
  \makeatother
  \begin{picture}(1,0.77909377)%
    \put(0,0){\includegraphics[width=\unitlength]{amoeba_projections.pdf}}%
    \put(0.62992586,0.27887402){\color[rgb]{0,0,0}\makebox(0,0)[lb]{\smash{$x$}}}%
    \put(0.41322257,0.27719795){\color[rgb]{0,0,0}\makebox(0,0)[lb]{\smash{$\bar R$}}}%
    \put(0.27007825,0.27840385){\color[rgb]{0,0,0}\makebox(0,0)[lb]{\smash{$x$}}}%
  \end{picture}%
\endgroup

%% file: amoeba_level_sets.pdf_tex

\begingroup
  \makeatletter
  \providecommand\color[2][]{%
    \errmessage{(Inkscape) Color is used for the text in Inkscape, but the package 'color.sty' is not loaded}
    \renewcommand\color[2][]{}%
  }
  \providecommand\transparent[1]{%
    \errmessage{(Inkscape) Transparency is used (non-zero) for the text in Inkscape, but the package 'transparent.sty' is not loaded}
    \renewcommand\transparent[1]{}%
  }
  \providecommand\rotatebox[2]{#2}
  \ifx\svgwidth\undefined
    \setlength{\unitlength}{283.13234863pt}
  \else
    \setlength{\unitlength}{\svgwidth}
  \fi
  \global\let\svgwidth\undefined
  \makeatother
  \begin{picture}(1,0.36101738)%
    \put(0,0){\includegraphics[width=\unitlength]{amoeba_level_sets.pdf}}%
    \put(0.45733283,0.22699696){\color[rgb]{0,0,0}\makebox(0,0)[lb]{\smash{$\h$}}}%
    \put(0.76640308,0.14406386){\color[rgb]{0,0,0}\makebox(0,0)[lb]{\smash{$\mathcal H_{r_1}$}}}%
  \end{picture}%
\endgroup

%% file: co_am_like.pdf_tex

\begingroup
  \makeatletter
  \providecommand\color[2][]{%
    \errmessage{(Inkscape) Color is used for the text in Inkscape, but the package 'color.sty' is not loaded}
    \renewcommand\color[2][]{}%
  }
  \providecommand\transparent[1]{%
    \errmessage{(Inkscape) Transparency is used (non-zero) for the text in Inkscape, but the package 'transparent.sty' is not loaded}
    \renewcommand\transparent[1]{}%
  }
  \providecommand\rotatebox[2]{#2}
  \ifx\svgwidth\undefined
    \setlength{\unitlength}{146.82945557pt}
  \else
    \setlength{\unitlength}{\svgwidth}
  \fi
  \global\let\svgwidth\undefined
  \makeatother
  \begin{picture}(1,0.74059526)%
    \put(0,0){\includegraphics[width=\unitlength]{co_am_like.pdf}}%
    \put(0.19218188,0.23849281){\color[rgb]{0,0,0}\makebox(0,0)[lb]{\smash{$\mathcal A^+_{J, \epsilon}$}}}%
  \end{picture}%
\endgroup

%% file: ref_points.pdf_tex

\begingroup
  \makeatletter
  \providecommand\color[2][]{%
    \errmessage{(Inkscape) Color is used for the text in Inkscape, but the package 'color.sty' is not loaded}
    \renewcommand\color[2][]{}%
  }
  \providecommand\transparent[1]{%
    \errmessage{(Inkscape) Transparency is used (non-zero) for the text in Inkscape, but the package 'transparent.sty' is not loaded}
    \renewcommand\transparent[1]{}%
  }
  \providecommand\rotatebox[2]{#2}
  \ifx\svgwidth\undefined
    \setlength{\unitlength}{194.81159668pt}
  \else
    \setlength{\unitlength}{\svgwidth}
  \fi
  \global\let\svgwidth\undefined
  \makeatother
  \begin{picture}(1,0.74233136)%
    \put(0,0){\includegraphics[width=\unitlength]{ref_points.pdf}}%
    \put(0.52378904,0.42891502){\color[rgb]{0,0,0}\makebox(0,0)[lb]{\smash{$b_{\check f}$}}}%
    \put(0.17766708,0.2015891){\color[rgb]{0,0,0}\makebox(0,0)[lb]{\smash{$r_{e, f}$}}}%
    \put(0.02807199,0.06079378){\color[rgb]{0,0,0}\makebox(0,0)[lb]{\smash{$\check e$}}}%
    \put(0.15692597,0.0377343){\color[rgb]{0,0,0}\makebox(0,0)[lb]{\smash{$r_{e, d}$}}}%
  \end{picture}%
\endgroup

%% file: innerpol_subdiv.pdf_tex

\begingroup
  \makeatletter
  \providecommand\color[2][]{%
    \errmessage{(Inkscape) Color is used for the text in Inkscape, but the package 'color.sty' is not loaded}
    \renewcommand\color[2][]{}%
  }
  \providecommand\transparent[1]{%
    \errmessage{(Inkscape) Transparency is used (non-zero) for the text in Inkscape, but the package 'transparent.sty' is not loaded}
    \renewcommand\transparent[1]{}%
  }
  \providecommand\rotatebox[2]{#2}
  \ifx\svgwidth\undefined
    \setlength{\unitlength}{152.44564209pt}
  \else
    \setlength{\unitlength}{\svgwidth}
  \fi
  \global\let\svgwidth\undefined
  \makeatother
  \begin{picture}(1,0.83241561)%
    \put(0,0){\includegraphics[width=\unitlength]{innerpol_subdiv.pdf}}%
    \put(0.32942246,0.14460165){\color[rgb]{0,0,0}\makebox(0,0)[lb]{\smash{$Y_{d,f}$}}}%
    \put(0.40822365,0.40231333){\color[rgb]{0,0,0}\makebox(0,0)[lb]{\smash{$\rho_f$}}}%
    \put(0.10450455,0.18261851){\color[rgb]{0,0,0}\makebox(0,0)[lb]{\smash{$Y_{e,f}$}}}%
    \put(-0.00348485,0.08920702){\color[rgb]{0,0,0}\makebox(0,0)[lb]{\smash{$\check e$}}}%
    \put(0.29280885,0.00617021){\color[rgb]{0,0,0}\makebox(0,0)[lb]{\smash{$\check d$}}}%
  \end{picture}%
\endgroup

%% file: neighborhoods_d.pdf_tex

\begingroup
  \makeatletter
  \providecommand\color[2][]{%
    \errmessage{(Inkscape) Color is used for the text in Inkscape, but the package 'color.sty' is not loaded}
    \renewcommand\color[2][]{}%
  }
  \providecommand\transparent[1]{%
    \errmessage{(Inkscape) Transparency is used (non-zero) for the text in Inkscape, but the package 'transparent.sty' is not loaded}
    \renewcommand\transparent[1]{}%
  }
  \providecommand\rotatebox[2]{#2}
  \ifx\svgwidth\undefined
    \setlength{\unitlength}{162.9487793pt}
  \else
    \setlength{\unitlength}{\svgwidth}
  \fi
  \global\let\svgwidth\undefined
  \makeatother
  \begin{picture}(1,0.85282053)%
    \put(0,0){\includegraphics[width=\unitlength]{neighborhoods_d.pdf}}%
    \put(0.7580983,0.34937253){\color[rgb]{0,0,0}\makebox(0,0)[lb]{\smash{$B_{d,f}$}}}%
    \put(0.37569177,0.49762296){\color[rgb]{0,0,0}\makebox(0,0)[lb]{\smash{$B_{d}$}}}%
  \end{picture}%
\endgroup

%% file: k_and_h.pdf_tex

\begingroup
  \makeatletter
  \providecommand\color[2][]{%
    \errmessage{(Inkscape) Color is used for the text in Inkscape, but the package 'color.sty' is not loaded}
    \renewcommand\color[2][]{}%
  }
  \providecommand\transparent[1]{%
    \errmessage{(Inkscape) Transparency is used (non-zero) for the text in Inkscape, but the package 'transparent.sty' is not loaded}
    \renewcommand\transparent[1]{}%
  }
  \providecommand\rotatebox[2]{#2}
  \ifx\svgwidth\undefined
    \setlength{\unitlength}{136.42803955pt}
  \else
    \setlength{\unitlength}{\svgwidth}
  \fi
  \global\let\svgwidth\undefined
  \makeatother
  \begin{picture}(1,1.08644265)%
    \put(0,0){\includegraphics[width=\unitlength]{k_and_h.pdf}}%
    \put(0.7329872,0.42513252){\color[rgb]{0,0,0}\makebox(0,0)[lb]{\smash{$K_{d}$}}}%
    \put(0.55707027,0.59476674){\color[rgb]{0,0,0}\makebox(0,0)[lb]{\smash{$H_d$}}}%
  \end{picture}%
\endgroup

%% file: vertex_quintic.pdf_tex

\begingroup
  \makeatletter
  \providecommand\color[2][]{%
    \errmessage{(Inkscape) Color is used for the text in Inkscape, but the package 'color.sty' is not loaded}
    \renewcommand\color[2][]{}%
  }
  \providecommand\transparent[1]{%
    \errmessage{(Inkscape) Transparency is used (non-zero) for the text in Inkscape, but the package 'transparent.sty' is not loaded}
    \renewcommand\transparent[1]{}%
  }
  \providecommand\rotatebox[2]{#2}
  \ifx\svgwidth\undefined
    \setlength{\unitlength}{238.81972656pt}
  \else
    \setlength{\unitlength}{\svgwidth}
  \fi
  \global\let\svgwidth\undefined
  \makeatother
  \begin{picture}(1,0.44584487)%
    \put(0,0){\includegraphics[width=\unitlength]{vertex_quintic.pdf}}%
    \put(0.05024711,0.00590801){\color[rgb]{0,0,0}\makebox(0,0)[lb]{\smash{$(a)$}}}%
    \put(0.38283509,0.00830066){\color[rgb]{0,0,0}\makebox(0,0)[lb]{\smash{$(b)$}}}%
    \put(0.75609928,0.00590796){\color[rgb]{0,0,0}\makebox(0,0)[lb]{\smash{$(c)$}}}%
  \end{picture}%
\endgroup

%% file: trop_hyp_lag.bbl
\begin{thebibliography}{10}

\bibitem{diri_branes}
Paul~S. Aspinwall, Tom Bridgeland, Alastair Craw, Michael~R. Douglas, Mark
  Gross, Anton Kapustin, Gregory~W. Moore, Graeme Segal, Bal{\'a}zs Szendroi,
  and P.~M.~H. Wilson.
\newblock {\em Dirichlet branes and mirror symmetry}, volume~4 of {\em Clay
  Mathematics Monographs}.
\newblock American Mathematical Society, Providence, RI, 2009.

\bibitem{CB-M}
R.~Castano-Bernard and D.~Matessi.
\newblock {Lagrangian 3-torus fibrations}.
\newblock {\em J. of Differential Geom.}, 81(3):483--573, 2009.
\newblock \href{http://arxiv.org/abs/math/0611139}{arXiv:math/0611139}.

\bibitem{TMS}
M.~Gross.
\newblock {Topological Mirror Symmetry}.
\newblock {\em Invent. Math.}, 144:75--137, 2001.

\bibitem{GHJ}
M.~Gross, D.~Huybrechts, and D.~Joyce.
\newblock {\em ``Calabi-Yau manifolds and related geometries'' Lecture notes at
  a summer school in Nordfjordeid, Norway, June 2001}.
\newblock Springer Verlag, 2003.

\bibitem{onHmstoric}
M.~Gross and D.~Matessi.
\newblock On homological mirror symmetry of toric {C}alabi-{Y}au three-folds.
\newblock \href{http://arxiv.org/abs/1503.03816}{arXiv:1503.03816}.

\bibitem{G-Siebert2003}
M.~Gross and B.~Siebert.
\newblock {Mirror Symmetry via Logarithmic degeneration data I}.
\newblock {\em J. Differential Geom.}, 72(2):169--338, 2006.
\newblock \href{http://arxiv.org/abs/math.AG/0309070}{math.AG/0309070}.

\bibitem{kalinin_wave_dyn}
Nikita Kalinin and Mikhail Shkolnikov.
\newblock Introduction to tropical series and wave dynamic on them.
\newblock \href{https://arxiv.org/abs/1706.03062}{arXiv:1706.03062}.

\bibitem{kerr_zha_phase_trop}
Gabriel Kerr and Ilia Zharkov.
\newblock Phase tropical hypersurfaces.
\newblock \href{https://arxiv.org/abs/1610.05290}{arXiv:1610.05290}.

\bibitem{mak_ruddat_tropLag}
Cheuk~Yu Mak and Helge Ruddat.
\newblock Tropically constructed {L}agrangians in mirror quintic 3-folds.
\newblock In preparation.

\bibitem{lag_pants}
Diego Matessi.
\newblock Lagrangian pairs of pants.
\newblock \href{https://arxiv.org/abs/1802.02993}{arXiv:1802.02993}.

\bibitem{mikh_trop_to_lag}
G.~Mikhalkin.
\newblock Examples of tropical-to-{L}agrangian correspondence.
\newblock \href{https://arxiv.org/abs/1802.06473}{arXiv:1802.06473}.

\bibitem{mikh_pants}
G.~Mikhalkin.
\newblock Decomposition into pairs-of-pants for complex hypersurfaces.
\newblock {\em Topology \textbf{43}}, pages 1035 -- 1065, 2004.

\bibitem{nisse_sottile_nonAch_coam}
Mounir Nisse and Frank Sottile.
\newblock Non-{A}rchimedean coamoebae.
\newblock In {\em Tropical and non-{A}rchimedean geometry}, volume 605 of {\em
  Contemp. Math.}, pages 73--91. Amer. Math. Soc., Providence, RI, 2013.
\newblock \href{https://arxiv.org/abs/1110.1033}{arXiv:1110.1033}.

\end{thebibliography}
